\documentclass{article}
\usepackage{dcolumn}
\usepackage{verbatim}       

\usepackage[pdftex]{graphicx}
\usepackage{amsthm}
\usepackage{amssymb}
\usepackage{bm}
\usepackage{amstext}
\usepackage{amsmath}
\usepackage{amsfonts}
\usepackage{mathrsfs}
\usepackage{cite}
\usepackage{hyperref}
\usepackage[toc]{appendix}

\newfont{\bb}{msbm10 at 12pt}

\newcommand{\bd}{\begin{definition}}                
\newcommand{\ed}{\end{definition}}                  
\newcommand{\bc}{\begin{corollary}}                 
\newcommand{\ec}{\end{corollary}}                   
\newcommand{\bl}{\begin{lemma}}                     
\newcommand{\el}{\end{lemma}}                       
\newcommand{\bp}{\begin{proposition}}            
\newcommand{\ep}{\end{proposition}}                
\newcommand{\bere}{\begin{remark}}                  
\newcommand{\ere}{\end{remark}}                     

\newcommand{\bt}{\begin{theorem}}
\newcommand{\et}{\end{theorem}}

\newcommand{\be}{\begin{equation}}
\newcommand{\ee}{\end{equation}}

\newcommand{\bit}{\begin{itemize}}
\newcommand{\eit}{\end{itemize}}
\newtheorem{theorem}{Theorem}[section]
\newtheorem{corollary}[theorem]{Corollary}

\newtheorem{lemma}[theorem]{Lemma}
\newtheorem{proposition}[theorem]{Proposition}
\theoremstyle{definition}
\newtheorem{definition}[theorem]{Definition}
\theoremstyle{remark}
\newtheorem{remark}[theorem]{Remark}
\newtheorem{example}[theorem]{Example}

\begin{document}
%

\title{Lorentzian metric spaces and GH-convergence: the unbounded case}

\author{A. Bykov,\footnote{Dipartimento di Matematica e Informatica ``U. Dini'', Universit\`a degli Studi di Firenze,  Via
S. Marta 3,  I-50139 Firenze, Italy. E-mail: aleksei.bykov@unifi.it} \ \ E. Minguzzi\footnote{Dipartimento di Matematica, Universit\`a degli Studi di Pisa,  Via
B. Pontecorvo 5,  I-56127 Pisa, Italy. E-mail:
ettore.minguzzi@unipi.it} \ \ and   \ S. Suhr\footnote{Fakult\"at f\"ur Mathematik, Ruhr-Universit\"at Bochum, Universit\"atsstr. 150, 44780 Bochum, Germany. E-mail: Stefan.Suhr@ruhr-uni-bochum.de}}

\date{}

\maketitle

\begin{abstract}
\noindent We introduce a notion of  Lorentzian metric space which drops the  boundedness condition from our previous work
and argue that the properties defining our spaces  are  minimal.
In fact, they are defined by three conditions given by (a) the reverse triangle inequality for chronologically related events, (b) Lorentzian distance continuity and  relative compactness of chronological diamonds, and (c) a distinguishing condition via the Lorentzian distance function. By adding a countably generating condition we confirm the validity of desirable properties for our spaces including the Polish  property.
The definition of (pre)length space given in our previous work on the bounded case is  generalised to this setting. We also define a notion of Gromov-Hausdorff convergence for Lorentzian metric spaces and prove that  (pre)length  spaces  are GH-stable. It is also shown that our (sequenced) Lorentzian metric spaces bring a natural quasi-uniformity (resp.\ quasi-metric).
 Finally, an explicit comparison with other recent constructions based on  our previous work on bounded Lorentzian metric spaces is presented.
\end{abstract}


\setcounter{secnumdepth}{2}
\setcounter{tocdepth}{2}
\tableofcontents

\section{Introduction}

The General Theory of Relativity is  formulated through the mathematics of smooth Lorentzian manifolds but, in order to prepare the ground for a  possible quantization of spacetime or gravity, it seems worthwhile to develop a mathematics capable of describing more general objects and possibly more abstract frameworks \cite{minguzzi17e}. This direction of investigation  is also motivated by the need of exploring some technical aspects of the standard theory, from the study of gravitational shock waves \cite{podolsky22}, to issues on spacetime inextendibility related to the problem of cosmic censorship \cite{galloway17b,grant18,galloway18b,sbierski20,minguzzi19b}.
Thus, it is important to find  a general concept of  spacetime that could retain at least the  geometrical notions that have proved of interest in the smooth theory.

Looking for a guide in this type of research, it is natural to consider  the example of metric geometry, which might be seen as a generalization of smooth Riemannian geometry \cite{burago01}.
A remarkable feature of metric geometry is its minimalism: the only structure introduced is a two-point function satisfying a few intuitively clear requirements. Other structures, such as the topology, are deduced from the metric. Often, this setting proves too general to establish any interesting result, so one usually restricts the attention to compact (or, say, boundedly compact) metric spaces. A virtue of the theory of compact metric spaces is the existence of a natural notion of distance between such spaces, namely the Gromov-Hausdorff distance and the associated concept of  Gromov-Hausdorff limit. From this perspective, the most natural properties of metric spaces are those  preserved by  Gromov-Hausdorff limits. Notable examples are provided by (pre-)length spaces and Alexandrov spaces (the latter generalize the Riemannian manifolds of bounded sectional curvature).
While it makes sense to speak of Gromov-Hausdorff distance  between compact metric spaces only,  the notion of Gromov-Hausdorff convergence makes sense also for sequences of boundedly compact spaces.

%

Synthetic Lorentzian geometry has a long history, in fact it is arguably as old as relativity itself. Before the discovery of general relativity many researchers approached special relativity with a mindset that took as reference the axiomatic development of Euclidean geometry \cite{robb14}. This same way of looking at relativity persisted over the years.

A more modern axiomatic framework was introduced in Kronheimer and Penrose's 1967 work \cite{kronheimer67}, which focused on the causal (and chronological) structure of spacetime. Later, the idea of characterizing spacetime using both causal structure and a spacetime measure was proposed \cite{myrheim78}. Further refinements of these concepts, particularly in the context of finite sets,  led to one of the  approaches to quantum gravity: Causal Set Theory \cite{surya19}.

Interestingly, the idea of using Lorentzian distance, i.e. (maximal) proper time, as a Lorentzian analog of the distance function in metric geometry appeared in Busemann's work \cite{busemann67}---published the same year as Kronheimer and Penrose's paper. However, Busemann's Lorentzian distance-based approach received little attention until the relatively recent work of Kunzinger and S\"amann \cite{kunzinger18}, who introduced the concept of Lorentzian length spaces in analogy to the theory of length spaces in metric geometry. Their work, along with subsequent studies, introduced many key notions and established powerful results for this class of spaces, mirroring the success of Alexandrov's synthetic geometry of length spaces, as well as some results of smooth Lorentzian geometry.

A  drawback of these approaches, from our perspective, is the technical complexity of their foundational definitions, which introduce auxiliary structures—such as a predefined topology (in Busemann's case) or an auxiliary metric (in Kunzinger and S\"amann's framework)-lacking clear Lorentzian geometric significance. As a result, these theories cannot be regarded as fully minimalist Lorentzian counterparts to metric geometry. Moreover, since Lorentzian geometry is not only a mathematical field but also a potential foundation for future physical theories, it is preferable to avoid mathematical objects without clear operational meaning. We argue, therefore, that despite these advances, the determination of the correct definition  for the Lorentzian analogue of metric space remained an open problem.

 Coming to our previous work \cite{minguzzi22}, a notion of  \emph{bounded} Lorentzian metric space was there introduced  and some properties of this class of objects were studied. The main virtue of the approach  in \cite{minguzzi22} is that in it, similarly to the ordinary metric space theory,  the only structure introduced is a two-point function $d$ satisfying certain conditions. Later,
the recent work by Braun and McCann \cite{braun23b} also motivated by optimal transport in an abstract Lorentzian length space setting appeared. Their framework was more closely connected to our previous work \cite{minguzzi22}, rather than Kunzinger and S\"amann's \cite{kunzinger18}, so we shall discuss that contribution more closely in a following section.

The studies of Gromov-Hausdorff convergence in the Lorentzian case also have some history.
The first ideas that should be mentioned is that  by Noldus \cite{noldus04,noldus04b} on Gromov-Hausdorff distance and convergence specifically for compact globally hyperbolic manifolds with spacelike boundaries    and Noldus and Bombelli \cite{bombelli04}, where the analysis was extended to the non-manifold case. Generally considered as inspiring, they contained useful ideas such as Noldus' strong metric. This concept  was later shown to be equivalent to the distinction metric \cite{minguzzi22}, a tool proved useful in our previous work \cite{minguzzi22} (though not essential, as we shall confirm in this paper).

More recently, M\"uller revisited some of Noldus' constructions, including his strong metric, within a broad  analysis of the Gromov-Hausdorff convergence problem \cite{muller19,muller22,muller22c,muller24}. In the first version of \cite{muller22c}, M\"uller explored the Gromov-Hausdorff metric applied to compact Lorentzian prelength spaces and partially ordered measure spaces, (independently) obtaining results which in the literature are the most similar to our own, at least for what concerns Gromov-Hausdorff convergence and precompactness. Later he refined the work including more material.

Within the framework of Kunzinger and S\"amann \cite{kunzinger18}, Cavaletti and Mondino \cite{cavalletti20,cavalletti22} introduced a concept of measured Gromov-Hausdorff convergence for measured Lorentzian geodesic spaces. They also proved a stability result concerning the timelike curva\-tu\-re--di\-men\-sion condition, see also Braun \cite{braun22}.

A distinct approach was  taken by Sormani and collaborators on the basis of Sormani and Vega's \cite{sormani16} null distance, see for instance later contributions with Sakovich \cite{sakovich22} including the recent work \cite{sakovich24} on definite distances on spaces of Lorentzian spacetimes,  including their new timed-Hausdorff distance.

 This null distance is well-suited to express convergence in cosmological spacetimes, a perspective expanded by Allen and Burtscher \cite{allen19,allen23} (see also \cite{burtscher22} for causality-related results). The idea was taken up by other groups, see e.g.\ the work by Kunzinger and Steinbauer \cite{kunzinger22}. Ultimately, these approaches require a time function to build a  metric and from there, they typically employ ideas from the classical positive signature version of Gromov-Hausdorff convergence.


 Returning to our previous paper \cite{minguzzi22}, the notion of bounded Lorentzian metric spaces allowed for a natural generalization of  Gromov-Hausdorff convergence to the Lorentzian case.
The treatment  was still limited to spaces with bounded timelike diameter, the general case being left for subsequent research.\footnote{For this introduction purposes, the reader can think of the bounded Lorentzian metric spaces as a form of compact  Lorentzian metric spaces though they are more general than that.}



The goal of the present paper is to remove this restriction, while preserving the minimalism of the approach as much as possible.  As we shall see, we shall also clarify the physical meaning of some of the assumptions introduced previously. In metric geometry the theory of unbounded metric spaces builds upon that of compact metric spaces. Similarly here, the theory for unbounded Lorentzian metric spaces makes use of compactness assumptions that will ultimately be understood as the well-known condition of global hyperbolicity.

In our previous work we emphasized that our spaces were meant to describe  rough globally hyperbolic bounded spaces. Having removed the boundedness condition it is not entirely surprising, but still reassuring, that what we get are rough globally hyperbolic spaces. What is non-trivial, and even remarkable, is the fact that, up to some differences at the chronological boundary (introduced mostly for compatibility with the bounded case), our new unbounded spaces are essentially characterized by three  desirable properties  roughly identified as follows: the reverse triangle inequality, the precompactness of chronological diamonds (and continuity of $d$), and a distinction property. 
Still, the so determined family of spaces  is sufficiently large to comprise within it  causets and  smooth globally hyperbolic spacetimes.


The paper is organized as follows. We end this section explaining some notations and conventions used throughout the paper. In the second section we introduce our notion of  Lorentzian metric space and explain the motivation behind it. We also discuss  the inclusion of the bounded Lorentzian metric spaces into this class. The third section is devoted to the topological properties of Lorentzian metric spaces. In the fourth section we introduce the (maximal) causal relation and prove the  existence of  time functions (Lemma \ref{lem:time-func-seq}).  We show that the requirements we impose on the Lorentzian metric space  are essentially equivalent to global hyperbolicity (as there is a time function  and the causally convex hull of compact sets is compact, cf.\ Thm.\ \ref{thm:cnqg}).

 The fifth section introduces the Lorentzian (pre)length spaces and the limit curve theorem (Thm.\ \ref{nner}).  We also explain, for comparison with other works in the literature, how the theory would be expressed using a time separation function $l$ with the allowed value $-\infty$.
   For instance, we provide a comparison of our approach with that used by Braun and McCann in \cite{braun23b}, which, though not concerned with problems of GH-convergence, still develops a non-compact framework based on our previous work on bounded Lorentzian metric spaces \cite{minguzzi22}.

The sixth section is devoted to the notion of Gromov-Hausdorff convergence. Here one needs to take into account that, unlike the metric case in which every point is within a ball of any other point, in the Lorentzian setting one needs several chronological diamonds to cover the whole spacetime and one cannot rely on just  one point (the center) and one number (the radius). This naturally leads to the notion of {\em sequenced Lorentzian metric space}, which is the analog of {\em pointed metric space}. The sequence will  be rather sparse in some cases, depending on the space under study and the structure of its boundary. Gromov-Hausdorff convergence is then defined using these sequenced spaces, the limit being dependent on the sequence chosen for each space. Examples show that changing sequences changes the limit space, much as it happens on metric space theory for GH-convergence of pointed spaces (Example \ref{exmp:bad-renum}). We shall also study to what extent sequences can be rearranged without spoiling convergence (Thm.\ \ref{thm:change-gensec-X}), and how GH-convergence in the sense of bounded Lorentzian metric spaces (developend in our previous work) can be recast as GH-convergence of their chronological interiors (Prop.\ \ref{nner}). Our main result in the section is the stability of the (pre)length condition under GH-limits (Thm.\ \ref{thm:prelength-prserved}).

In the seventh section we start obtaining the generalization of the Kuratowski embedding, which was already introduced in our previous work (Thm.\ \ref{cmqnn}).  In the subsequent subsections we prove that every Lorentzian metric space induces a closed ordered space which is quasi-uniformizable. Remarkably, the quasi-uniformity is canonical and easily constructed from the Lorentzian distance (Thm.\ \ref{thm:qu-lms}). The induced topology and order are those that we canonically associated to a Lorentzian metric space. This structure appears to be new already for the smooth case and allows one to easily prove the limit curve theorem and an equivalence between pointwise and uniform convergence of isocausal curves (Lemma \ref{lem:curves-from-dense}, Thm.\ \ref{thm:d-uni}, Cor.\ \ref{cor:isocaus-uni}).

We also show that in the sequenced case every Lorentzian metric space has a canonically associated quasi-metric (Thm.\ \ref{vnqjp}). The quasi-metric brings with it  the information on both the topology (upon symmetrization) and order (the zero locus). These facts not only show that the definition of Lorentzian metric space has a certain elegance and simplicity, they also pave the way for further developments.

\subsection{Notation and conventions}\label{ssec:notation}
As this paper is a continuation of \cite{minguzzi22}, we adopt most of its notation and terminology. In particular, in most of the paper we work with a set $X$ endowed with a function $d: X\times X \to [0,+\infty)$. We say that $d$ satisfies the reverse triangle inequality if $d(x,z)\geq d(x,y)+d(y,z)$ whenever $d(x,y)$, $d(y,z)>0$. Only in Subsection \ref{ssec:braun} we sometimes refer to this property as the \emph{restricted} reverse triangle inequality (in order to distinguish it from the \emph{extended} inequality which holds for any $x,y,z\in X$ when $d$ has a different codomain including $-\infty$). For each $x\in X$ we introduce the functions $d_x,d^x: X\to [0,+\infty)$ as follows
\[
d_x(y):=d^y(x):=d(x,y),\,\forall x,y\in X.
\]

Given such an $X$ and $d$ we define the chronological  relation $\ll$ through
\[
x\ll y \Leftrightarrow d(x,y)>0.
\]
As shown in \cite{minguzzi22}, the reverse triangle inequality implies that $\ll$ is transitive and antisymmetric. It is convenient to have a symbol for the relation $\ll$ understood as a subset of $X\times X$, so we shall write
\[
I:=\{(x,y)\in X\times X| x\ll y\}.
\]
For $x\in X$ we set
\[
I^+(x):=\{y\in X| x\ll y\},\quad I^-(x):=\{y\in X| y\ll x\},
\]
while for any subset $A\subset X$ we set $I^\pm(A)=\cup_{p\in A} I^\pm(p)$ for the chronological future/past and
\[
I(A):=\cup_{p,q\in A} I^+(p) \cap I^-(q)=I^+(A)\cap I^{-}(A),
\]
for the chronological convex hull.
For $A$ finite, we drop the curly brackets,
\[
I(p^1,\ldots, p^k):=\cup_{\{1\le i,j\le k\}} I^+(p^i) \cap I^-(p^j).
\]

In Sections \ref{sec:causal} and \ref{sec:GH} we use analogues of $I^{\pm}(x)$ and $I(A)$ with $I$ replaced by other relations.

We also define the closed relation
\[
I_{\epsilon}=\{(x,y)\in X\times X| d(x,y)\geq \epsilon\}.
\]
Note that if $d$ satisfies the reverse triangle inequality, then
\[
(x,y)\in I_{\epsilon}, \, (y,z)\in I_{\epsilon'} \Rightarrow (x,z)\in I_{\epsilon+\epsilon'},\,\forall x,y,z\in X,\,\forall \epsilon,\epsilon'>0.
\]
Taking into account the obvious property $I_{\epsilon}(x,y)\subset I_{\epsilon'}(x,y)$ for every $x,y\in X$ and $\epsilon>\epsilon'>0$, we conclude that $I_{\epsilon}$ is a transitive relation.

We also adopt the following conventions. The subset symbol $\subset$ is reflexive. The natural numbers $\mathbb{N}$ do not contain 0. A {\em preorder} is a reflexive and transitive relation. An {\em order} is an antisymmetric preorder.
\section{Definition and interpretation}
In this section we give the formal definition of the object we deal with, discuss its meaning and the alternative versions.
We start from the following natural generalization of Definition 1.1 in \cite{minguzzi22}.
\begin{definition}\label{cg-lms}
    A \emph{Lorentzian metric space} $(X,d)$ is a set $X$ endowed with a function $d: X\times X\to [0,\infty)$,  called {\em Lorentzian distance}, with the following properties:
    \begin{enumerate}
        \item[(i)] $d$ satisfies the reverse triangle inequality;
        \item[(ii)] There is a topology $T$ on $X$ such that $d$ is continuous in the product topology and, for every $x,y\in X$ and every $\epsilon>0$, the set $I_{\epsilon}\cap \left(\overline{I(x,y)} \times \overline{I(x,y)}\right)$ is compact;
        \item[(iii)] $d$ distinguishes points.
        \label{countGen}
    \end{enumerate}
\end{definition}
From now on by property (i), (ii) or (iii), we always mean the corresponding property of Definition \ref{cg-lms}.

The property (ii) may look a bit artificial due to the presence of the sets $I_{\epsilon}$. To clarify its motivation, let us first consider the case of bounded Lorentzian metric spaces, where (ii) is replaced by continuity of $d$ and  compactness of $I_\epsilon$
\begin{example}\label{exmp:blms}
    Any bounded Lorentzian metric space is a Lorentzian metric space, since $(\overline{I(x,y)}\times \overline{I(x,y)}) \cap I_{\epsilon}\subset I_{\epsilon}$ is compact as a closed subset of a compact set.
\end{example}
The goal of the remainder of this section is to understand the meaning of property (ii) beyond the bounded case. We provide the following sufficient condition.
\begin{lemma}\label{lem:chron-diamonds-imply-ii}
    Let $(X,d)$ satisfy property (i), and let $T$ be some topology of $X$ such that $d$ is continuous. Suppose that for every $x,y\in X$ the chronological diamond $I(x,y)$ is relatively compact. Then for every $x,y\in X$ and $\epsilon>0$ the set
    \[(\overline{I(x,y)}\times \overline{I(x,y)})\cap I_{\epsilon}\]
    is compact. If, in addition to that, $(X,d)$ satisfies property (iii), then it is a Lorentzian metric space.
\end{lemma}


\begin{proof}
    If $\overline{I(x,y)}$ is compact, then the same is true for $\overline{I(x,y)}\times \overline{I(x,y)}$. Therefore, $(\overline{I(x,y)}\times \overline{I(x,y)})\cap I_{\epsilon}$ is a closed subset of a compact set, so it is compact. The last statement follows directly from Definition \ref{cg-lms}.
\end{proof}
The lemma above leads to another important class of examples of Lorentzian metric spaces.
\begin{proposition}\label{prop:globhyp-is-LMS}
    Let $(M,g)$ be a globally hyperbolic (smooth) spacetime  and let $d: M\times M \to [0,+\infty)$ be the associated Lorentzian distance. Then  $(M,d)$ is a Lorentzian metric space. In particular, the manifold topology of $M$ satisfies property (ii).
\end{proposition}
\begin{proof}
    This fact essentially follows from the proof of \cite[Theorem 2.5]{minguzzi22}, so we only sketch the main ideas.
    It is known that the Lorentzian distance function on globally hyperbolic spacetimes is finite, continuous in the manifold topology and satisfies the reverse triangle inequality. In globally hyperbolic spacetimes chronological diamonds are relatively compact and the points distinguished by the chronological order $\ll$ (and thus by the distance function), so by Lemma \ref{lem:chron-diamonds-imply-ii} $(X,d)$ is a Lorentzian metric space.
\end{proof}
The converse of Lemma \ref{lem:chron-diamonds-imply-ii} does not hold in general. There are bounded Lorentzian metric spaces with non-relatively compact chronological diamonds (e.g.\ sets of the form $J^+(p)\cap J^-(q)$ in Minkowski spacetime with the rim $E^+(p)\cap E^-(q)$ removed). Still, as we shall see in a moment, this can happen only if the space has a non-empty chronological boundary.


Let us recall the following notions introduced in \cite{minguzzi22}:

\begin{definition}\label{def:chron-bound}
     Let $(X,d)$ be a set endowed with a function $d: X\times X \rightarrow [0,+\infty)$. Its \emph{future boundary} is the set
    \[
    X^{+}=\{x\in X| d(x,y)=0,\, \forall y\in X\},
    \]
    and  its \emph{past boundary} is the set
    \[
    X^{-}=\{x\in X| d(y,x)=0,\, \forall y\in X\}.
    \]
    We refer to the union of these sets as the \emph{chronological boundary} of $X$.
    We say that $X$ \emph{has no chronological boundary} if $X^{+}=X^{-}=\varnothing$. Finally, we note that $I(X)=X\setminus (X^+\cup X^-)$. It is convenient to refer to the set $I(X)$ as the \emph{chronological interior} of $X$.
    \end{definition}
    The following remarks are in order.
    \begin{remark}
    Since $d$ distinguishes points, $X^{+}\cap X^{-}$ can contain at most one point (as in the bounded case, it might be called the {\em  spacelike boundary}).
    \end{remark}
    \begin{remark}\label{rmk:IntIsOpen}
         Assume that $X$ is a topological space, and $d$ is continuous. Then $X^+$ and $X^-$ are closed sets, and $I(X)$ is open.
    \end{remark}

    Now we are ready to give an alternative formulation of property (ii) for a large class of spaces.
    \begin{proposition}\label{prop:chdiamonds-compact}
    Let $X$ be a set and $d:X\times X\to [0,+\infty)$ a function satisfying the reverse triangle inequality. Suppose that $X^{+}=\varnothing$ or $X^{-}=\varnothing$. Let $T$ be a topology of $X$ such that $d$ is continuous in the product topology. Then the following statements are equivalent:
    \begin{enumerate}
        \item For every $p,q\in X$ and $\epsilon>0$ the set $\left(\overline{I(p,q)}\times \overline{I(p,q)}\right)\cap I_{\epsilon}$ is compact;
        \item All chronological diamonds are relatively compact.
    \end{enumerate}
    \end{proposition}
    \begin{proof}
The second statement implies the first one by Lemma \ref{lem:chron-diamonds-imply-ii}, so we concentrate on proving the converse.

        Assume that the first statement holds and consider $p,q\in X$. We intend to show that $\overline{I(p,q)}$ is compact in $T$. Note that if neither $p\ll q$ nor $q\ll p$, then $I(p,q)=\varnothing$ and there is nothing to prove. So, without  loss of generality we can assume $p\ll q$.

        Let us assume for definiteness that $X^{+}$ is empty. Then we can find $q',q'' \in X$ such that $q\ll q'\ll q''$. Define $\epsilon=d(q,q')>0$. Let $\pi_1$ and $\pi_2$ be the canonical projections $X\times X\to X$ on the first and the second factors respectively. We claim that
        \begin{equation}\label{eqn:prop:chdiamonds-compact-Adef}
            I(p,q)\subset A:=\pi_1\left(\left(\overline{I(p,q'')}\times \overline{I(p,q'')}\right)\cap I_{\epsilon}\right).
        \end{equation}

        For that take $x\in I(p,q)$. We have
        \[
        d(x,q')\geq d(x,q)+d(q,q')> \epsilon,
        \]
        so $(x,q')\in I_{\epsilon}$. Taking into account that $x,q'\in I(p,q'')$ we get that $x\in A$. We also note that $A$ is compact, because it is the image under a continuous map of a compact set in $X\times X$. For the same reason the set
        \[
         B=\pi_2\left(\left(\overline{I(p,q'')}\times \overline{I(p,q'')}\right)\cap I_{\epsilon}\right)
        \]
        is compact. Since
        \[
        \left(\overline{I(p,q'')}\times \overline{I(p,q'')}\right)\cap I_{\epsilon}\subset X\times B,
        \]
         $\pi_1$ in (\ref{eqn:prop:chdiamonds-compact-Adef}) can be considered as the projection $X\times B\to X$, so compactness of $B$ implies that $\pi_1$ is  proper (in the terminology of \cite{bourbaki66}), and thus a closed map \cite[10.2, Cor.\ 5 and 10.1, Prop.\ 1]{bourbaki66}. So, $A$ is closed, $\overline{I(p,q)}\subset A$  and thus $\overline{I(p,q)}$ is compact.

        Similarly, if $X^{-}=\varnothing$, there are $p',p''\in X$ such that $p''\ll p'\ll p$. Setting $\epsilon=d(p',p)$ and reproducing the argument above, we get that
        \[
            I(p,q)\subset A:=\pi_2\left(\left(\overline{I(p'',q)}\times \overline{I(p'',q)}\right)\cap I_{\epsilon}\right).
        \]
        Arguing along the same lines, one can show that $A$ is a closed compact set, so any of its subsets is relatively compact.
    \end{proof}

    \begin{remark}
        In the next section (see Lemma \ref{lem:Hsdf}) we shall see that the topology of a Lorentzian metric spaces is always Hausdorff, so any compact set is automatically closed and the implication in the proof of Proposition \ref{prop:chdiamonds-compact} can be simplified. The proof of Lemma \ref{lem:Hsdf} relies on property (iii). Although in this paper we do not consider spaces with indistinguishable points, it is still worth avoiding usage of this property if not needed. In particular, by examining the proofs one may verify that Corollary \ref{crl:prop-iip} and Theorem \ref{thm:propii-via-chdiamonds}, as well as Theorem \ref{thm:lms-via-emeralds} can be straightforwardly generalised to a more general setting without imposing property (iii).
    \end{remark}

    For completeness, we provide a generalisation of Proposition \ref{prop:chdiamonds-compact} that applies also to the cases when neither of the chronological boundaries is empty.
    \begin{proposition}
        Let $(X,d)$ be a Lorentzian metric space. Suppose that $A\subset I(I(X))$ is a finite set. Then $\overline{I(A)}$ is compact in any topology, satisfying the requirements of property (ii).
    \end{proposition}

    \begin{proof}
        It is enough to show that if $p,q\in I(I(X))$, then $\overline{I(p,q)}$ is relatively compact. Since $q\in I(I(X))$, there are $q',q''\in X$ such that $q\ll q'\ll q''$. The rest of the proof reproduces the one of Proposition \ref{prop:chdiamonds-compact}.
    \end{proof}

    Let us return to the case of at least one chronological boundary being empty. Combining Lemma \ref{lem:chron-diamonds-imply-ii} with Proposition \ref{prop:chdiamonds-compact} one can conclude that for such spaces relative compactness of chronological diamonds is equivalent to property (ii). For future use we state this a bit more generally.

    \begin{corollary}\label{crl:prop-iip}
    Let $(X,d)$ satisfy the property (i) of Definition \ref{cg-lms}. Assume that $X^{+}=\varnothing$ or $X^{-}=\varnothing$. Let $\mathscr{F}$ be a family of subsets of $X$ such that:
    \begin{enumerate}
        \item For each $F\in \mathscr{F}$ there is a finite set $A \in X$ such that
        $F\subset I(A)$;
        \item For every $p,q\in X$ there is $F\in\mathscr{F}$ such that $I(p,q)\subset F$.
    \end{enumerate}
    Then the condition (ii) of Definition \ref{cg-lms} is equivalent to the following:
    \begin{enumerate}
        \item[(ii')]  There is a topology $T$ of $X$ such that $d$ is continuous w.r.t.\ the product topology of $T$ and all members of $\mathscr{F}$ are relatively compact.
    \end{enumerate}

\end{corollary}

The trivial example is the family $\mathscr{F}$ consisting of all chronological diamonds.

\begin{proof}
    In every topological space a subset of a relatively compact set is relatively compact. If the elements of $\mathscr{F}$ are relatively compact then the chronological diamonds are relatively compact. Conversely, if the chronological diamonds are relatively compact then all sets of the form $I(A)$, where $A\subset X$ is finite, are relatively compact (because $I(A)$ is the union of  finitely many chronological diamonds), which implies that all sets of $\mathscr{F}$  are relatively compact.

    Thus (ii') is equivalent to the chronological diamonds being relatively compact which, by  Prop.\ \ref{prop:chdiamonds-compact}, is equivalent to (ii).
\end{proof}

    More useful families can be defined with the help of the following concept which will play an important role throughout the paper.
    \begin{definition}\label{def:gen-set}
        Let $X$ be a set and $d: X\times X \to [0,+\infty)$  a function satisfying the reverse triangle inequality. We say that $\mathscr{G}\subset X$ \emph{generates} $X$ (or is a \emph{generating set} for $X$) if $X=I(\mathscr{G})$.
    \end{definition}
    In other words, $\mathscr{G}$ is a generating set if for any $x\in X$ there are $p,q\in\mathscr{G}$ such that $p\ll x\ll q$.
    Note that no assumption on the cardinality of $\mathscr{G}$ was made, although it is natural to take it as small as possible. We also note that the existence of a generating set should not be taken for granted.
    \begin{remark}\label{rmk:genset=no-bnd}
       The existence of a generating set implies  emptiness  of the chronological boundary. Indeed, if $(X,d)$ admits a generating set $\mathscr{G}$, then $X=I(\mathscr{G})\subset I(X)\subset X$, so $X=I(X)$, i.e. $X^+\cup X^-=\varnothing$. Conversely, if the chronological boundary of $X$ is empty, then $X=I(X)$, and $X$ is itself is a generating set.
    \end{remark}

    Then we can summarize the main results of this section into the following theorem.
    \begin{theorem}\label{thm:propii-via-chdiamonds}
        Let $X$ be a set endowed with a function $d: X\times X\rightarrow [0,+\infty)$ satisfying properties (i) and (iii) of Definition \ref{cg-lms}. Suppose that $X^+=\varnothing$ or $X^{-}=\varnothing$. Then the following statements are equivalent:
        \begin{enumerate}
            \item $(X,d)$ is a Lorentzian metric space;
            \item There is a topology on $X$ such that $d$ is continuous and for each $p,q\in X$ the set $\overline{I(p,q)}$ is compact.
        \end{enumerate}
        Moreover, if $X$ admits a generating set $\mathscr{G}\subset X$, then either of the statements above is equivalent to the following:
                \begin{enumerate}
            \item[3.] There is a topology on $X$ such that $d$ is continuous and for each $p,q\in \mathscr{G}$ the set $\overline{I(p,q)}$ is compact.
        \end{enumerate}
    \end{theorem}

    \begin{proof}
    The equivalence of the first two statements follows from Prop.\  \ref{prop:chdiamonds-compact}.


    If $\mathscr{G}\subset X$ is a generating set, then for any $p,q\in X$ such that $p\ll q$ there are $p',q'\in\mathscr{G}$ such that $p'\ll p\ll q\ll q'$, thus $I(p,q)\subset I(p',q')$. So,
     \[
     \mathscr{F}=\{I(p,q)|p,q\in \mathscr{G}\}
     \]
     satisfies requirements of Cor.\ \ref{crl:prop-iip}, therefore the first and the last statements are equivalent.
    \end{proof}

     We stress that the cases considered above, namely the case of bounded Lorentzian metric spaces and the case of metric spaces without at least one of the chronological boundaries are the most interesting for possible physical applications. They cover, for example, the local problems, when one is interested in a compact sub-region of the spacetime, the case of infinitely extendable spacetime without any boundary, the Big Bang scenarios and formation of a black hole where the singularity could, in principle, be included in the space. Since the bounded case was thoroughly considered in \cite{minguzzi22}, it is instructive to concentrate in this paper on the case of spacetimes with one or both chronological boundaries being empty.

     Before finishing this preliminary general considerations, let us discuss the issue of removing the chronological boundaries. We recall that in the bounded case, sometimes it was more convenient to work with spaces with a spacelike boundary, and sometimes with spaces without the spacelike boundary. Fortunately, it was possible to add or remove the spacelike boundary without affecting most of the space properties. This is not always so for the chronological boundary.

    \begin{remark}\label{InterirorMayHaveBoundary}
    In general, it is not true that $(I(X) ,d\vert_{I(X) \times I(X) })$ is a Lorentzian metric space without a boundary. Removing the boundary may spoil the distinguishing property (one can amend this problem by taking the distance quotient\cite{minguzzi22}) and also the resulting set can still have non-empty chronological boundary, because in general $I(I(X))\neq I(X)$. Examples are provided by causets (i.e.\ finite bounded Lorentzian metric spaces, see Subsection 2.2 of \cite{minguzzi22}).
    \end{remark}
     In some important practical cases the chronological boundary can be removed.

    \begin{lemma}\label{lem:InteriorHasNoBoundary}
    Let $(X,d)$ be a Lorentzian metric space and assume that $I(X)$ is dense in $X$ in some topology satisfying the requirements of Definition \ref{cg-lms}. Then $I(X)$ is a Lorentzian metric space with no chronological boundary.
    \end{lemma}
    That is, the points in $I(X)$ are still distinguished by $d$ and so there is no need to pass to the distance quotient. Also, the proof shows that if $T$ is a topology that satisfies (ii) of Definition \ref{cg-lms} for $X$ then, the induced topology satisfies the same property for $I(X)$.

    \begin{proof}

    Clearly $I(X)$ endowed with  $d':=d\vert_{I(X) \times I(X) }$ satisfies (i). Moreover, let us show that $d'$ distinguishes points and that the chronological boundary is empty.

     For the distinguishing property, i.e.\ property (iii), let $x,y\in I(X)$. There is $z\in X$ such that either $d(x,z)\neq d(y,z)$ or $d(z,x)\neq d(z,y)$. As usual, we consider only the first possibility. Moreover, by symmetry between $x$ and $y$ we can always assume $d(x,z)<d(y,z)$. Take $m$ such that $d(x,z)<m< d(y,z)$ Consider the set
    \[
    U=\{u\in X | d(x,u) <m \,\mathrm{and}\, d(y,u)>m\}.
    \]
    Clearly, $U$ is open and $z\in U$.  Since $I(X) $ is dense in $X$, there is some $z' \in U \cap I(X)$. By construction $z'$ distinguishes $x$ from $y$.

    In order to show that $I(X)$ has no chronological boundary, for every $x\in I(X)$ we have to find $p,q\in I(X) $ such that $p\ll x \ll q$. Note that $I^{+}(x)$ and $I^{-}(x)$ are non-empty (as $x\in I(X)$) open sets, thus they contain some points of the dense set $I(X)$.

    Let $T$ be a topology that satisfies (ii) for $X$. The function $d'$ is continuous in the induced topology $T'$.  By Thm.\ \ref{thm:propii-via-chdiamonds}, in order to prove that $I(X)$ is a Lorentzian metric space, it is sufficient to prove that for every $p,q\in I(X)$, $I(p,q)$ is relatively compact in the induced topology.
    Observe that there are $p',q'\in I(X)$ such that $p'\ll p \ll q \ll q'$, thus there is $\epsilon>0$ sufficiently small such that
    \[
    I(p,q)\subset \{d_{p'}\ge \epsilon \}\cap \{d^{q'}\ge \epsilon \}=I_\epsilon(p',q')\subset I(p',q')\subset I(X),
    \]
     where the second set is closed. This proves that
    $\textrm{Closure}_{T}(I(p,q))\subset I(X)$ and hence $\textrm{Closure}_{T'}(I(p,q))=\textrm{Closure}_{T}(I(p,q))\cap I(X)=\textrm{Closure}_{T}(I(p,q))$. Every induced topology open covering of $\textrm{Closure}_{T'}(I(p,q))$ comes thus from a  topology open covering of $\textrm{Closure}_{T}(I(p,q))$, and since the latter set is compact, we conclude that $I(p,q)$ is relatively compact in the induced topology.
    \end{proof}

    \begin{remark}\label{rmk:InteriorDistinguishes}
        For further use we note that by the same argument one can show that if $(X,d)$ is a Lorentzian metric space, and $I(X)$ is dense in $X$, then for every $x,y\in X$ (not necessarily in $I(X)$) such that $x\neq y$, there is $z\in I(X)$ such that $d(x,z)\neq d(y,z)$ or $d(z,x)\neq d(z,y)$.
    \end{remark}

\section{Topology}
\subsection{General results}
Definition \ref{cg-lms} mentions a topology $T$ on the Lorentzian metric space $(X,d)$. So far it is not clear how many such topologies  exist, and if there are any common topological properties of Lorentzian metric spaces. The goal of this section is to clarify these points.

Not much can be stated without conditions on the boundary.
\begin{lemma}\label{lem:Hsdf}
    Let $(X,d)$ be a Lorentzian metric space and let  $T$ be a topology satisfying the requirements of property (ii) of Definition \ref{cg-lms}. Then $T$ is Hausdorff.
\end{lemma}
\begin{proof}
    The proof goes as in \cite[Prop.\ 1.4.]{minguzzi22}. We use the notation and the terminology there introduced.
We already know that for every $p, r\in X$ the functions $d_p$ and $d^r$ are continuous. Let $x\ne y$. Since $d$ distinguishes points we can find $z$ such that $d(z,x)=:a\ne b:= d(z,y)$ (the other case being analogous).
Set $m=(a+b)/2$. Then only one among $x$ and $y$ belongs to the open set $\{r: d(z,r)< m\}=d_z^{-1}((-\infty,m))$, the other belonging to $\{r: d(z,r)> m\}=d_z^{-1}((m,\infty))$, thus $x$ and $y$
are separated by open sets.
\end{proof}
Much more can be said if we assume that the chronological boundary is empty.
We are going to use the following fact which follows from the proof of \cite[Prop. 1.6]{minguzzi22} (so we omit it here).

\begin{lemma}\label{lem:subbasis}
    Let $X$ be a topological space, $x\in X$ and let $\mathscr{A}$ be a family of open subsets of $X$ such that (a) for any $y\in X$ not coinciding with $x$ there are $A,B\in \mathscr{A}$ satisfying $x\in A$, $y\in B$, $A\cap B=\varnothing$  and (b) there is $U\in \mathscr{A}$ such that $x\in U\subset C$, where $C$ is compact. Then $\mathscr{A}$ is a subbasis for the  neighborhood system of $x$.
\end{lemma}

\begin{proposition}\label{prop:lms-hdf-lc-unique}
    If $(X,d)$ is a Lorentzian metric space with no chronological boundary, then a topology $T$, satisfying the property (ii) is  Hausdorff, locally compact and  unique.

Moreover,
the sets  of the form
\begin{equation}  \label{don}
\{q: a<d(p,q)<b\}\cap  \{q: c<d(q,r)<e\}
\end{equation}
with $p,r\in X$ and $a,b,c,e\in \mathbb{R}\cup \{-\infty,\infty\}$, form a subbasis for $T$.

\end{proposition}
\begin{proof}
 We start from local compactness. Since $X$ has no chronological boundary, for any $x\in X$ there are $p,q\in X$ such that $x\in I(p,q)$. The chronological relation is open because $d$ is continuous, thus the chronological diamonds are open. Compactness of $\overline{I(p,q)}$ (Proposition \ref{prop:chdiamonds-compact}) proves that every point has a compact neighborhood.

Hausdorffness follows from Lemma \ref{lem:Hsdf}.
Let us come to  uniqueness. By the properties just proved, the family $\mathscr{A}$ given by the $T$-open sets of the form (\ref{don}) satisfies the assumptions  (a) and (b)   of  \cite[Prop. 1.6]{minguzzi22} (or Lemma \ref{lem:subbasis}), thus it provides a subbasis for the topology.

Any two topologies that share the same subbasis coincide, thus $T$ is unique.
\end{proof}

From now on, if a Lorentzian metric space $(X,d)$ has no chronological boundary, we  endow it with the unique topology satisfying  property (ii), and refer to it as \emph{the topology of the Lorentzian metric space} $(X,d)$ denoting it with $\mathscr{T}$.

\begin{corollary}\label{crl:isometry-isomorphism}
    Let $(X,d)$ and $(Y,d)$ be two Lorentzian metric spaces without  chronological boundary and let $\phi:X\to Y$ be a bijective distance-preserving map. Then $\phi$ is a topological homeomorphism.
\end{corollary}
\begin{proof}
    It is clear that the preimage and image of a set of the form (\ref{don}) under a distance-preserving map is again a set of the form (\ref{don}). So, by Proposition \ref{prop:lms-hdf-lc-unique} any such $\phi$ is both continuous and open, thus it is a homeomorphism.
\end{proof}
It is natural  to refer to bijective distance-preserving maps as \emph{isometries}.

In the general case only the neighborhood basis for points in $I(X)$ is determined by $d$.

\begin{proposition}\label{prop:IntNbhdSBs}
        Let $(X,d)$ be a Lorentzian metric space, and let $x\in I(X)$. The sets of the form (\ref{don}) provide a subbasis of neighborhoods of $x$. If $I(X)$ is dense in $X$ (in any topology of $X$ satisfying requirements of the property (ii)), it is enough to consider the points $p,r\in  I(X)$ in (\ref{don}).
\end{proposition}

\begin{proof}
    We intend to use Lemma \ref{lem:subbasis}, so we have to show that the family consisting of the open sets of the form (\ref{don}) satisfies the two properties. We  have already seen that these sets separate all points in Lemma \ref{lem:Hsdf}, so we need only to show that every $x\in I(X)$ has a neighborhood among them contained in a compact set. Take $p,q\in X$ such that $x\in I(p,q)$. Then $I(p,q)$ can be written in the form (\ref{don}) and its closure is compact by Proposition \ref{prop:chdiamonds-compact}, as needed.

    Now, assume that  $I(X)$ is dense in $X$. We have to show that the sets of neighborhoods of the form (\ref{don}) with the restriction  $p,r\in I(X)$ is still large enough to apply Lemma \ref{lem:subbasis}. By Lemma \ref{lem:InteriorHasNoBoundary} for $x \in I(X)$ we can choose $p,r\in I(X)$ so that  $p\ll x \ll r$. Also, using the same argument as in proof of Lemma \ref{lem:InteriorHasNoBoundary}, we conclude that for any $y\in X$ not coinciding with $x$ there is $z\in  I(X)$ such that $d(x,z)\neq d(y,z)$ or $d(z,x)\neq d(z,y)$. So, in the proof of Lemma \ref{lem:Hsdf} we can also use only the points of  $I(X)$.
\end{proof}

The following result follows from Lemma \ref{lem:InteriorHasNoBoundary} and comments given there, taking into account that $I(X)$, being without chronological boundary, has a unique topology that satisfies (ii), which then necessarily has to be the induced topology.
\begin{corollary}
    Let $(X,d)$ be a Lorentzian metric space and assume that $I(X)$ is dense in $X$ in some topology satisfying the requirements of Definition \ref{cg-lms}. Then the tautological inclusion $I(X)\to X$ is a topological embedding (assuming that $I(X)$ is endowed with its unique, by Lemma \ref{lem:InteriorHasNoBoundary}  and Proposition \ref{prop:lms-hdf-lc-unique}, topology).
\end{corollary}
%
%
%

\begin{lemma}\label{lem:chron-convex}
    Let $(X,d)$ be a Lorentzian metric space. Consider $x,p,q,r\in X$ and $a,b\in [-\infty,\infty]$. Assume that $p\ll x \ll q$. Then the following holds:
    \begin{enumerate}
        \item if $p,q\in {(d_r)}^{-1}((a,b))$, then $x\in  {(d_r)}^{-1}((a,b))$;
        \item if $p,q\in {(d^r)}^{-1}((a,b))$, then $x\in  {(d^r)}^{-1}((a,b))$.
    \end{enumerate}
\end{lemma}
\begin{proof}
We prove the former statement, the proof of the latter being analogous.
The case $b\le  0$ is trivially satisfied as the assumption in the implication is void, ${(d_r)}^{-1}((a,b))$ being empty.  Thus, let us assume $b>0$.


    Assume that $d(r,p),d(r,q)\in (a,b)$. We have to show that $d(r,x)\in (a,b)$.

     Let us begin with proving $d(r,x)>a$. If $a<0$, this inequality follows from non-negativity of d. If $a\geq 0$ we have $d(r,p)>a\geq 0$ and $x\gg p$, so
    \[
        d(r,x)\geq d(r,p)+d(p,x)>a.
    \]

    Now, let us prove $d(r,x)<b$. If $d(r,x)=0$, then $d(r,x)<b$ by positivity of $b$. Otherwise, if $d(r,x)> 0$, then, by $x\ll q$,
    \[
        d(r,q)\geq d(r,x)+d(x,q).
    \]
    Recall that $d(r,q)<b$. Therefore, by non-negativity of the Lorentzian distance, $d(r,x)<b$.

\end{proof}
    \begin{corollary}\label{cor:chron-convex}
    Let $(X,d)$ be a Lorentzian metric space, and let $V\subset X$ be a finite intersection of sets of the form (\ref{don}). If $p,q\in V$, then $I(p,q)\subset V$.
\end{corollary}
\begin{proof}
    $V$ is an intersection of sets considered in Lemma \ref{lem:chron-convex}. Applying the lemma to each set and using elementary properties of intersections, we get the statement.
\end{proof}
\begin{theorem}\label{thm:aleksandrov}
    Let $(X,d)$ be a Lorentzian metric space. Assume that for every $x\in X$ we have 
    $x \in \overline{I^\pm (x)}$. Then the set of all chronological diamonds of $X$ is a basis for the topology of $X$. In other words, under these assumptions, the Lorentzian metric space topology coincides with the Alexandrov topology.
\end{theorem}

Note that the assumption implies $X=I(X)$, thus by Prop.\ \ref{prop:lms-hdf-lc-unique} the topology of $X$ is uniquely determined.
\begin{proof}
    Fix an open set $U\subset X$ and $x\in U$. The assumptions imply the absence of chronological boundary, thus by Proposition \ref{prop:lms-hdf-lc-unique}, there is $V\subset U$ such that $x\in V$, and $V$ is a finite intersection of sets of the form (\ref{don}). By assumption, there are $p\in I^-(x)\cap V$ and $q\in I^+(x)\cap V$. Clearly, $x\in I(p,q)$.
    By Corollary \ref{cor:chron-convex}, $I(p,q)\subset V$. We conclude that starting from a point $x$ and its open neighborhood $U$ we can always find a subneighborhood of the form $I(p,q)$ with some $p,q\in X$, i.e.\ chronological diamonds form a base of the topology of $X$.
\end{proof}

Further results on topology will be obtained in Prop.\ \ref{prop:cg-Polish} by adding a countability condition.

\subsection{Bounded regions of Lorentzian metric spaces}\label{ssec:breg}
The local properties of a locally compact metric space can  be studied by looking at its bounded regions, and thus essentially reducing everything to the compact case. Up to the issues with the chronological boundaries, the Lorentzian metric spaces are locally compact (see Proposition \ref{prop:lms-hdf-lc-unique}). So, it is instructive to use a similar strategy for the Lorentzian metric spaces, with the bounded Lorentzian metric spaces playing the roles of the compact regions. The main change in comparison with the standard metric theory is that Definition \ref{cg-lms} included the distinguishing property (iii) which is non-local. So, for a general space, not only we have to restrict our attention to a compact region of the space, but also to take a distance quotient \cite{minguzzi22}, losing more information about the properties of the space. Still, this approach can be powerful if used appropriately.  In this subsection we develop this idea, mostly focusing on topological application.

To start with, let us discuss how to get a bounded Lorentzian metric space from a compact region of an unbounded one. For any subset $Y\subset X$ we introduce
\[
I^0(Y):=\{y\in Y| d(x,y)=d(y,x)=0, \forall x\in Y\}.
\]
This set will play the role of a local spacelike infinity when we restrict our considerations to the subset $Y$. We also introduce $\mathring{Y}=Y \setminus I^{0}(Y)$.

\begin{lemma}\label{lem:clms-blms}
    Let $(X,d)$ be a  Lorentzian metric space without the chronological boundary, and $Y\subset X$ a compact subset. Let us endow $\mathring{Y}$ with a distance function $d\vert_{\mathring{Y} \times \mathring{Y}}$ and define
    \[
    BY:=\mathring{Y}/\!\sim,
     \]
     $d_{BY}=d\vert_{\mathring{Y} \times \mathring{Y}}/\!\sim$, where $\sim$ is the equivalence relation on $\mathring{Y}$ induced by the distance function $d\vert_{\mathring{Y} \times \mathring{Y}}$ (see Subsection 1.3 of \cite{minguzzi22}). Then $(BY,d_Y)$ is a bounded Lorentzian metric space. The topology $\mathscr{T}_{BY}$ of the bounded Lorentzian metric space $BY$ coincides with $\mathscr{T}\vert_{\mathring{Y}}/\!\sim$.
\end{lemma}
\begin{proof}
We need to verify three properties. Property (i) on the reverse triangle inequality is trivially satisfied. By  \cite[Prop.\ 1.19]{minguzzi22} we need only to show that there is a topology $\check T$ for $\mathring{Y}$ such that $d\vert_{\mathring{Y}\times \mathring{Y}}$ is continuous and  $\{(p,q)\in \mathring{Y}\times \mathring{Y}: d(p,q) \ge \epsilon \}$ is compact in $\check T \times \check T$. Let $\check T=\mathscr{T}\vert_{\mathring{Y}
}$ i.e.\ the induced topology on $\mathring{Y}$ by $\mathscr{T}$, then the induced metric $d\vert_{\mathring{Y}\times \mathring{Y}}$ is continuous in this induced topology. Moreover,
\[
\{(p,q)\in \mathring{Y}\times \mathring{Y}: d(p,q) \ge \epsilon \}=I_\epsilon \cap (Y\times Y)
\]
thus it is compact in $\mathscr{T}\times \mathscr{T}$ and hence in $\check T \times \check T$ as $\mathscr{T}$ and $\check T$  agree on $\mathring{Y}$.

Since $\mathring{Y}/\sim$ does not contain $i^0$ its topology is unique and coincides with any (necessarily unique) topology satisfying property (ii) of bounded Lorentzian metric space \cite[Cor.\ 1.7]{minguzzi22}. In particular, $T:=\check T/\sim$ satisfies such property (because the projection quotient of a compact set is a compact set), which concludes the proof.
\end{proof}

\begin{corollary}
    For any finite subset $A\subset X$ of a Lorentzian metric space without a chronological boundary, $B\overline{I(A)}$ is a bounded Lorentzian metric space. In particular, the closure of any chronological diamond produces a bounded Lorentzian metric space.
\end{corollary}
\begin{proof}
    By Proposition \ref{prop:chdiamonds-compact} and Lemma \ref{lem:clms-blms}.
\end{proof}

In the sequel we omit the index $Y$ in $d_{BY}$, because it will be  clear from the context. However, it is very important to distinguish between a point $x\in \mathring{Y}$ and its equivalence class $[x]_Y \in BY$. Since we need to consider different compact sets $Y$ to get information about the whole space $X$, we will always keep the subscript $Y$ of the equivalence class $[x]_Y$.

Let us study the relation between the points of a Lorentzian metric space and the corresponding equivalence classes in more detail. We interchangeably understand $[x]_Y$ both as a point of the bounded Lorentzian metric space $BY$ and as the corresponding subset of $\mathring{Y}$ consisting of indistinguishable points. We will frequently use the observation that for $x\in \mathring{Y}$ and $y \in BY$ the statements $[x]_Y=y$ and $x\in y$ are equivalent.

Intuitively, one would expect that if we enlarge the compact region $Y$, the equivalence class of a fixed point $x$ will shrink (because the larger is $Y$, the better it distinguishes the points). There is a caveat in this argument, because as $Y$ becomes larger, new points equivalent to $x$ may appear. The correct version of this statement is the following.

\begin{lemma} \label{cqptx}
    Let $(X,d)$ be a Lorentzian metric space. Let $p,q\in X$ and let us consider compact sets $Y,Y'\subset X$ such that $p,q\in Y$ and $I(p,q)\subset Y\subset Y'$. For  $x\in I(p,q)$ we have $[x]_{Y'} \subset [x]_{Y}$.
\end{lemma}
\begin{proof}
    It is enough to consider the case $p\ll q$.
     Assume that $y\in [x]_{Y'}$. Then $d(p,y)=d(p,x)>0$ and $d(y,q)=d(x,q)>0$, so $y\in  I(p,q)\subset Y$. Moreover, for any $z\in Y\subset Y'$ we have $d(x,z)=d(y,z)$ and $d(z,x)=d(z,y)$. But this means that $y\in [x]_{Y}$ as desired.
\end{proof}

A natural question is to what extent the space $X$ can be reconstructed from its bounded regions. The answer is given by the following lemma.

\begin{lemma}\label{lemClasses}
    Let $X$ be a Lorentzian metric space without chronological boundary. Take $p,q \in X$ and a family $\{Y_{n}\}_{n\in \mathbb{N}}$ of compacts subset of $X$. Assume that for every $n\in \mathbb{N}$ it holds $I(p,q)\subset Y_{n}$ and $p,q \in Y_{n}$. Finally, assume that $\bigcup_{n}Y_{n}=X$. Then:
     \begin{enumerate}
        \item For every $x\in I(p,q)$ we have $\bigcap_{m}[x]_{Y_m}=\{x\}$;
        \item Suppose that a sequence of points $y_n \in BY_{n}$  has been given, such that, regarding them as equivalence classes,
        $y_m\subset y_n$ for any pair  $m\geq n$, and $y_{1} \subset I(p,q)$. Then there is one and only one $x\in I(p,q)$ such that $[x]_{Y_n}=y_n$ for every $n \in \mathbb{N}$.
    \end{enumerate}
\end{lemma}
In other words, there is one-to-one correspondence between points of $X$ and nested families of   equivalence classes.

We note that the present formulation can be applied to the $\sigma$-compact Lorentzian metric spaces only. However, it is easy to see that countability of the family $\{Y_n\}$ plays no role in the proof. In fact, we can take any enlarging family of compact sets enumerated by any directed set. We omit this generalisation for the sake of readability of the statement, as we shall apply it only in the $\sigma$-compact case.
\begin{proof}

    \begin{enumerate}
        \item Clearly, $x\in \bigcap_{m}[x]_{Y_m}$ as it belongs to all of its equivalence classes. Now take $y\in \bigcap_{m\ge 1}[x]_{Y_m}$ and $z\in X$. There is  $m$ such that $z\in Y_m$. If $z\in I^0(Y_m)$, then $d(z,x)=d(z,y)=0$, and $d(x,z)=d(y,z)=0$. Otherwise, $z\in \mathring{Y}_m$, and, as $y\in [x]_{Y_m}$, $d(x,z)=d(y,z)$ and $d(z,x)=d(z,y)$. Since $z$ is arbitrary, it means that $x=y$. So, the only point in the intersection is $x$.
        \item All equivalence classes $y_m$ are closed (as intersections of closed sets), non-empty, and contained in $Y_1$. Moreover, the nesting property $y_m\subset y_n$ for $m\ge n$ implies that any finite subfamily of them has a non-empty intersection. Then by compactness of $Y_1$ the set $L=\bigcap_{m\ge 1}y_m$ is non-empty. Clearly, $x\in L$ if and only if for every $m$ it holds $[x]_{Y_m}=y_m$. But then, by the previously proven point, $L$ contains exactly one element, so such an element $x$ is unique.
    \end{enumerate}
\end{proof}

The lemma above explains how the points of the original space can be recovered from the equivalence classes of its bounded regions. Another useful observation is that the open sets of a Lorentzian metric space without chronological boundary can be recovered from the topologies of its bounded regions via the following procedure.

\begin{lemma}\label{lem:top-from-bound}
    Let $(X,d)$ be a Lorentzian metric space and let $\mathscr{G}\subset X$ be a generating set. Then for every $x\in X$ and open set $O\ni x$, there are an open set $V\subset X$, a finite subset $A\subset \mathscr{G}$ and an open set $U\subset B\overline{I(A)}$ such that $x\in V\subset O$ where $V$ is the preimage of $U$ under the canonical projection $\mathring{\overline{I(A)}}\to B\overline{I(A)}$. If $A$ does the job, so does any larger finite set $A'\subset \mathscr{G}$.
\end{lemma}
In other words, we can get the base of the topology of $X$ by lifting the bases of topologies of its bounded regions.
\begin{proof}
    Fix $p,q\in \mathscr{G}$ so that $p\ll x\ll q$ and let us replace $O$ by $O'=O\cap I(p,q)$. Clearly, $x\in O'\subset O$, so it is enough to find a sub-neighborhood of $O'$ satisfying the requirements.
    By Proposition \ref{prop:lms-hdf-lc-unique}, we can find a sub-neighborhood $V\subset O'$ such that $x\in V$ and
    \[
    V=\bigcap_{i\leq n}(d^{r_i})^{-1}((a_i,b_i)) \cap \bigcap_{i\leq n}(d_{s_i})^{-1}((c_i,d_i))
    \]
    for some $n\in\mathbb{N}$, $r_1,\ldots,r_n,s_1,\ldots, s_n\in X$ and $a_i,b_i,c_i,d_i\in [-\infty,+\infty]$. As $\mathscr{G}$ is a generating set, we can find a finite set $A\subset \mathscr{G}$ so that $r_i,s_i\in I(I(A))$ for every $i=1,\ldots,n$, and $p,q\in I(A)$. We  have $V\subset I(p,q)\subset \mathring{\overline{I(A)}}$. Observe that all points $r_i$ and $s_i$ belong to $\mathring{\overline{I(A)}}$. Thus $V$ is the preimage of an open set $U\subset B\overline{I(A)}$ defined by
    \[
    U=\bigcap_{i\leq n}(d^{[r_i]_{\overline{I(A)}}})^{-1}((a_i,b_i)) \cap \bigcap_{i\leq n}(d_{[s_i]_{\overline{I(A)}}})^{-1}((c_i,d_i)).
    \]
\end{proof}

\begin{lemma}\label{lem:dense-sets-lift}
    Let $\{Y_i\}_{i\in \mathbb{N}}$ be a compact exhaustion of a Lorentzian metric space without chronological boundary $(X,d)$. Assume that for each $n\in\mathbb{N}$, $\mathscr{S}_n\subset \mathring{Y}_n$ is such that
    \[
    \mathscr{S'}_n=\{[p]_{Y_n}| p\in \mathscr{S}_n\}
    \]
    is dense in $BY_n$.
    Then $\mathscr{S}=\bigcup_{n\in\mathbb{N}} \mathscr{S}_n$ is dense in $X$.
\end{lemma}
\begin{proof}
    To start with, we observe that for any finite set $A$ we can find $n\in\mathbb{N}$ such  that $A\subset \mathring{Y}_n$. Indeed, $(X,d)$ has no chronological boundary, so $A\subset I(A')$ for some finite set $A'\subset X$. For some $n\in\mathbb{N}$, $A'\cup A\subset Y_n$. It is easy to see that in this case $A\cap Y_n^0=\varnothing$.

    Consider a non-empty open set  $O\subset X$ and let $x\in O$.  By Proposition \ref{prop:lms-hdf-lc-unique}, we can find a subset $V\subset O$ such that $x\in V$ and
    \[
    V=\bigcap_{i\leq m}(d^{r_i})^{-1}((a_i,b_i)) \cap \bigcap_{i\leq m}(d_{s_i})^{-1}((c_i,d_i))
    \]
    for some $m\in\mathbb{N}$, $r_1,\ldots,r_m,s_1,\ldots, s_m\in X$ and $a_i,b_i,c_i,d_i\in [-\infty,+\infty]$. Let $n\in\mathbb{N}$ be such that $r_1,\ldots,r_m,s_1,\ldots,s_m\in \mathring{Y}_n$. Then $V\cap \mathring{Y}_n$ is the preimage of an open set
    \[
    U=\bigcap_{i\leq m}(d^{[r_i]_{Y_n}})^{-1}((a_i,b_i)) \cap \bigcap_{i\leq m}(d_{[s_i]_{Y_n}})^{-1}((c_i,d_i))\subset BY_n
    \]
    under the canonical projection $\mathring{Y}_n\to BY_n$.
    As $\mathscr{S'}_n$ is dense in $BY_n$, there is $y'\in U\cap \mathscr{S'_n}$, and the representative $y$, $y'= [y]_{Y_n}$, is such that  $y\in  V \cap \mathscr{S}$. We have found that any open set of $X$ contains a point of $\mathscr{S}$, i.e.\ $\mathscr{S}$ is dense.
\end{proof}
\begin{corollary}\label{crl:comp-dense}
    Let $(X,d)$ be a Lorentzian metric space without chronological boundary, and $\{Y_n\}$ be a compact exhaustion of $X$. Then there is a dense countable  subset $\mathscr{S}\subset X$ such that for every $n\in\mathbb{N}$ the set $\mathscr{S}\cap \mathring{Y}_n$ projects to a dense countable subset of $BY_n$.
\end{corollary}
\begin{proof}
    Each bounded Lorentzian metric space $BY_n$ has a countable dense subset $\mathscr{S}'_n\subset BY_n$ \cite{minguzzi22}. By choosing an arbitrary representative of each class, we can lift $\mathscr{S}'_n$ to $\mathscr{S}_n\subset X$ which projects to $\mathscr{S}_n$. By Lemma \ref{lem:dense-sets-lift}, $\mathscr{S}=\bigcup_{n\in\mathbb{N}} \mathscr{S}_n$ is dense in $X$. At the same time, for any $n\in\mathbb{N}$, $\mathscr{S}\cap \mathring{Y}_n \supset \mathscr{S}_n$ is dense in $BY_n$.
\end{proof}

\subsection{Countably-generated and sequenced spaces}
\begin{definition}
    Let $(X,d)$ be a Lorentzian metric space. We say that $X$ is \emph{countably-generated} if it admits a countable generating set.
    We say that $(X,d,\{p^k\}_{k\in\mathbb{N}})$ is a \emph{sequenced} Lorentzian metric space if $(X,d)$ is a Lorentzian metric space and $\{p^i\}_{i\in \mathbb{N}}$ is a sequence of points in $X$ such that
    \[
    \{p^i| i\in \mathbb{N}\}
    \]
    is a generating set. We will refer to $\{p^k\}_{k\in\mathbb{N}}$ as a generating sequence.
\end{definition}
Note that  countably-generated Lorentzian metric spaces
do not have a chronological boundary (Remark \ref{rmk:genset=no-bnd}).

\begin{proposition}
    Let $(M,g)$ be a smooth globally hyperbolic Lorentzian  manifold (without boundary), and let $(M,d)$ be the associated, by Proposition \ref{prop:globhyp-is-LMS}, Lorentzian metric space. Then $M$ is countably generated.
\end{proposition}
\begin{proof}
    Since $M$ is a manifold, there is a countable dense subset $\mathscr{G}\subset M$. For every $x\in M$ both $I^{+}(x)$ and $I^{-}(x)$ are non-empty open sets, so there are $p\in I^{-}(x)\cap \mathscr{G}$ and $q\in I^{+}(x)\cap \mathscr{G}$. Therefore, $\mathscr{G}$ generates $X$.
\end{proof}

The difference between countably generated   and  sequenced Lorentzian metric spaces is similar to the difference between metric spaces and pointed metric spaces. We shall  discuss some properties of  countably generated Lorentzian metric spaces.

Observe that we can pass from a countably generated to a sequenced Lorentzian metric space by  fixing an order for the sequence,  whenever it could be convenient. Given two sequenced Lorentzian metric spaces, say $(X,d,\{p^k\}_{k\in\mathbb{N}})$ and  $(X',d',\{p'{}^k\}_{k\in\mathbb{N}})$, it is natural to consider a special class of maps $X\to X'$, mapping $p^k$ to $p'{}^k$.
Clearly, this class depends on  specific choices for the generating sequences. In particular, we introduce the following notion.
\begin{definition}\label{def:cg-iso}
    We say that two sequenced Lorentzian metric spaces
    \[
    (X,d,\{p^k\}_{k\in \mathbb{N}}) \ \ \textrm{and} \ \
      (X',d',\{p'{}^k\}_{k\in \mathbb{N}})
     \]
       are \emph{isomorphic} if there exists a bijective distance preserving map $\phi: X \to X'$ such that $p'{}^n=\phi(p^n)$ for every $n\in \mathbb{N}$. We also say that such a map $\phi$ is an \emph{isomorphism}.
\end{definition}
By Corollary \ref{crl:isometry-isomorphism} any isomorphism of sequenced Lorentzian metric spaces is an isometry and thus a topological homeomorphism (see Corollary \ref{crl:isometry-isomorphism} and the paragraph below it).

In Section \ref{sec:GH}, devoted to  Gromov-Hausdorff convergence, the choice of the sequence will be of importance.

\begin{proposition}
    \label{prop:cg-Polish}
    If $(X,d)$ is a countably generated Lorentzian metric space, then its topology is $\sigma$-compact, second-countable and Polish.
\end{proposition}
\begin{proof}
We start from $\sigma$-compactness.
Let $\mathscr{G}=\{p^n| n\in \mathbb{N}\}$ be a countable generating set of $X$. Set $X^m=\overline{I(p^1,\ldots,p^m)}$. Then
\[
X=\bigcup_{m}X^m.
\]
By Proposition \ref{prop:chdiamonds-compact} the sets $\overline{I(p^1,\ldots,p^n)}$ are compact, so $X$ is $\sigma$-compact.

Let us prove second countability. We recall that for every $m$ the topology $\mathscr{T}^m$ of the bounded Lorentzian metric space $BX^m$ is second countable \cite{minguzzi22}. We then have only to show that a union of their countable bases (lifted to $X$) is a base of the topology topology $\mathscr{T}$. This follows from Lemma \ref{lem:top-from-bound}.

Every second
countable locally compact Hausdorff space (hence $\sigma$-compact) is Polish (see the argument in \cite[Prop.\ 3.1]{minguzzi11c}\cite[Thm.\ 1.10]{minguzzi22}).
\end{proof}
\begin{corollary}\label{crl:dense}
   Let $(X,d)$ be a countably generated Lorentzian metric space. Then there is a countable dense subset $\mathscr{S}\subset X$. Moreover, if $x,y\in X$ and $x\neq y$, then there is $z\in \mathscr{S}$ such that $d(x,z)\neq d(y,z)$ or $d(z,x)\neq d(z,y)$. Finally, the subbasis of the topology of $X$ is given by the sets of the form (\ref{don}) with $p,r\in \mathscr{S}$.
\end{corollary}
\begin{proof}
Existence of a countable dense subset is a consequence of $X$ being Polish. The rest is identical to  \cite[Prop.\ 1.11, Cor.\ 1.14]{minguzzi22}.
\end{proof}

\subsection{Bounded Lorentzian metric spaces}
In Subsection \ref{ssec:breg} we have seen that the compact regions of a Lorentzian metric space produce bounded Lorentzian metric spaces that provide information on the local properties of the original space. We applied this idea in Prop.\ \ref{prop:cg-Polish}. In this subsection we go in the reverse direction, namely, we study how the general theory of Lorentzian metric spaces can be applied to the bounded case.

From Example \ref{exmp:blms} we already know that bounded Lorentzian metric spaces, as the terminology suggests, are special cases of Lorentzian metric spaces. At the same time, the bounded Lorentzian metric spaces do not belong to the class of the spaces without boundary (and thus they can not be countably generated) for which the most strong results were achieved.
\begin{lemma}
    Let $(X,d)$ be a bounded Lorentzian metric space. Then $I^{+}(X)=I^{+}(X^-)$ and $I^{-}(X)=I^{-}(X^+)$. In particular, $I(X)=I(X^+\cup X^-)$.
\end{lemma}
In other words, the boundaries have enough points to generate (in the sense of Definition \ref{def:gen-set}) everything except the boundaries themselves.
\begin{proof}
    Let us first prove  $I^{+}(X)=I^{+}(X^-)$. By monotonicity of $I^+$, $I^{+}(X^-)\subset I^+(X)$, so we have to show the other inclusion only.
    Fix $x\in I^{+}(X)$. By definition of $I^+$, there is $y\in X$ such that $\epsilon=d(y,x)>0$. Let $p$ be a point of maximum of $d^{x}$ on the compact set $K=\pi_1(I_{\epsilon})$, where $\pi_1$ is the canonical projection $X\times X\to X$ to the first factor. As $y\in K$ we have $d(p,x)\geq \epsilon>0$. Now, if there were $r\ll p$, then we would have
    \[d^{x}(r)=d(r,x)\geq d(r,p)+d(p,x)>d(p,x)\geq \epsilon,\]
    so $r\in K$ and $p$ would not be a point of maximum, in contradiction to our assumption. We conclude that  $p\in X^-$ and $p\ll x$, i.e. $x\in I^{+}(X^-)$.

    The equality $I^{-}(X)=I^{-}(X^+)$ can be proven similarly. Finally,
    \[
    I(X)=I^{+}(X^+)\cap I^{-}(X^-)\subset I^{+}(X^+\cup X^-) \cap I^{-}(X^+\cup X^-)=I(X^+\cup X^-),
    \]
    $I(X^+\cup X^-)\subset I(X)$ by monotonicity, so $I(X)=I(X^+\cup X^-)$.
\end{proof}
When the boundaries can be cut out, the resulting space is countably generated.
\begin{theorem}\label{thm:blmc-cg}
    Let $(X,d)$ be a bounded Lorentzian metric space with the topology $\mathscr{T}$.
    If $I(X) $ is dense in $X$, then $(I(X) ,d_{I(X) \times I(X) })$ is a countably-generated  Lorentzian metric space.
\end{theorem}
Note that the density assumption does not hold for finite causets.
\begin{proof} By Lemma \ref{lem:InteriorHasNoBoundary} we need only to show that $I(X) $ is countably generated.

By Remark \ref{rmk:IntIsOpen}, $I(X) $ is open. It follows that if $\mathscr{S}$ is dense in $X$, then so is $\mathscr{S} \cap I(X) $ (take a non-empty open set $O\subset X$, then $O\cap I(X) $ is also open and non-empty by density of $I(X) $, so  $O\cap I(X) \cap \mathscr{S}$ is non-empty).

    Now let $\mathscr{S}$ be a countable dense set of $X$ and set $\mathscr{G}=\mathscr{S} \cap I(X) $. We claim that it generates $I(X) $. Indeed, take $x\in I(X) $ and consider the open sets $I^+(x)$ and $I^{-}(y)$. They are non-empty by definition of $I(X) $, thus there are some $p\in I^-(x) \cap \mathscr{G}$ and $q\in I^+(x) \cap \mathscr{G}$ as desired.
\end{proof}

Another question worth answering is how one can identify bounded Lorentzian metric spaces among general Lorentzian metric spaces. We give one such  criterion.
\begin{lemma}
    The bounded Lorentzian metric spaces with spacelike boundary are precisely the Lorentzian metric spaces, such that
    \begin{enumerate}
        \item The spacelike infinity is not empty;
        \item At least one of the topologies satisfying property (ii) of Definition \ref{cg-lms} is compact.
    \end{enumerate}
\end{lemma}
\begin{proof}
    This is direct consequence of \cite[Prop. 1.15]{minguzzi22}.
\end{proof}
\section{Causal relation and time functions}\label{sec:causal}\label{ssec:glob-hype}

\subsection{Causal relation and global hyperbolicity }
In a Lorentzian metric space  there is a natural chronological order relation, defined as the set $I=\{d>0\}$. In \cite{minguzzi22} a (maximal) causal relation compatible with the distance function $d$ and the reverse triangle inequality was defined. We recall it here, dropping unnecessary topological assumptions  for future convenience.
\begin{definition}\label{def:causal}
   Let $X$  be a set endowed with  a function $d: X\times X \rightarrow [0,+\infty)$ satisfying the reverse triangle inequality (i.e.\ property (i) of a Lorentzian metric space). The {\em (extended/maximal) causal relation} $J \subset  X\times X$ is the set
    \begin{equation}
        J=\{(x,y)\in X\times X| \ d(p,y) \ge d(p,x) \ \mathrm{and} \ d(x,p)\ge d(y,p) ,\ \forall p\in X\}.
    \end{equation}
\end{definition}
As in \cite{minguzzi22}, we use the notation $x\le y$ for $(x,y)\in J$, and $x<y$ for $x\leq y$ but $x\ne y$.
\begin{remark}\label{rmk:glob-caus-strong}
    If $(X,d)$ is a Lorentzian metric space and $C\subset X$ is a compact subset, then there are two natural orders on $\mathring{C}$, namely the  restriction of $J$ to $\mathring{C}$ and the  lift of the causal order on $BC$. It is easy to see that the former is contained in  the latter, i.e.\  for $x,y\in \mathring{C}$,  $x \leq y$  implies $[x]_C\leq [y]_C$.
\end{remark}
\begin{definition}
     Let $X$ and $d$ be as in Definition \ref{def:causal}, and let us consider a set $A\subset X$. We define the \emph{causal future} of $A$ as
    \[
    J^{+}(A)=\{p\in X| \ \exists q\in A, \, p \ge q\},
    \]
    the \emph{causal past} of $A$ as
    \[
    J^{-}(A)=\{p\in X| \ \exists q\in A, \, p \le q\},
    \]
    and the \emph{causal hull} of $A$ as
    \[
    J(A)=J^{+}(A)\cap J^{-}(A).
    \]
    If $A$ is finite it is convenient to use the notation
    \[
    J^{\pm}(p_1,\ldots,p_n):=J^{\pm}(\{p_1,\ldots, p_n\}),
    \]
    and similarly for $J$.
\end{definition}
We remark that this notation is compatible with the conventions of Subsection \ref{ssec:notation}.
\begin{theorem} \label{thm:Jproperties}
    Let $(X,d)$ be as in Definition \ref{def:causal}. The relation $J$ is closed, reflexive and transitive, $I\subset J$ and $I\circ J\cup J\circ I\subset I$. If $(x,y), (y,z)\in J$ then $d(x,y)+d(y,z)\le d(x,z)$. If  $(X,d)$ satisfies (iii)  then $J$ is antisymmetric.
\end{theorem}
\begin{proof}
    The proof of \cite[Thm.\ 5.7]{minguzzi22} applies. Note that  closedness follows from continuity of $d$, while transitivity  and the last two properties relating $J$ to $I$ follow from the triangle inequality (but not from the other properties of a bounded Lorentzian metric space). In particular, the compactness assumptions of property (ii) are never used. The last statement on antisymmetry is clear from the definition of $J$.
\end{proof}

  We  establish a connection with the  property of global hyperbolicity for ordered spaces in the sense of \cite{minguzzi12d,minguzzi23}.
    \begin{theorem}\label{thm:lms-via-emeralds}
        Let $(X,d)$ be a Lorentzian metric space without chronological boundary, and $C\subset X$ any compact subset. Let $K$ be a closed order on $X$ such that $I\subset K\subset J$. Then the set $K(C)$ is compact.
    \end{theorem}

    The typical choice will be $K=J$.

    \begin{proof}
        By \cite[Prop. 1.4]{nachbin65} $K^{+}(C)$ and $K^{-}(C)$ are closed, so $K(C)$ is closed.

        Following \cite{minguzzi19c}, we are going to show that $J(C)\subset I(A)$ for some finite set $A\subset X$. By assumptions, $K(C)\subset J(C)$, and $I(A)$ is relatively compact by Proposition \ref{prop:chdiamonds-compact}, so this will be enough to conclude that $K(C)$ is compact.

         For every $x\in C$ take $p,q\in X$ such that $p\ll x \ll q$. Then $I(p,q)$ is an open neighborhood of $x$, and thus sets of this type form an open covering of $C$. So, there is a finite set $A\subset X$ such that $C\subset I(A)$.

    Now take $x\in J(C)$. It means that there are $p,q\in C$ such that $p\le x \le q$. Since $C\subset I(A)$, there are $p',q'\in A$ such that $p'\ll p$ and $q\ll q'$. By Theorem \ref{thm:Jproperties}, $p'\ll x \ll q'$, therefore $x\in I(A)$. We conclude that $J(C)\subset I(A)$.
    \end{proof}

In other words, in the absence of a chronological boundary, the property (ii) is essentially the global hyperbolicity assumption in the sense of \cite{minguzzi12d,minguzzi23}.
     Moreover, as the next result finish clarifying, the extended causal relation $J$, which is in a sense the maximal relation compatible with $d$, can be replaced by any smaller closed order $K$.

\begin{theorem} \label{thm:cnqg}
    Assume that $(X,d)$ satisfies Definition \ref{cg-lms} except for the property (ii) and that it has no chronological boundary.
    Then $(X,d)$ is a Lorentzian metric space if and only if there is a   relation $K$,  $I\subset K\subset J$,  and a topology of $X$ with respect to which $d$ is continuous and  $K(C)$ is compact for every compact set $C$.
\end{theorem}

Note that for $p,q\in X$, taking $C=\{p,q\}$, we get that $K(p,q)$ is compact.

\begin{proof}
If $(X,d)$ is a Lorentzian metric space then the claim on the existence of $K$ is true,  due to Theorem \ref{thm:lms-via-emeralds}, taking $K=J$.
For the converse, suppose that there is a topology such that $d$ is  continuous and $K(C)$ is compact for every compact set $C$, then for every $p,q\in X$, setting $C=\{p,q\}$ we find $I(p,q)\subset I(C)\subset K(C)$, and since $C$ is compact we conclude that $I(p,q)$ is relatively compact (since (iii) holds, any topology that makes $d$ continuous is Hausdorff by the usual arguments, so every compact set is closed). By Theorem \ref{thm:propii-via-chdiamonds}, $(X,d)$ is a Lorentzian metric space.
\end{proof}

\begin{remark}\label{rmk:larger-than-causal}
    Let $(X,d)$ be as above and assume that there is a generating set $\mathscr{G}\subset X$.
     By the property $I\circ J\cup J\circ I\subset I$, for every pair $x,y\in X$ there are $p,q\in \mathscr{G}$ such that $J(x,y)\subset I(p,q)$.
\end{remark}


\begin{theorem}\label{thm:lms-via-kdiam}
    Assume that $(X,d)$ satisfies Definition \ref{cg-lms} except for the property (ii), and let $\mathscr{G}\subset X$ be a generating set. Then $(X,d)$ is a Lorentzian metric space if and only if there is a   relation $K$,  $I\subset K\subset J$,  and a topology on $X$ with respect to which $d$ is continuous and  $K(p,q)$ is compact for every $p,q\in \mathscr{G}$.
\end{theorem}

This result is interesting already for spaces without chronological boundary as we can make the choice $\mathscr{G}= X$ which clarifies the connection  with a characterization of global hyperbolicity in the smooth case \cite{bernal06b} (causality, i.e.\ antisymmetry of $J$, and compactness of causal diamonds), in the sense that, taking also into account Thm.\ \ref{thm:cnqg}, property (ii) is reformulated as a compactness condition for $K$-causal diamonds where the choice of $K$ is largely arbitrary.

\begin{proof}
    We  have only to establish equivalence between the compactness assumptions. Assume that $(X,d)$ is a Lorentzian metric space, then by Remark \ref{rmk:larger-than-causal}, for every $p,q\in \mathscr{G}$, since $J(p,q)$ is closed, it is compact, thus the claim on the existence of $K$ is true, just choose $K=J$.

    For the converse, observe that the sets of the form $K(p,q)$, $p,q\in \mathscr{G}$ are closed  (since (iii) holds, any topology that makes $d$ continuous is Hausdorff by the usual arguments, so every compact set is closed). As $I\subset K$, $I(p,q)\subset K(p,q)$, for $p,q\in \mathscr{G}$ and so these chronological diamonds are relatively compact.
    By Corollary \ref{crl:prop-iip}, $(X,d)$ is a Lorentzian metric space.
    %
%
%
    \end{proof}

\begin{corollary}
    Let $C\subset X$ be a compact subset of a Lorentzian metric space $(X,d)$ with no chronological boundary. Then $BJ(C)$ is a bounded Lorentzian metric space.
\end{corollary}
\begin{proof}
    This follows from Lemma \ref{lem:clms-blms}.
\end{proof}

\subsection{Time functions}\label{ssec:time}
In the bounded case the existence of time functions constructed from $d$ was used recurrently in various arguments \cite{minguzzi22}. The existence of  time functions is  also an interesting problem in itself. Let us look for similar results in the unbounded case.

First, let us recall the definition.
\begin{definition}
    Let $(X, d)$ be a Lorentzian metric space. A continuous function $\tau: X\to \mathbb{R}$ with $\tau(x) <\tau(y)$ whenever $x<y$ is called a {\em time function}.
\end{definition}

\begin{lemma}
    \label{lem:time-func}
    Let $(X,d)$ be a countably-generated Lorentzian metric space. Then there is a bounded time function $\tau: X\to [-1,1]$.
\end{lemma}
\begin{proof}
    Let $\mathscr{S}\subset X$ be a countable dense subset in $X$ (its existence follows from Propositions \ref{prop:lms-hdf-lc-unique}). Fix an enumeration for $\mathscr{S}$, i.e.\ a sequence $\{x_n\}_{n\in \mathbb{N}}$ such that
    \[
    \{x_n|n\in\mathbb{N}\}=\mathscr{S}.
    \]
    Let us set
    \begin{equation} \label{oqrp}
    \tau(x)=\sum_{n=1}^{\infty}\frac{1}{2^{n}}\left(\frac{d(x_i,x)}{1+d(x_i,x)}  -\frac{d(x,x_i)}{1+d(x,x_i)}
    \right).
    \end{equation}
    The expression in the parenthesis has absolute value bounded by $1$, so the series converges absolutely and $|\tau(x)|\leq 1$. Now, if $x\leq y$, then $d(x_n,x)\leq d(x_n,y)$ and   $d(x,x_n)\geq d(y,x_n)$. Thus (using that $t\mapsto \frac{t}{t+1}$ is increasing on $[0,+\infty)$) we get $\tau(x)\leq \tau(y)$. If in addition to that $x\neq y$, then some of the points of $\mathscr{S}$ distinguish $x$ from $y$ (see Corollary \ref{crl:dense}), so some of the inequalities above become strict and we get $\tau(x)>\tau(y)$.
\end{proof}

We will also need another variant of this construction.
\begin{lemma}\label{lem:time-func-seq}
    Let $(X,d)$ be a countably-generated Lorentzian metric space, and let $\{Y_i\}_{i\in\mathbb{N}}$ be a family of compact subsets of $X$ such that $Y_i\subset Y_j$ for every $i\leq j$, and
    \[
    \bigcup_{i\in \mathbb{N}}Y_i=X.
    \]
    Then there is a family of time functions
    \[
    \tau_i: BY_i\to [-1,1], \ i\in\mathbb{N}
    \]
    such that the function $\tau: X\to [-1,1]$
    \[
    \tau(x):=\lim_{i\rightarrow \infty}\tau_i([x]_{Y_i}), \ \forall x\in X
    \]
    is well-defined and provides a time function.
\end{lemma}
In the lemma above the expression under the limit in general makes sense only for large enough  values of $i$ (namely, large enough to ensure that $x\in Y_i$), but the limit depends on the tail of the sequence only, so the whole expression is well-defined.
\begin{proof}
    Let $\mathscr{S}$ be a countable dense subset of $X$ such that  $\mathscr{S}_i=\mathscr{S}\cap \mathring{Y}_i$ projects to a  dense countable subset of  $B{Y}_i$, see Cor.\
    \ref{crl:comp-dense} for the existence.


   For each $n$ let $s^n_k$, $k\in\mathbb{N}$ be a denumerable sequence whose image is $\mathscr{S}_n$ (the sequence shall have repetitions if $\mathscr{S}_n$ is finite).


    For $n\in \mathbb{N}$ and $x\in \mathring{Y}_n$ set
    \[
    \tau_n([x]_{Y_n})=\sum_{k=1}^{n}\sum_{j=1}^{\infty} \frac{1}{2^{j+k}} \left(\frac{d(s_j^k,x)}{1+d(s_j^k,x)}-\frac{d(x,s_j^k)}{1+d(x,s_j^k)}\right).
    \]
    Note that the right hand side depends on the class $[x]_{Y_n}$ only, so $\tau_n$ is a well-defined function $BY_n\to [-1,1]$ for each $n\in \mathbb{N}$. The proof that $\tau_n$ is a time function goes in the same way as in Lemma \ref{lem:time-func}.

    Finally, for fixed $x\in X$ we have
    \[
    \lim_{n\rightarrow\infty} \tau_n([x]_{Y_n})=\sum_{k=1}^{\infty}\sum_{j=1}^{\infty} \frac{1}{2^{j+k}} \left(\frac{d(s_j^k,x)}{1+d(s_j^k,x)}-\frac{d(x,s_j^k)}{1+d(x,s_j^k)}\right) .
    \]
    It is easy to see that the sum converges for any $x\in X$, and the resulting function $\tau$ is a time function.
\end{proof}

\begin{remark}\label{rmk:time-func-seq-uni}
    In fact we  proved a stronger result. The convergence of $\tau_n$ to $\tau$ is actually uniform. Moreover, there is a universal (independent of the space $X$) bound
    \[|\tau_n(x)-\tau(x)|\leq 2^{-n} \qquad (\forall x\in X).
    \]
\end{remark}
\begin{remark}\label{rmk:time-func-resc}
    In the setting of Lemma \ref{lem:time-func} (or Lemma \ref{lem:time-func-seq}), for any two given $x,y$ such that $x<y$, by rescaling and shifting of $\tau_n$ one can always ensure that $\tau(x)=0$ and $\tau(y)=1$.
\end{remark}
\section{Lorentzian length space}\label{sec:length}
As in \cite{minguzzi22}, we say that a continuous curve $\sigma: [a,b]\to X$ is \emph{isochronal} if $a\le s<t\le b$ implies $\sigma(s)\ll \sigma(t)$, \emph{isocausal} if  $a\le s<t\le b$ implies $\sigma(s)< \sigma(t)$, and  \emph{maximal} if it is isocausal and if for $a\le t<t'<t''\le b$ it satisfies
\begin{equation}
d(\sigma(t),\sigma(t'))+d(\sigma(t'),\sigma(t''))=d(\sigma(t),\sigma(t''))
\label{eqn:maxcurv}
\end{equation}
We recall that for any time function $\tau$, by\footnote{Of course, in \cite{minguzzi22} this was proved for bounded Lorentzian metric spaces, but the proof relies on continuity and monotonicity of $\tau$ and $\sigma$ only, thus applies to much more general setting.} \cite[Prop. 5.8]{minguzzi22} any isocausal curve $\sigma$ can be reparametrized to be $\tau$-uniform, i.e.\ such that
$\tau(\sigma(s))=s$ for any $s$ in the domain of $\sigma$.

\begin{definition}\label{def:lenspace}
    We say that a Lorentzian metric space $(X,d)$ is a Lorentzian \emph{prelength} space if for every $(x,y)\in I$ there exists an isocausal curve $\sigma:[a,b]\to X$ connecting $x$ to $y$. We say that $X$ is a Lorentzian \emph{length space}, if for any $(x,y)\in I$ there is a maximal curve $\sigma:[a,b]\to X$ connecting $x$ to $y$.
\end{definition}
\begin{remark}
    We note that the notions of Lorentzian (pre)length spaces differ from homonymous notions introduced in \cite{kunzinger18}. In fact, any Lorentzian metric space can be considered as a pre-length space of \cite{kunzinger18} (but not vice versa). A closer analogue of the prelength space property introduced above is the notion of causally path connected pre-length space of \cite{kunzinger18}, but the analogy is not complete since in \cite{kunzinger18} only locally Lipschitz curves are considered. The notions of Lorentzian length spaces here and in \cite{kunzinger18} are close, save, again for the locally Lipschitz condition restricting the set of the curves in the latter.  Since the locally Lipschitz condition clearly depends on the auxiliary metric, absent in our approach, more precise comparison is not possible.

\end{remark}

  For convenience, we extend the domain of definition of each curve to $\mathbb{R}$ by setting
    \begin{equation} \label{gid}
    \hat{\sigma}_n(s)= \begin{cases}
        \sigma_n(a_n),& \text{ if }s< a_n\\
        \sigma_n(s),& \text{ if }s\in[a_n,b_n]\\
        \sigma_n(b_n),& \text{ if }s>b_n.
    \end{cases}
    \end{equation}

    \begin{definition}
    Let $(X,d)$ be a Lorentzian metric space and
     let $\sigma_n:[a_n,b_n] \to X$, $\sigma: [a,b]\to X$ be continuous curves.
    We say that $\sigma_n$ {\em converges pointwise} (respectively, {\em converges uniformly}) to $\sigma$  if $a_n \to a$, $b_n\to b$, and $\hat \sigma_n$  converges  to $\hat \sigma$ pointwise (respectively, uniformly with respect to a uniformity on $X$ which should be specified).  Given a subset $Q\subset \mathbb{R}$, we say that  $\sigma_n$ {\em converges pointwise on $Q$} if $a_n \to a$, $b_n\to b$, and $\hat \sigma_n|_{Q}$  converges pointwise to $\hat \sigma|_Q$.
    \end{definition}

    \begin{remark}
        We deal only with sequences of curves whose image is contained in a compact subset. In this case, as we observed \cite{minguzzi22}, there is no need to specify the uniformity over the compact set as  it is unique when it exists \cite[Thm. 36.19]{willard70}. A uniformity exists when $(X,d)$ is countably generated because  it is then Polish (Prop.\ \ref{prop:cg-Polish}) and hence metrizable.
    \end{remark}

The following construction is of use.
\begin{lemma}\label{lem:curves-from-dense}
     Let $(X,d)$ be a countably-generated Lorentzian metric space, and let $\tau: X\to \mathbb{R}$ be a time function. Let $a,b\in \mathbb{R}$  be such that $a<b$, and assume that $Q$ is a dense subset of $[a,b]$ such that $a,b\in Q$. Let $\zeta: Q\to X$ be such that
    \begin{itemize}
        \item $\tau(\zeta(t)))=t$, $\forall t\in Q,$
        \item  $\zeta(t)\leq \zeta(t')$ whenever $t,t'\in Q$, $a\leq t <t' \leq b$.
    \end{itemize}
    Then $\zeta$ extends to an isocausal curve $\overline{\zeta}:[a,b]\to X$ parametrised by $\tau$. Moreover, if $\zeta$ satisfies the maximality condition (\ref{eqn:maxcurv}) for all $t,t',t''\in Q$, then $\overline{\zeta}$ is maximal.
\end{lemma}

\begin{remark} \label{prvt}
We note that this lemma has a consequence that is interesting by itself in the case $Q=[a,b]$: a monotonous map $[a,b]\to X$ parametrized by a time function (i.e. any map satisfying the two conditions above) is necessarily continuous, and thus an isocausal curve.
\end{remark}

\begin{proof}
    The key idea is extracted from the proof of \cite[Thm. 5.12]{minguzzi22}.
    To begin with, let us verify that $t< t'$ implies the strict inequality $\zeta(t)< \zeta(t')$. Indeed by the first property of $\zeta$,
    \[
    \tau(\zeta(t))=t<t'=\tau(\zeta(t')),
    \]
    which, together with the definition of a time function, excludes the possibility $\zeta(t)=\zeta(t')$.

     To finish the proof we have to show that $\zeta$ is Cauchy continuous\footnote{We recall that a function is called Cauchy continuous if it maps Cauchy sequences to Cauchy sequences.}, thus has a continuous extension. Here we use the fact that $X$ is completely metrisable (Prop.\ \ref{prop:cg-Polish}) and fix some complete metric $\gamma$. Note that the extension will be $\tau$-uniform and isocausal due to closedness of $J$. For that, analogously to \cite{minguzzi22}, we assume the contrary, namely that there is a Cauchy sequence $q_n\in Q$ such that $\zeta(q_n)$ is not Cauchy. Define $r=\lim_{n}q_n$. Since $\zeta(q_n)$ is not Cauchy, we can find $\epsilon>0$ and sequences $n_k$, $m_k$, ${m_k}<{n_k}$, such that
    \[
        \gamma(\zeta(q_{n_k}),\zeta(q_{m_k}))>\epsilon
    \]
    for all $k\in \mathbb{N}$. By passing to subsequences  we can assume that $\lim_{k}\zeta(q_{m_k})=w$, $\lim_{k}\zeta(q_{n_k})=z$ (note that $J(\zeta(a),\zeta(b))$ is compact). We have $\gamma(w,z)\geq \epsilon$, so $w\neq z$.  Moreover, we have $\zeta(q_{m_k})<\zeta(q_{n_k})$, thus, since $J$ is closed, $w\le z$. We conclude that $\tau(w)< \tau(z)$. But by the continuity of $\tau$, $\tau(\zeta(w))=\tau(\zeta(z))=r$, which contradicts the fact that $\tau$ is a time function. So, such subsequences do not exist and $\zeta$ is Cauchy continuous. Therefore, it can be extended to $[a,b]$. Its extension $\overline{\zeta}$ is isocausal by closedness of $J$, and parametrised by $\tau$ by continuity of the latter.

    Finally, we note that the maximality condition (\ref{eqn:maxcurv}) is closed, so, as long as we deal with continuous curves, it holds for all arguments if and only if it holds on a dense subset.
\end{proof}

\begin{theorem}\label{thm:d-uni}
    Let us consider a countably generated Lorentzian metric space $(X,d)$,
    let $K$ be a compact subset of $X$, let $\sigma_n: [a_n,b_n]\to K$ be a sequence of continuous curves, and let $\sigma: [a,b]\to K$ be yet another continuous curve. The sequence $\sigma_n$ uniformly converges to $\sigma$ iff for each $z\in X$ the sequences of functions  $d_z\circ\hat \sigma_n$ and $d^z\circ\hat \sigma_n$   converge uniformly to $d_z\circ\hat \sigma$ and $d^z\circ\hat \sigma$, respectively.
\end{theorem}
From now on we shall be somewhat sloppy in distinguishing between $\sigma$ and $\hat \sigma$, hoping that it shall not cause confusion.
\begin{proof}
    Assume that $\sigma_n$ converges to $\sigma$ uniformly. Note that the restrictions to $K$ of the continuous functions $d^z$ and $d_z$ are uniformly continuous. As a consequence, uniform convergence is preserved in the sense that $d_z\circ \sigma_n$ ($d^z\circ \sigma_n$) uniformly converges to  $d_z\circ \sigma$ (resp.\ $d^z\circ \sigma$) whenever $\sigma_n$ uniformly converges to $\sigma$. This gives one of the implications.

    Conversely, suppose that for each $z\in X$ the sequences of functions  $d_z\circ\sigma_n$ and $d^z\circ\sigma_n$   converge uniformly respectively to $d_z\circ\sigma$ and $d^z\circ\sigma$. Let $\gamma$ be any metric inducing the topology of $X$. Fix $\epsilon>0$.
     By continuity of $\gamma$, for any $t\in [a,b]$ there is a an open set $W_t\subset X\times X$ such that $(\sigma(t),\sigma(t))\in W_t$, and
     \[
     \gamma(x,y)<\epsilon,\,\forall (x,y)\in W_t.
     \]
     We can always assume that $W_t=V_t\times V_t$, where $V_t$ is a finite intersection of sets of the form (\ref{don}). By continuity of $\sigma$, there are $s_t,f_t\in \mathbb{R}$ such that $s_t<t<f_t$, and $\sigma((s_t,f_t)\cap [a,b])\subset V_t$.
     Note that the sets $(s_t,f_t)$  form an open cover of $[a,b]$, so we can find some finite $m$ and $\{t_1,\ldots,t_m\}\in [a,b]$ such that
     \[
     \bigcup_{i=1}^m (s_{t_i},f_{t_i})\supset [a,b].
     \]
     It is always possible to choose $s'_i,f'_i\in (s_{t_i},f_{t_i})\cap [a,b]$ so that
      \[
      \bigcup_{i=1}^m [s_{i}',f'_{i}] =[a,b].
      \]
      The sets $V_{t_i}$ in the most general case can be described as
      \begin{equation}\label{eqn:Vti}
      V_{t_i}=\bigcap_{j=1}^l \left(
      \left(d^{p_j})^{-1}\left(\left(a_{i,j},b_{i,j}\right)\right) \right)
      \cap
      \left(d_{p_j})^{-1}\left(\left(c_{i,j},e_{i,j}\right)\right)\right)
      \right)
      \end{equation}
      for some $l\in \mathbb{N}$, and $\{p_1.\ldots,p_l\}\subset X$ and\footnote{Here we use the same set of points $p_1,\ldots, p_l$ to define the  neighborhoods $V_{t_i}$ for each $i\in\{1,\ldots,m\}$ for notational purposes. Such a presentation is always possible,  since  for every $p\in X$ we have $\left(d_p\right)^{-1}((-\infty,+\infty))=\left(d_p\right)^{-1}((-\infty,+\infty))=X$. Therefore, we can always enlarge the set of points involved in the right hand side of (\ref{eqn:Vti}) without affecting the left hand side. In particular, we can work with a union of sets of points, needed to define each  $V_{t_i}$.} $a_{i,j},b_{i.j},c_{i,j},e_{i,j}\in [-\infty,+\infty]$. Define
      \[
      a'_{i,j}=\inf_{s_i'\leq t \leq f_i'}\big(d^{p_j}(\sigma(t))\big),
      \]

      \[
      b'_{i,j}=\sup_{s_i'\leq t \leq f'_i}\big(d^{p_j}(\sigma(t))\big),
      \]

      \[
      c'_{i,j}=\inf_{s_i'\leq t \leq f_i'}\big(d_{p_j}(\sigma(t))\big),
      \]

      \[
      e'_{i,j}=\sup_{s_i'\leq t \leq f_i'}\big(d_{p_j}(\sigma(t))\big).
      \]
      As we are working with continuous functions on compact regions, the infima and suprema are actually maximums and minumums, so $a'_{i,j},b'_{i,j} \in (a_{i,j},b_{i,j})$ and $c'_{i,j},e'_{i,j} \in (c_{i,j},e_{i,j})$ for all $i,j$ for which it makes sense. Then there is $\delta>0$ such that
      $a'_{i,j}-\delta,b'_{i,j}+\delta \in (a_{i,j},b_{i,j})$ and $c'_{i,j}-\delta,e'_{i,j}+\delta \in (c_{i,j},e_{i,j})$.

      To conclude the proof, let $N\in \mathbb{N}$ be such that $\forall t, \forall n\geq N,\ \forall i\in\{1,\ldots,l\}$
      \[
      |d^{p_i}\circ \sigma_n(t)-d^{p_i}\circ \sigma(t)|<\epsilon,\quad \mathrm{and}\quad |d_{p_i}\circ \sigma_n(t)-d_{p_i}\circ \sigma(t)|<\delta .
      \]
      Existence of such a number $N$ follows from uniform convergence of functions $d^{p_i}\circ\sigma_n\to d^{p_i}\circ\sigma$ and $d_{p_i}\circ\sigma_n\to d_{p_i}\circ\sigma$. Then for $t\in [s_i',f_i']$ we have $\sigma_n(t)\in V_{t_i}$ whenever $n\geq N$. This implies
      \[
      \gamma(\sigma(t),\sigma_n(t))<\epsilon, \ \forall t,\ \forall n\geq N
      \]
      Existence of such a number $N$ for every $\epsilon>0$ is precisely the uniform convergence.
\end{proof}
\begin{corollary}\label{cor:isocaus-uni}
    Let $Q\subset \mathbb{R}$ be a dense set.  A sequence of isocausal curves $\sigma_n$ uniformly converges to an isocausal curve $\sigma$ iff it converges pointwise on $Q$.
\end{corollary}
In particular, by setting $Q=\mathbb{R}$, we get that the pointwise and uniform convergence for isocausal curves coincide. Note that the requirement that $\sigma$ is isocausal curve is essential, because there are sequences of strictly decreasing functions $[0,1]\to [0,1]$ that converge to zero pointwise on $(0,1]$ but not at $0$, which gives a counterxample for $X=\mathbb{R}$ with the natural Lorentzian metric space structure.
\begin{proof}
    Obviously, we need to check only one of the implications. Let $\sigma_n: [a_n,b_n]\to X$ be a sequence of isocausal curves converging to $\sigma: [a,b]\to X$ pointwise on $Q$. Consider $z\in X$. Consider the continuous non-decreasing function $d_z\circ \hat{\sigma}: \mathbb{R}\to [0,+\infty)$. Since $a_n\to a$ and $b_n\to b$, without any loss of generality we may assume that $a_n,b_n\in [a-1,b+1]$ for all $n\in\mathbb{N}$. By continuity of $d_z\circ \sigma$, we can find some $m\in\mathbb{N}$ and $s_1,\ldots,s_m\in Q$ such that
    \[
    s_1<a-1<s_2<\cdots < b+1 < s_n
    \]
    in such a way that $d_z\circ \sigma(s_{i+1})-d_z\circ\sigma(s_{i})<\epsilon$.
    Let $N\in\mathbb{N}$ be such that for every $n\geq N$ and $i\in \{1,\ldots,l\}$,
    \[
    |d_z(\sigma_n(s_i))-d_z(\sigma(s_i))|<\epsilon.
    \]
    Such a number $N$ exists by the assumed pointwise convergence.

    Then for $t\in [s_{i},s_{i+1}]$ and $n\geq N$ we have
    \begin{align*}
    |d_z\circ \sigma_n(t)-d_z\circ \sigma(t)| & \leq     |d_z\circ \sigma_n(t)-d_z\circ \sigma_n(s_i)|+    |d_z\circ \sigma_n(s_i)-d_z\circ \sigma(s_i)| \\
    &\quad +    |d_z\circ \sigma(s_i)-d_z\circ \sigma(t)|\\
    &\leq
        |d_z\circ \sigma_n(s_{i+1})-d_z\circ \sigma_n(s_i)|+|d_z\circ \sigma_n(s_i)-d_z\circ \sigma(s_i)| \\
        &\quad +    |d_z\circ \sigma(s_i)-d_z\circ \sigma(s_{i+1})| \\
        &\leq
     2|d_z\circ \sigma(s_{i+1})-d_z\circ \sigma(s_i)|+3\epsilon\leq 5 \epsilon.
    \end{align*}
    Here we have used the fact that
    \[
    0<d_z\circ \sigma(s_i)-d_z\circ \sigma(t)\leq d_z\circ \sigma(s_i)-d_z\circ \sigma(s_{i+1}), \ \forall t\in [s_i,s_{i+1}],
    \]
    and similarly for $d_z\circ \sigma_n$, because those functions are non-decreasing.
     This proves that $d_z\circ \sigma_n$ converges uniformly to $d_z\circ \sigma$. Replacing $d_z$ with $d^z$ and `non-decreasing' with `non-increasing' we get that $d^z\circ \sigma_n$ converges uniformly to $d^z\circ \sigma$. So, the claim follows from Theorem \ref{thm:d-uni}.
\end{proof}
\begin{theorem}[Limit curve theorem]\label{thm:lim-curv}
    Let $(X, d)$ be a Lorentzian metric space without chronological boundary. Let $\sigma_n: [a_n, b_n] \to X$ be a sequence of isocausal curves parametrized with respect to a given time function $\tau$, $\tau(\sigma_n(t))=t$. Suppose that
    \[
        \lim_{n\to \infty }\sigma_n(a_n)=x, \          \lim_{n\to \infty }\sigma_n(b_n)=y,
    \]
    and $x\neq y$. Then there exists a $\tau$-uniform isocausal curve $\sigma : [a, b] \to X$ and a subsequence $\{\sigma_{n_k}\}_k$ that converges uniformly to $\sigma$. If the curves $\sigma_n$ are maximizing (i.e. Equation (\ref{eqn:maxcurv}) holds) then so is $\sigma$.
\end{theorem}
\begin{remark}
    It will be shown in the very beginning of the proof that under these assumptions the curves stay in a compact set, so the uniform convergence is defined unambiguously.
\end{remark}
\begin{proof}
    Let $p,q\in X$ be such that $p\ll x$ and $y\ll q$. Then for large enough $n$ we have $\sigma_n(a_n)\gg p$ and $\sigma_n(b_n)\ll q$. By the isocausality of the curves,  it follows that for large enough $n$, $\sigma_n$ stays in the compact set $\overline{I(p,q)}$. By cutting the sequence if necessary, we can assume that this happens for all $n\in\mathbb{N}$.
    This allows us to adopt most of the technique of the proof of \cite[Theorem 5.12]{minguzzi22}.

    Note that by the continuity of $\tau$,
    \[
    \lim_{n\to \infty}\tau(a_n)=\tau(x)=:a,
    \]
    \[
    \lim_{n\to \infty}\tau(b_n)=\tau(y)=:b.
    \]
    By passing to subsequences and the diagonal argument, we can ensure that for every $t\in \mathbb{Q}\cap [a,b]$ the sequence $\sigma_n(t)$ has a limit which we denote with $\sigma(t)$. It is easy to see that $\tau(\sigma(t))=t$. Let us show that for $t,t'\in \mathbb{Q}\cap [a,b]$, $t<t'$ implies $\sigma(t)\leq \sigma(t')$. Assume that this is not so, i.e.\ there is $z\in X$ such that $d(z,\sigma(t))>d(z,\sigma(t'))$ or $d(\sigma(t),z)<d(\sigma(t'),z)$.


    For large enough $n\in\mathbb{N}$, $d(\sigma_n({t}),z)<d(\sigma_n({t}'),z)$ or $d(z,\sigma_n({t}))<d(z,\sigma_n({t}'))$, which is in contradiction with isocausality of $\sigma_n$. We conclude that such a point $z$ can not exist, and thus $\sigma$ satisfies requirements of Lemma \ref{lem:curves-from-dense} with $Q=\mathbb{Q}\cap [a,b]$, i.e.\ it extends to an isocausal curve $\sigma: [a,b]\to X$.

    The convergence $\sigma_n\to \sigma$ (after passing to the subsequences according to the procedure above) is uniform by Corollary \ref{cor:isocaus-uni} for $Q=\mathbb{Q}$.
\end{proof}

Being a (pre)length space is a local property, so one would expect that it can be studied by looking at the bounded regions. This is indeed so, although establishing this fact requires some work.
\begin{theorem}\label{thm:length-local}
    Let $(X,d)$ be a countably-generated Lorentzian metric space. Let $\mathscr{F}$ be a family of compact subsets of $X$ such that for every $m\in\mathbb{N}$ and $p^1,\ldots,p^m \in X$ there is $F\in \mathscr{F}$ such that $I(p^1,\ldots,p^m)\subset F$.

    Then the following statements are equivalent:
    \begin{enumerate}
        \item $X$ is a prelength (resp.\ length) space;
        \item For any $p,q,x,y \in X$ and $F\in \mathscr{F}$ such that $p\ll x \ll y \ll q$ and $I(p,q)\subset F$, there is an isocausal (resp.\ maximal isocausal)  curve connecting $[x]_F$ to $[y]_F$ in $BF$.
    \end{enumerate}
\end{theorem}
Obvious examples of such families are the compact sets of the form $\overline{I(A)}$ and $\overline{J(A)}$ with $A\subset \mathscr{G}$ running over all finite subset of a generating set $\mathscr{G}\subset X$. In Section \ref{sec:GH} we will see another example.
\begin{remark}\label{rmk:thm:length-local}
    The first statement in Theorem \ref{thm:length-local} implies the second for any family $\mathfrak{F}$ of compact subsets of $X$  (e.g.\ a family consisting of just one compact subset). This will be clear from the proof, but formally follows from the given formulation, because any family can be enlarged to one, satisfying the requirements of the theorem.
\end{remark}
\begin{proof}
    \begin{enumerate}
        \item  Assume that $(X,d)$ is a prelength space. Take $x,y,p,q\in X$ and $F\in \mathscr{F}$ so that $p\ll x \ll y \ll q$ and $I(p,q)\subset F$. Let $\zeta: [0,1]\to X$ be an isocausal curve connecting $x$ with $y$ in $X$. We want to project it to a curve connecting $[x]_{F}$ with $[y]_{F}$ in $BF$. The main problem here is that the naive projection of $\zeta$ may be constant at some intervals of $[0,1]$ (and thus fails to be isocausal), so we have to find the right parametrisation.

         Let us take a countable set  $\mathscr{S}_F\subset \mathring{F}$ such that its projection to $BF$ is dense, cf.\ Cor.\ \ref{crl:comp-dense}. Let us consider  a sequence $\{s_k\}_{k\in\mathbb{N}}$  whose image is $\mathscr{S}_F$.

        We construct a function $\tau_F: X\rightarrow [-1,1]$ as in Lemma \ref{lem:time-func-seq}
        \[
        \tau_F(x)=\alpha \sum_{j=1}^{\infty} \frac{1}{2^{j}} \left(\frac{d(s_{j},x)}{1+d(s_{j},x)}-\frac{d(x,s_{j})}{1+d(x,s_{j})}\right) + \beta.
        \]
        Here $\alpha>0$ and $\beta$ can be always chosen so that $\tau_F(x)=0$ and $\tau_F(y)=1$, as noted in Remark \ref{rmk:time-func-resc}.
        Clearly, $\tau_F\vert_{\mathring{F}}$ passes to the quotient $BF$ and becomes a time function on $BF$ which we denote in the same way.

        For every $t\in [0,1]$, by continuity, we can always choose $s$ and hence $\zeta(s)$ on the curve $\zeta$ such that $\tau_F(\zeta(s))=t$.

        Note that although $\zeta(s)$ may be non-unique, its equivalence class in  $BF$ always is, because $\tau_F$ is a time function on $BF$. More explicitly, if $\zeta(s)$ and $\zeta(s')$ satisfy $\tau_F(\zeta(s))=\tau_F(\zeta(s'))=t$, for $s\ne s'$, then without loss of generality we can assume $s<s'$ which by isocausality implies $\zeta(s)<\zeta(s')$ and hence, since $J\cap (F\times F)\subset J_F$, $[\zeta(s)]_F\le [\zeta(s')]_F$, thus since $\tau_F$ is a time function on $BF$ and $\tau_F([\zeta(s)]_F)=\tau_F([\zeta(s')]_F)=t$ it must be $[\zeta(s)]_F=[\zeta(s')]_F$.

        So, there is a well defined map $t\mapsto [\zeta(s(t))]_{F}$ such that (here $s(t)$ might be regarded as a multivalued map), by the same argument of the previous paragraph, satisfies $t<t' \Rightarrow [\zeta(s(t))]_F<[\zeta(s(t'))]_F$. Moreover, $[\zeta(s(0))]_{F}=[x]_F$ and $ [\zeta(s(1))]_{F}=[y]_F$.

The continuity of the map $t\mapsto [\zeta(s(t))]_{F}$  follows from Remark \ref{prvt}, thus, it is an isocausal curve.

        \item Conversely, assume that $(X,d)$ has the property described in the second statement.

        By fixing a countable generating set and its enumeration we can always assume that we have a sequenced Lorentzian metric space $(X,d,\{p^n\}_{n\in\mathbb{N}})$. For each $n\in \mathbb{N}$ let $F_n\in\mathscr{F}$ be such that $I(p^1,\ldots,p^m)\subset F_n$.

         Let $x,y\in X$,  $x\ll y$. There is some $m\in \mathbb{N}$ such that $x,y\in I(p^1,\ldots,p^m)$.  The second statement implies that for every $n\geq m$ the points $[x]_{F^n}$ and $[y]_{F^n}$ are connected by an isocausal curve in $BF^n$, which we denote with $\sigma^n$. We intend to prove that this implies the existence of an isocausal curve connecting $x$ with $y$.

        The rest of the proof follows the approach used several times in \cite{minguzzi22}. To start with, let us choose a convenient parametrisation. Namely, we apply Lemma \ref{lem:time-func-seq} to build time functions $\tau^n$ on $BF^n$ for every $n\in \mathbb{N}$ in such a way that for any $p\in X$
        \[
    \lim_{n\rightarrow\infty} \tau^n([p]_{F^n})=\tau(p),
        \]
       where $\tau$ is a time function on $X$. By Remark \ref{rmk:time-func-resc} we can assume that $\tau(x)=0$ and $\tau(y)=1$. Assume that each curve $\sigma^n$ is parametrized by the time function $\tau^n$, i.e.\ $\tau^n(\sigma^n(t))=t$ for $t\in (a_n,b_n)$, where  $a_n=\tau^n(x)$, $b_n=\tau^n(y)$, see \cite[Prop.\ 5.8]{minguzzi22}.

       By Remark \ref{rmk:time-func-resc}
        \begin{equation}
        |\tau(\sigma^n(t))-t|\leq \alpha 2^{-n}
        \label{eqn:inhlem:tausigma}
        \end{equation}

        for $t\in (a_n,b_n)$, where $\alpha\in (0,+\infty)$ is the scaling constant. Similarly, $|a_n|<\alpha 2^{-n}$ and $|1-b_n|< \alpha 2^{-n}$.
        By enlarging $m$ if necessary, we can always achieve $\alpha 2^{-m}<1$.

        For any $n\geq m$ and $t\in (\alpha 2^{-n},1-\alpha 2^{-n})\cap \mathbb{Q}$ take arbitrary representative $\zeta_n(t)\in \sigma^n(t)$. For $t\in [0,\alpha 2^{-n}]\cap \mathbb{Q}$ set $\zeta_n(t)=x$, and  for  $t\in [1-\alpha 2^{-n},1]\cap \mathbb{Q}$ set $\zeta_n(t)=y$.

        Note that for any $n\geq m$ and any $t\in [0,1]\cap \mathbb{Q}$ we have
        \[\zeta_n(t)\in I(p^1,\ldots,p^m)\subset F^m.\]
        So, for fixed $t$ we can find a convergent subsequence of $\zeta_n(t)$, and by the diagonal argument make it  convergent simultaneously for all $t\in [0,1]\cap \mathbb{Q}$. Let $\zeta(t)$ stand for the limit of the chosen subsequence of $\zeta_n(t)$. By passing to the limit in (\ref{eqn:inhlem:tausigma}) (taking into account that for every $n\in\mathbb{N}$, $\tau(\zeta_n(0))=\tau(x)=0$, $\tau(\zeta_n(1))=\tau(y)=1$,
        we get
        \begin{equation}\label{eqn:inhlem:tauzeta}
        \tau(\zeta(t))=t, \ \forall t\in [0,1]\cap \mathbb{Q}.
        \end{equation}

        Let us verify that $t<t'$, implies $\zeta(t)\leq \zeta(t')$. Assume the contrary. Then for some $t,t'\in [0,1]$ and $z\in X$ we have $t<t'$, but
        $d(z,\zeta(t))>d(z,\zeta(t'))$, or $d(\zeta(t),z)<d(\zeta(t),z))$. Take $N$ such that $z\in I(p^1,\ldots,p^N)$. Then $z\in F_n$ for all $n\geq N$. By construction $[\zeta_n(t)]_{F_n}\leq [\zeta_n(t')]_{F_n}$, so for every $n\geq N$ we have
        \[
        d(z,\zeta_n(t))\leq d(z,\zeta_n(t')),\quad         d(\zeta_n(t),z)\geq d(\zeta_n(t'),z).
        \]
        By passing to subsequences and taking the limit we get a contradiction with the assumption made.


        Together with (\ref{eqn:inhlem:tauzeta}) this makes $\zeta$ admissible for Lemma \ref{lem:curves-from-dense} with $Q=[a,b]\cap \mathbb{Q}$. Thus $\zeta$ extends to a continuous isocausal curve $[0,1]\to X$ such that $\zeta(0)=x$, $\zeta(1)=y$.

       \item It is easy to see that in both arguments if we start from a maximal curve, we get a maximal curve again. So, the version in the brackets follows easily.
    \end{enumerate}
\end{proof}

\subsection{The length functional}
In this subsection we collect results on the length functional and the related equivalent definition of  length space. For this material, the general case does not differ significantly  from the bounded one, so most of the statements and proofs are very close to those in \cite{minguzzi22}. Yet, these results are important as they allow one to compare our Lorentzian length spaces with alternative definitions existing in the literature.
\begin{definition}
Let $(X,d)$ be a Lorentzian metric space, and let $\sigma: [0,1]\to X$ be an  isocausal curve. Its  \emph{Lorentzian length} is
\[
L(\sigma):=\textrm{inf} \sum_{i=0}^{k-1} d(\sigma(t_i),\sigma(t_{i+1}))
\]
where the infimum is over the set of all partitions $\{t_0,t_1,t_2, \cdots, t_k\}$, $k\in \mathbb{N}$, $t_i\in [0,1]$, $t_i<t_{i+1}$, $t_0=0$, $t_k=1$.
\end{definition}
By the reverse triangle inequality, $L(\sigma)\le d(\sigma(0),\sigma(1))$.
\begin{proposition}
Let $\sigma: [0,1]\to X$ be an isocausal curve with  endpoints $x$ and $y$.
We have $L(\sigma)=d(x,y)$ iff $\sigma$ is maximal.
\end{proposition}
\begin{proof}
    The proof of \cite[Prop. 6.2]{minguzzi22} applies without any change.
\end{proof}

\begin{theorem}[Upper semi-continuity of the length functional] $\empty$ \\
Let $(X,d)$ be a Lorentzian metric space. Let $\sigma_n: [0,1]\to X$ and $\sigma: [0,1]\to X$ be  isocausal curves and suppose that $\sigma_n\to \sigma$ pointwise. Then
\[
\limsup L(\sigma_n)\le L(\sigma).
\]
\end{theorem}
\begin{proof}
    The proof of \cite[Thm. 6.3]{minguzzi22} applies without any change.
\end{proof}

Let $(X,d)$ be a Lorentzian prelength space. Then for $x\ll y$ we can define
\[
\check{d}(x,y)=\sup_{\sigma(0)=x,\sigma(1)=y} L(\sigma),
\]
where the supremum goes over isocausal curves. If $x\ll y$ fails, we define $\check{d}(x,y)=0$. Then we have the following.
\begin{theorem}\label{thm:alt-lenspace}
Let $(X,d)$ be a countably generated Lorentzian prelength space. The following conditions are equivalent
\begin{itemize}
\item[(i)] $ \check d=d$,
\item[(ii)] $(X,d)$ is a length space.
\end{itemize}
\end{theorem}
\begin{proof}
 The proof of \cite[Thm. 6.4]{minguzzi22} applies with the Limit Curve Theorem \cite[Thm.\ 5.12]{minguzzi22} replaced by its unbounded version Theorem \ref{thm:lim-curv}.
\end{proof}

For further use, it is convenient to introduce the reduced causal relation $\check{J}$.
\begin{definition}
Let $(X,d)$ be a Lorentzian prelength space. The \emph{restricted causal relation} $\check J\subset J$ is defined as the set of pairs of points connected by an isocausal curve
or coincident.
\end{definition}

Note that with this definition the prelength space condition reads $I\subset \check J$.

\begin{proposition} \label{pby}
Let $(X,d)$ be a countably generated Lorentzian prelength space. Then for every $(x,y)\in \bar I$, $x \ne y$,  there exists a  isocausal curve
$\sigma: [a,b]\to X$ connecting $x$ to $y$.
\end{proposition}
\begin{proof}
 The proof of \cite[Prop. 5.14]{minguzzi22} applies with Limit Curve Theorem \cite[Thm.\ 5.12]{minguzzi22} replaced by its unbounded version Theorem \ref{thm:lim-curv}.

\end{proof}

\begin{corollary}
Let $(X,d)$ be a countably generated Lorentzian prelength space. The restricted causal relation $\check J$  is reflexive, transitive, antisymmetric and closed.  Moreover, $I\subset \check{J}$ and $I\circ \check{J}\cup \check{J}\circ I \subset I$.
\end{corollary}

\begin{corollary}\label{crl:l-as-sup}
    Let $(X,d)$ be a countably generated Lorentzian length space and $x,y\in X$. Define
    \[
    l(x,y)=\sup_{\sigma(0)=x,\sigma(1)=y} L(\sigma),
    \]
    where we set $\sup \varnothing=-\infty$. Then $l(x,y)=d(x,y)$ if $(x,y)\in \check{J}$, and $l(x,y)=-\infty$ otherwise.
\end{corollary}
\begin{proof}
    If $x\ll y$, then the statement follows from Theorem \ref{thm:alt-lenspace}. If $x\ll y$ fails, then any isocausal curve connecting $x$ with $y$ necessarily has zero length. So, in the case $(x,y)\in \check{J}$, when such a curve exists, $l(x,y)=0$, and $l(x,y)=-\infty$ otherwise.
\end{proof}

\subsection{Approaches using the time-separation function }\label{ssec:braun}

In the application of rough spacetime geometry to optimal transport it has proved convenient to adopt a modified two-point function $l$ of codomain $\{-\infty\}\cup [0,+\infty)$ which we might call {\em time-separation}, in order to distinguish it from $d$, and  for consistency with recent literature. It contains more information than the Lorentzian distance $d$, as we shall clarify in a moment.

In a recent work  by Braun and McCann \cite{braun23b} the authors reprhased some of our results on bounded Lorentzian metric spaces by means of the function $l$. The translation between the two settings might be not immediately obvious to readers not acquainted to both formalisms, so it is investigated in some detail here.  Compared to Kunzinger-S\"amann's approach \cite{kunzinger18}, to which previous work by these authors relied, our approach has some important advantages as the absence of any auxiliary metric.
 In the KS-approach an additional metric was used to parametrize curves, prove the limit curve theorem, define causality properties such as non-total imprisonment and global hyperbolicity, in analogy to the smooth setting, and  thus its presence can be regarded as a distinct feature of KS's approach. Instead, our approach  relies just on function $d$, in fact stability under Gromov-Hausdorff convergence was used to identify the best definition for a (pre)length space based on function $d$.

Braun and McCann \cite{braun23b} studied a non-compact framework with assumptions conceived  so as to obtain, in compact regions, bounded Lorentzian metric spaces, and hence so as to import our results on topology (e.g.\ Polish property), limit curve theorems, absence of additional metric, etc.\ to their non-compact setting (see \cite[Appendix B]{braun23b}). They were not concerned with Gromov-Hausdorff convergence, so they included in their definitions properties that can be shown not to be preserved under GH-limits. Those properties were meant to facilitate  additional results. For instance, their assumption that through every point passes an isochronal curve implies that  $J=\bar I$ by \cite[Remark 5.2]{minguzzi22}.


Let us compare the classes of   spacetimes  considered in \cite{braun23b}  with that given by Definition \ref{cg-lms}. Since they are based on the notion of bounded Lorentzian metric space \cite{minguzzi22} over compact subsets they should not be too different. This expectation will be confirmed, cf.\ Cor.\ \ref{kxuy}.

 We shall use repeatedly the definition of positive part of a number belonging to $\{-\infty\}\cup [0,+\infty)$. This is given by $(x)_+:=x$ if $x>0$ and $(x)_+:=0$ otherwise. It is also denoted $x_+$ for brevity. This function is non-decreasing, that is, for $x,y \in  \{-\infty\}\cup [0,+\infty)$, we have
$x\le y \Rightarrow x_+\le y_+$.

In Definition \ref{cg-lms} we have a set $X$ and a continuous function $d: X\times X \to [0,+\infty)$. In \cite{braun23b} the so-called time-separation function $l$ appears. Unlike the Lorentzian distance $d$, the time-separation function $l$ is allowed to take also the infinite value $- \infty$. Its topological properties are weaker than those of $d$: the function $l$ itself is only upper semi-continuous, while the lower semi-continuity is required for its positive part $l_+$ only.
\begin{remark}\label{rmk:lfinite}
    The most general setting of \cite{braun23b} allows also $l_{+}(x,y)=+\infty$ for some $x,y \in X$.  However, other assumptions made in \cite{braun23b}, namely causality and the length space property, exclude this possibility \cite[Cor. 2.17]{braun23b}. For this reason, and in order to simplify the exposition, we assume that $l$ is valued in $\{-\infty\}\cup [0,+\infty)$ from the very beginning.
\end{remark}

\begin{lemma}\label{lem:braun-d}
    Let $X$ be a topological space, $l: X\times X \rightarrow \{-\infty\}\cup [0,+\infty)$ an arbitrary function.
    Then the following conditions are equivalent:
    \begin{enumerate}
        \item $l$ is upper continuous, and $l_{+}$ is lower semi-continuous;
        \item $l_+$ is continuous and the set
        \[
        K_{l}:=\{(x,y)\in X\times X| l(x,y)\geq 0\}
        \]
        is closed.
    \end{enumerate}
\end{lemma}

\begin{proof}
  Let us assume that $l$ is upper semi-continuous. The sets
  \[
  \tilde{R}_c=\{(x,y)\in X\times X| l(x,y)< c\}
  \]
  are open for all $c\in \mathbb{R}$. It follows that $K_l=X\times X \setminus \tilde{R}_0$ is closed. Moreover, it implies that for any $c$ the sets
  \[
  R_c=\{(x,y)\in X\times X|l_{+}(x,y)< c\}
  \]
  are open. Indeed, if $c\leq 0$, then $R_c=\varnothing$, otherwise $R_c=\tilde{R}_c$. We conclude that the upper semi-continuity of $l$ together with the lower semicontinuity of $l_{+}$ implies continuity of $l_{+}$ and closedness of $K_l$.

  Now let us assume that $l_{+}$ is continuous  and $K_l$ is closed. We have to show only that $l$ is upper semi-continuous, i.e. that the sets $\tilde{R}_c$ defined above are open for all values of $c$. If $c\leq 0$, then $\tilde{R}_c=X\times X \setminus K_l$, and thus is open. If $c>0$, then $\tilde{R}_c=R_c$, which is open by continuity of $l_{+}$.
\end{proof}

\begin{proposition}
Let $X$ be a topological space endowed with a continuous function  $d: X\times X\to [0,+\infty)$ and a closed relation $K$ such that $d(x,y)>0$ implies $(x,y)\in K$ (i.e.\ $I\subset K$). Then there is a unique function $l: X \times X\to \{-\infty\} \cup [0,+\infty)$ such that $l_+=d$ and $K_l=K$. It is given by
 \begin{align*}
 l(x,y)&=d(x,y),& &\forall (x,y)\in K, \\
 l(x,y)&=-\infty,& &\forall (x,y)\in X\times X \setminus K.
\end{align*}
\end{proposition}

\begin{proof}
It is immediate from the definitions.
\end{proof}


    So, we conclude that there is a one-to-one correspondence between functions $l$ satisfying the properties listed in the first statement of Lemma \ref{lem:braun-d} and pairs $(d,K)$, where $d$ is continuous function on $X\times X$ and $K$ is a closed relation such that $I\subset K$.

    In the rest of this subsection we will use these objects interchangeably.


We are ready to show that the reverse triangle inequality assumption imposed in \cite{braun23b} is equivalent to the reverse triangle inequality of Definition \ref{cg-lms} plus some natural properties for the relation $K$.
\begin{lemma}\label{lem-BMC-triangle}
    The following conditions are equivalent:
    \begin{enumerate}
        \item $l$ satisfies the extended reverse triangle inequality, i.e. $l(x,z)\geq  l(x,y)+l(y,z)$ for any $x,y,z\in X$;
        \item $d$ satisfies the (restricted) reverse triangle inequality, the relation $K$ is transitive, and $K\subset J$;
    \end{enumerate}
\end{lemma}
\begin{proof}
     Let us assume that $l$ satisfies the extended reverse triangle inequality. Then $d$ satisfies the usual restricted one as $d$ and $l$ coincide when they are positive. Moreover, if $(x,y)\in K$ and $(y,z)\in K$, then $l(x,z)\geq l(x,y)+l(y,z)\geq 0$, so $(x,z)\in K$, i.e.\ $K$ is transitive. Finally, to establish $K\subset J$ we have to show that for every $(x,y)\in K$ and $z\in X$ we have $d(x,z)\ge d(y,z)$ and $d(z,y)\ge d(z,x)$. In fact, as $l(x,y)\ge 0$
    \[
   l(x,z)\geq l(x,y)+l(y,z)\geq l(y,z).
    \]
    and
    \[
    l(z,y)\geq l(z,x)+l(x,y)\geq l(z,x).
    \]
    This implies that $d(x,z)\geq d(y,z)$ and $d(z,y)\geq d(z,x)$ as desired.

    Conversely, let us assume that all the conditions of the second statement are satisfied and let us show that
    \[
    l(x,z)\ge l(x,y)+l(y,z), \,\forall x,y,z\in X.
    \]
    If either $l(x,y)$ or $l(y,z)$ is equal to $-\infty$, then the inequality is trivial. So, we assume that $(x,y)\in K$ and $(y,z)\in K$. By transitivity of $K$,  $(x,z)\in K$, so we need only to prove that
    \[
    d(x,z)\ge d(x,y)+ d(y,z)
    \]
    whenever $(x,y),(y,z)\in K$. If $d(x,y)>0$ and $d(y,z)>0$, the desired inequality follows from the (restricted) reverse triangular inequality, so the only non-trivial case is for $d(x,y)=0$ or $d(y,z)=0$. But as $(x,y),(y,z)\in K \subset J$, we have the inequalities
    \[
    d(x,z) \ge d(y,z),\,\quad d(x,z) \ge d(x,y),
    \]
    which are precisely the extended triangle inequalities in the cases $d(x,y)=0$ and $d(y,z)=0$, respectively.
\end{proof}

\begin{lemma}\label{lem:braun-caus}
    The following conditions are equivalent:
    \begin{enumerate}
        \item The relation $K_l$ is antisymmetric;
        \item The function $l$ satisfies: if for some $x,y\in X$
        \[\min(l(x,y),l(y,x))>-\infty,\]
        then $x=y$.
    \end{enumerate}
     Further, the following conditions are equivalent:
    \begin{enumerate}
        \item The relation $K_l$ is reflexive;
        \item The function $l$ satisfies:  for every $x$,  $l(x,x)\ge 0$.
    \end{enumerate}

Joining the two equivalences: $K_l$ is antisymmetric and reflexive iff     the \emph{causality condition} of \cite[Def.\ 2.1]{braun23b} holds.
\end{lemma}

\begin{proof}
    Straightforward consequence of definition of $K_l$.
\end{proof}
\begin{remark}
    Note that if the equivalent conditions of Lemma \ref{lem-BMC-triangle} are satisfied and $d$ distinguishes points, then $K$ is automatically anti-symmetric (because $J$ is, and $K\subset J$).
\end{remark}

From Lemmas \ref{lem:braun-d}-\ref{lem:braun-caus} we conclude that a topological space endowed with a signed time separation function $l$ in the sense of \cite[Def. 2.1]{braun23b} is exactly the same as a space endowed with a function $d$ satisfying all the properties of Definition \ref{cg-lms} except for the compactness assumption of property (ii) and the distinguishing property (iii), and endowed with an additional intermediate closed causal relation $K$.  We recall from Subsection \ref{ssec:glob-hype}, that the compactness condition of property (ii) is essentially equivalent to global hyperbolicity. Without it the class of the spaces is too general to establish any strong common properties. In fact, in the major part of \cite{braun23b} stronger conditions have been imposed. In particular, the so-called  `First standing assumption' \cite[Hypothesis 2.9]{braun23b} includes, among other conditions, compactness of the sets $K_l(p,q)$ and absence of the chronological boundary, which means that all such spaces can be considered within the setting of this paper. In fact we have the following result.

\begin{corollary} \label{kxuy}
    There is a one-to-one correspondence between the following objects:
    \begin{enumerate}
        \item Lorentzian length spaces $(X,d)$ such that any point lies on an isochronal curve;
        \item Length metric spacetimes in the sense of \cite[Def. 2.7]{braun23b} satisfying the `first standing assumption' \cite[Hypothesis 2.9]{braun23b}.
    \end{enumerate}
\end{corollary}

Note that the fact that every point lies on an isochronal curve implies $p\in \overline{I^\pm(p)}$ which implies $\bar I= J$ by \cite[Remark 5.2]{minguzzi22}, so $\bar I=K=J$ and there is no gap between $K$ and $J$ in the context of \cite{braun23b} under their first standing assumption. By Thm.\ \ref{thm:aleksandrov} this assumption also implies that the Lorentzian metric space topology coincides with the Alexandrov topology (these results, \cite[Remark 5.2]{minguzzi22} and  Thm.\ \ref{thm:aleksandrov}, on the consequence of $p\in \overline{I^\pm(p)}$ constitute what might be called the no-gaps theorem).

\begin{proof}
Recall that in \cite{braun23b} one works with a first countable space $X$ endowed with a time-separation function $l$, which is allowed to take values in $\{-\infty\}\cup [0,+\infty]$. The latter is assumed to satisfy the extended reverse triangle inequality, be upper semi-continuous, have a lower semi-continuous positive part and satisfy causality conditions (see the comment after Lemma \ref{lem:braun-caus}). By \cite[Def.\ 2.7]{braun23b} the space is called a length metric space if  (a) every point  is a midpoint of a timelike curve, and (b) the time separation function is the supremum of the lengths of  isocausal curves connecting the given two points. Finally, The First Standing Assumption in addition to that imposes compactness of the causal diamonds defined by the causal relation $K_l$.

2 $\Rightarrow$ 1. Let us first show that any  length metric spacetime $(X,l)$ satisfying the first standing assumption is a Lorentzian metric space. By Remark \ref{rmk:lfinite}, under these assumptions $l(x,y)<+\infty$ for all $x,y\in X$, so we can apply Lemma \ref{lem:braun-d} to define Lorentzian distance function $d$ from the time-separation $l$. By Lemma \ref{lem-BMC-triangle}, $d=l_+$ satisfies the reverse triangle inequality. The distinguishing property (iii) follows from \cite[Corollary 2.21]{braun23b}. So, $(X,d)$ satisfies Definition \ref{cg-lms} except for property (ii).  Taking into account that property (a) of \cite[Def.\ 2.7]{braun23b}  implies absence of chronological boundary, and $d$ is continuous in the original topology of $X$ (Lemma \ref{lem:braun-d}), we apply Theorem \ref{thm:lms-via-kdiam} to $K=K_l$ and establish that The First Standing Assumptions implies property (ii).
We need only to show that $(X,d)$ is a length space. It is a prelength space in our sense, because for any $x\ll y$, $d(x,y)$ is positive and a supremum of the lengths of isocausal curves connecting $x$ with $y$, thus at least one such curve exists. It is a Lorentzian length space by Theorem \ref{thm:alt-lenspace}.

Before going to the converse direction, let us check if any information is lost when we pass from $(X,l)$ to $(X,d)$. Potentially, the additional ingredients of the former are the original topology of $X$ and the relation $K_l$. As a byproduct of the argument above, the original topology of $X$ is equal to its unique Lorentzian metric space topology. By \cite[Remark 5.2 ]{minguzzi22}, the existence of an isochronal curve passing through every point implies $J=\overline{I}$. Thus the only possibility is $K_l=J$.

  1$\Rightarrow$ 2. We need only to prove that if $(X,d)$ is a Lorentzian length space with any point lying on an isochronal curve, then $(X,l)$ with $l$ being such that $l_+=d$ and $K_l=J$ is a length metric spacetime satisfying The First Standing Assumption.

To start with, we have to prove that the topology of $X$ is first countable. By \cite[Remark 2.8]{braun23b} the Alexandrov topology (called there the chronological topology) is first countable in this setting. Thus, it is enough to show that the Alexandrov topology is equal to the Lorentzian metric space topology of $X$. For that let $U\subset X$ be an open set and $x\in U$. By assumptions, there is an isochronal curve $\gamma:[0,1]\to X$ such that $\gamma(1/2)=x$. Then there are $a,b\in \gamma^{-1}(U)$ such that $a<1/2<b$, thus $\gamma(a)\in I^-(x)\cap U$ and $\gamma(b)\in I^+(x)\cap U$.  By Theorem \ref{thm:aleksandrov} this implies that the Alexandrov and the Lorentzian metric space topologies of $X$ coincide.

Then by Lemmas \ref{lem:braun-d}--\ref{lem:braun-caus}, the pair $(d,J)$ produces a time-separating function $l$. By \cite[Remark 5.2, Cor.\ 5.16]{minguzzi22} in our setting $J=\check{J}$, so this $l$ in fact coincides with the one defined in Corollary \ref{crl:l-as-sup}. So, by that corollary, for every $x,y\in X$, $l(x,y)$ is a supremum of lengths of isocausal curves connecting $x$ with $y$, and thus $(X,l)$ is a length metric spacetime.
\end{proof}

\section{Gromov-Hausdorff convergence}\label{sec:GH}
In this section we define the Gromov-Hausdorff convergence for non-compact Lorentzian metric spaces following the philosophy of the Gromov-Hausdorff convergence of pointed metric spaces.

It is convenient to introduce the following additional notation.
For a sequenced Lorentzian metric space $(X,d,(p^{n})_{n\in \mathbb{N}})$ we define $X^m=\overline{I(p^1,\ldots,p^m)}\cup \{p^1,\ldots,p^m\}$. Alternatively, we can write
\[
X^m=\overline{I_R(p^1,\ldots,p^m)},
\]
 where  $I_R=I\cup \Delta$ is the reflexive chronological relation with $\Delta\subset X\times X$ standing for the diagonal. It will be also convenient to have a short notation for equivalence classes in $BX^m$, so we define $[x]^m=[x]_{X^m}$.

For a relation $R\subset X\times X'$ we use the notation $R_1:=\pi_1(R)$ and $R_2:=\pi_2(R)$, where $\pi_1$ and $\pi_2$ are the canonical projections $X\times X'\to X$ and $X\times X'\to X'$ respectively. In this notation $R$ is a correspondence if and only if $R_1=X$ and $R_2=X'$.
If $X$ and $X'$ are endowed with Lorentzian distance functions $d$ and $d'$ respectively, by distortion of the relation $R$ we mean
\[
\mathrm{dis}R=\sup_{(x,x'),(y,y')\in R} |d(x,y)-d'(x',y')|.
\]
Note that we assume that the distortion is defined independently of $R$ being a correspondence.

\subsection{Quasi-correspondences}
Roughly speaking, in  metric geometry one says that a sequence of pointed metric spaces $\{(p_n,X_n)\}_{n\in \mathbb{N}}$ converges to a pointed metric space $(p,X)$ if for any $R>0$ the sequence of balls $B_R(p_n)\subset X_n$ converges to the ball $B_R(p)\subset X$. There are, however, two important details. Firstly, it is necessary to require that, in the right sense, the sequence of points $p_n$ converges to $p$, otherwise the limit may fail to be unique \cite{burago01}. Secondly, the direct realization of this idea, e.g.\ the requirement of the existence, for large enough $n$, of a correspondence of arbitrarily small distortion between the balls $B_R(p_n)$ and $B(p)$ is too restrictive. It is much more practical to require that the ball $B_R(p_n)$ is mapped to some set $Y\subset X$ that only approximates $B_R(p)$. Namely, in \cite{burago01} it is required that an $\epsilon$-neighborhood (in the sense of the Hausdorff distance) of $Y$ contains $B_{R-\epsilon}(p)$.

The natural analogue of a ball in the Lorentzian case is a closed chronological diamond. However, in general, a spacetime can not be written as a union of an increasing sequence of chronological diamonds, so we should work with closed chronological sets $\overline{I(p^1,\ldots,p^m)}$ instead. Another fundamental difference from the metric case is that the causal diamond is defined by two points instead of one point and a number. For this reason, we should deal with the sequenced spaces instead of the pointed ones. The points of the chosen generating sequence, similarly to the preferred points of pointed spaces, have two roles: they define compact sub-regions and serve as benchmarks ensuring uniqueness of the limit. Moreover, we can not slightly adjust the radius to get a subset of the causal diamond, therefore we need another way to deform it. These ideas lead to the concept of quasi-correspondence to which this subsection is devoted.

\begin{definition}
    Let $(X,d,\{p^k\}_{k\in\mathbb{N}})$ and $(X',d',\{p'{}^k\}_{k\in\mathbb{N}})$ be sequenced Lorentzian metric spaces. For $m\in \mathbb{N}$ and $\epsilon>0$ we say that $R\subset X^m\times X'{}^m$ is an $(m,\epsilon)$ \emph{quasi-correspondence between  $(X,d,\{p^k\}_{k\in\mathbb{N}})$ and $(X',d',\{p'{}^k\}_{k\in\mathbb{N}})$} (or just $(m,\epsilon)$ $X-X'$-\emph{quasi-correspondence}) if the following holds:
    \begin{itemize}
        \item $R$ is compact;
        \item $I_{\epsilon}(p^1,\ldots,p^m)\subset R_1$
        \item  $I_{\epsilon}(p'{}^1,\ldots,p'{}^m)\subset R_2$;
        \item $\mathrm{dis}R<\epsilon$;
        \item $(p^r,p'{}^r)\in R$ for $r=1,\ldots,m$.
    \end{itemize}
\end{definition}

\begin{remark}\label{rmk:quasi-cor-big-eps}
    Clearly, if $R$ is an $(m,\epsilon)$ $X-X'$ quasi-correspondence, then it is also an $(m,\epsilon')$ $X-X'$ quasi-correspondence for any $\epsilon'\geq \epsilon$.
\end{remark}
\begin{remark}
    The compactness assumption is not essential and is introduced for convenience only. Indeed, if $R$ satisfies all the conditions but  compactness, then its closure $\overline{R}$ is compact and still satisfies the remaining properties, cf.\ \cite[Prop.\ 4.5]{minguzzi22}.
\end{remark}
\begin{remark}\label{rmk:quasi-cor-alt}
    One can view an $(m,\epsilon)$ quasi-correspondence as a usual correspondence between the compact spaces $R_1$ and $R_2$ such that
    \begin{itemize}
        \item $R$ is compact;
        \item $I_{\epsilon}(p^1,\ldots,p^m)\subset R_1\subset X^m$;
        \item  $I_{\epsilon}(p'{}^1,\ldots,p'{}^m)\subset R_2 \subset X'{}^m$;
       \item $\mathrm{dis}R<\epsilon$;
        \item $(p^r,p'{}^r)\in R$ for $r=1,\ldots,m$.
    \end{itemize}
    See also Fig.\ \ref{xm}
\end{remark}

\begin{figure}[ht]
\centering
\includegraphics[width=12cm]{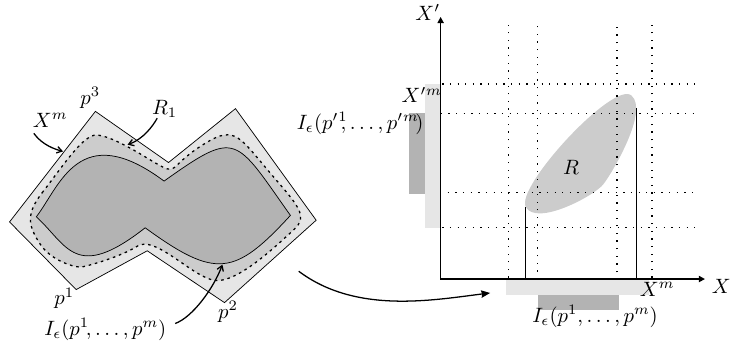}
\caption{The $(m,\epsilon)$ quasi-correspondence.}\label{xm}
\end{figure}
\begin{lemma}\label{lem:quasicor-op}
    If $(X,d,\{p^k\}_{k\in\mathbb{N}})$ and $(X',d',\{p'{}^k\}_{k\in\mathbb{N}})$ are sequenced Lorentzian metric spaces, and $R$ is an $(m,\epsilon)$ $X-X'$ quasi-correspondence, then $R^T$ is an $(m,\epsilon)$ $X'-X$-quasi-correspondence. If in addition to that $(X'',d'',\{p''{}^k\}_{k\in\mathbb{N}})$ is a sequenced Lorentzian metric space and $R'$ is an $(m,\epsilon')$ $X''-X'$ quasi-correspondence, then $R'\circ R$ is a $(m,\epsilon+\epsilon')$ $X''-X$ quasi-correspondence.
\end{lemma}
\begin{proof}
The only non-trivial fact is that for $R$, $R'$ as in the second statement,
    \[I_{\epsilon+\epsilon'}(p^1,\ldots,p^m)\subset (R'\circ R)_1,\quad I_{\epsilon+\epsilon'}(p'{}^1,\ldots,p'{}^m)\subset (R'\circ R)_2.
    \]
    We concentrate on the first inclusion, the other can be shown analogously. Consider $x\in I_{\epsilon+\epsilon'}(p^1,\ldots,p^m)$. Then there are some $i,j\in\{1,\ldots,m\}$ such that $d(p^i,x)\geq \epsilon+\epsilon'$ and $d(x,p^j)\geq \epsilon+\epsilon'$. Since $x\in I_\epsilon(p^1,\ldots,p^m)\subset R_1$ there is also some $x'\in X'$ such that $(x,x')\in R$. We have
    \[
    d(p'{}^i,x')>d(p^i,x)-\epsilon\geq \epsilon', \quad d(x',p'{}^j)>d(x,p^j)-\epsilon\geq \epsilon',
    \]
    so $x'\in I_{\epsilon'}(p'{}^1,\ldots p'{}^m)\subset {R'}_1$. Thus, there is $x''\in X'$ such that $(x',x'')\in R'$, and therefore $(x,x'')\in R'\circ R$. We conclude that $x\in (R'\circ R)_1$.

    The distortion estimates are the same as in \cite[Lemma 4.8]{minguzzi22}, the rest is clear.
\end{proof}

The two lemmas above are generalizations of the corresponding properties of the usual correspondences. Instead, the next ones are specific for the quasi-correspondences and this is essentially the reason to introduce them. We start from the following result.
\begin{lemma}\label{lem:sub-cor-quasi-cor}
    Let $X$ and $X'$ be two Lorentzian metric spaces and let $Y\subset X$ and $Y'\subset X'$ be compact subsets. Let $R\subset Y\times Y'$  be  a compact correspondence such that $\mathrm{dis}R<\epsilon$ where $\epsilon>0$. Let $m\in\mathbb{N}$ and $p^1,\ldots,p^m\in Y$, $p'{}^1,\ldots,p'{}^m\in Y'$ and let us assume that the following conditions hold:
    \begin{itemize}
        \item $I_{\epsilon}(p^1,\ldots,p^m)\subset Y$;
        \item $I_{\epsilon}(p'{}^1,\ldots,p'{}^m)\subset Y'$;
        \item $(p^i,p'{}^i)\in R$ for each $i=1,\ldots, m$.
    \end{itemize}
     The relation
    \[
    R'=R\cap \Big(\overline{I_R(p^1,\ldots,p^m)}\times \overline{I_R(p'{}^1,\ldots,p'{}^m)}\Big),
    \]
    satisfies $\mathrm{dis}R'<\mathrm{dis}R$, $I_{\epsilon}(p^1,\ldots,p^m)\subset R'_1$ and $I_{\epsilon}(p'{}^1,\ldots,p'{}^m)\subset R'_2$.
\end{lemma}
\begin{proof}
    By construction $R'\subset R$, so $\mathrm{dis}R'\leq \mathrm{dis}R<\epsilon$.

    Let us show that $I_{\epsilon}(p^1,\ldots,p^m)\subset R'_1$. For that take $x\in I_{\epsilon}(p^1,\ldots,p^m)$. By assumptions $x\in R_1$, so there is $x'\in Y'$ such that $(x,x')\in R$. We intend to show that $x'\in I(p'{}^1,\ldots,p'{}^m)$, which would imply $(x,x')\in R'$ and thus $x\in R'_1$.  Since $x\in  I_{\epsilon}(p^1,\ldots,p^m)$, there are $i,j\in \{1,\ldots,m\}$ such that $d(p^{i},x)\geq \epsilon$ and $d(x,p^{j})\geq \epsilon$. Taking into account that $(x,x')\in R$, $(p^{i},p'{}^i)\in R$, $(p^{j},p'{}^j)\in R$ and $\mathrm{dis}R<\epsilon$ we get
    \[
    |d(p^{i},x)-d'(p'{}^{i},x')|<\epsilon,
    \]
    \[
    |d(x,p^{j})-d'(x',p'{}^{j})|<\epsilon.
    \]
    thus
    \[
    d(p'{}^{i},x')>0,\quad d(x',p'{}^{j})>0,
    \]
    so $x'\in  I(p'{}^1,\ldots,p'{}^m)$.

    The proof of the inclusion  $I_{\epsilon}(p'{}^1,\ldots,p'{}^m)\subset R'_2$ is analogous.
\end{proof}

\begin{corollary}\label{crl:sub-quasi-corr}
    Let $(X,d,\{p^k\}_{k\in\mathbb{N}})$ and $(X',d',\{p'{}^k\}_{k\in\mathbb{N}})$ be sequenced Lorentzian metric spaces. Let $\{r_k\}_{k\in\mathbb{N}}$ be an increasing sequence of natural numbers such that
    \[
    \{p^{r_k}|k\in\mathbb{N}\} \quad \mathrm{and}\quad\{p'{}^{r_k}|k\in\mathbb{N}\}
    \]
    are still generating sets of $X$ and $X'$ respectively.  Let $R$ be a $(m,\epsilon)$ $X-X'$ quasi-correspondence and let $l\in\mathbb{N}$ be such that $r_{l}\leq m$. The relation
       \[
    R'=R\cap \Big(\overline{I_R(p^{r_1},\ldots,p^{r_{l}})}\times \overline{I_R(p'{}^{r_1},\ldots,p'{}^{r_{l}})}\Big)
    \]
    is a $(l,\epsilon)$ quasi-correspondence between $(X,d,\{p^{r_k}\}_{k\in \mathbb{N}})$ and $(X',d',\{p'{}^{r_k}\}_{k\in\mathbb{N}})$.

\end{corollary}

\begin{proof}
Keeping in mind Remark \ref{rmk:quasi-cor-alt}, apply Lemma \ref{lem:sub-cor-quasi-cor} with $Y=R_1$, $Y'=R_2$, and $p^{r_1},\ldots,p^{r_{l}}$ (respectively, $p'{}^{r_1},\ldots,p'{}^{r_{l}}$) playing the role of $p^1,\ldots,p^m$ (respectively, $p'{}^1,\ldots,p'{}^m$).
\end{proof}

\begin{corollary}\label{crl:bound-corr-quasi-corr}
    Let $(X,d)$  be a bounded Lorentzian metric space such that $I(X)$ is dense, and similarly for $(X',d')$. Let $R\subset X\times X'$  be  a compact correspondence such that $\mathrm{dis}R<\epsilon$ and let $\{p^k\}_{k\in\mathbb{N}}$ be a generating sequence for $I(X)$. Then for every $m\in\mathbb{N}$ it is possible to find a generating sequence
    $\{p'{}^k\}_{k\in\mathbb{N}}$ of $I(X')$ such that for every $l\leq m$ there is a $(l,\epsilon)$ quasi-correspondence between $(I(X),d,\{p^k\}_{k\in\mathbb{N}})$ and $(I(X'),d',\{p'{}^k\}_{k\in\mathbb{N}})$.
\end{corollary}
Recall that under the above assumptions, $I(X)$ and $I(X')$ are countably generated (Theorem \ref{thm:blmc-cg}).

\begin{proof}
Without loss of generality we can assume that $X$ and $X'$ include  the spacelike boundary, so that they are compact.
Let $\{q^k\}_{k\in\mathbb{N}}$ be any generating sequence of $I(X')$ and let $0<\delta<(\epsilon-\mathrm{dis}R)/2$.
For each $k\in\{1,\ldots,m\}$ let $p''{}^{k}\in X'$ be such that $(p^{k},p''{}^{k})\in R$ and choose $p'{}^{k}\in I(X')$  so  that $\gamma_{X'}(p'{}^k,p''{}^k)<\delta$,
where $\gamma_{X'}$ is the distinction metric \cite{minguzzi22} of $X'$. Such a point $p'{}^k$ always exists because $I(X')$ is dense and the distinction metric is continuous.
Set $p'{}^{m+k}=q^k$ for $k\in\mathbb{N}$. Clearly, this defines a generating sequence $\{p'{}^k\}_{k\in \mathbb{N}}$ of $X'$. Then for each $l\in\{1,\ldots,m\}$ apply Lemma \ref{lem:sub-cor-quasi-cor} to $Y=X$, $Y'=X'$ and $p^{1},\ldots p^l$ (respectively, $p'{}^{1},\ldots p'{}^l$) instead of $p^{1},\ldots p^m$ (respectively, $p'{}^{1},\ldots p'{}^m$) where the correspondence is $R\cup \{(p^k, p'{}^k), k=1, \cdots, m\}$.
\end{proof}

An important property of the Gromov-Hausdorff distance is its positivity on pairs of non-isometric spaces. This allows us to use it as a true measure of similarity of two metric spaces. The following result plays a similar role for the quasi-correspondences.
\begin{proposition}\label{prop:quasi-corr-implies-iso}
Let $({X},d,\{p^k\}_{k\in\mathbb{N}})$ and $({X}',d',\{p'^k\}_{k\in\mathbb{N}})$ be two sequenced Lorentzian metric spaces. Then the following statements are equivalent:
\begin{enumerate}
    \item $({X},d,\{p^k\}_{k\in\mathbb{N}})$ and $({X}',d',\{p'^k\}_{k\in\mathbb{N}})$ are isomorphic;
    \item  For every $m\in \mathbb{N}$ and $\epsilon>0$ there is a $(m,\epsilon)$ $X-X'$-quasi-correspondence $R_{m,\epsilon}$.
\end{enumerate}
\end{proposition}

We recall that the notion of isomorphism of sequenced Lorentzian metric space was introduced in Definition \ref{def:cg-iso}.
\begin{proof}
    If $\phi: X\to X'$ is an isomorphism, then
    \[
    R_{m,\epsilon}=\{(x,\phi(x))| x\in X^m\}
    \]
    is a $(m,\epsilon)$ $X-X'$ quasi-correspondence. So, we have to prove only the converse implication.

    Let us denote with $R_{m,\epsilon}$ the quasi-correspondences of  the second statement. We shall need some preliminary constructions.


    Let $\{\epsilon_m\}_{m\in\mathbb{N}}$ be a decreasing sequence of positive numbers converging to zero.
    For every $m\in\mathbb{N}$ we set $Q_m=R_{m,\epsilon_m}$.

     Let $\mathscr{S}\subset X$ be a countable dense subset of $X$.
     Without any loss of generality we can assume that $p^k\in\mathscr{S}$ for each $k\in\mathbb{N}$ (or include the sequence $\{p^k\}_{k
     \in \mathbb{N}}$ in $\mathscr{S}$).
    For every $m\in \mathbb{N}$ and $x\in \mathscr{S}\cap I_{\epsilon_m}(p^1,\ldots,p^m)$ we choose $\phi_m(x)\in X'$ so that $(x,\phi_m(x))\in Q_m$. We also define for every $k\leq m$, $\phi_m(p^k)=p'{}^k$. These two conditions are compatible by the defining properties of the quasi-correspondences. This construction defines $\phi_m(x)$ for each $x\in\mathscr{S}$ provided that $m$ is sufficiently large.

    Let us show that for every $r\in\mathbb{N}$,  $x\in \mathscr{S}\cap I_{\epsilon_r}(p^1,\ldots,p^r)$, and $m\geq r$ we have $\phi_m(x)\in X'{}^r$. Indeed, for every such a point $x$ there are $i,j\in \{1,\ldots,r\}$ such that
    \[
    d(p^i,x)\geq \epsilon_r,\ d(x,p^j)\geq \epsilon_r,
    \]
    which implies $d(p'{}^i,\phi_m(x))>0$, $d(\phi_m(x),p'{}^j)>0$, and thus $\phi_m(x)\in X'{}^r$. Here we used the assumption $\epsilon_r\geq\epsilon_m$ for $m\geq r$.
    The set $X'{}^r$ is compact, so for each $x\in\mathscr{S}$ the sequence $\phi_m(x)$ has a convergent subsequence. By the  diagonal argument we can do it simultaneously for all $x\in \mathscr{S}$.

    The procedure above produces a map $\phi: \mathscr{S} \to X'$.
    By straightforward arguments, $\phi$ preserves the distance, and $\phi(p^k)=p'{}^k$ for all $k\in\mathbb{N}$. Let us show that $\phi$ can be extended to a distance-preserving map $X\to X'$. The argument is very close to the proof of \cite[Prop.\ 4.18]{minguzzi22}. Let $x\in X$, then there are some $i,j\in \mathbb{N}$ and $\epsilon>0$ such that
    \[
        d(p^i,x)>\epsilon, \ d(x,p^j)>\epsilon.
    \]
    Since $X$ is first countable, there is a sequence $\{x_n\}_{n\in\mathbb{N}}$ of points of $\mathscr{S}$ such that
    \[
        d(p^i,x_n)>\epsilon, \ d(x_n,p^j)>\epsilon,
    \]
    and $x_n\to x$ as $x\to \infty$. Take $m\in\mathbb{N}$ such that $i,j\leq m$ and $\epsilon_m<\epsilon$. Then $\phi(x_n)\in X^m$. As $X^m$ is compact, we can set $\phi(x)$ to be the limit of an arbitrarily chosen convergent subsequence of $\{\phi(x_n)\}_{n\in\mathbb{N}}$. By the same argument as in \cite[Prop.\ 4.18]{minguzzi22}, $\phi$ is distance-preserving. Therefore, $\phi$ is continuous, as preimage of any set of the subbsasis (\ref{don}) of $X'$ contains a set of the analogous subbasis of $X$. It follows that $\phi(X^r)\subset X'{}^r$ for every $r\in \mathbb{N}$.

    By reversing the roles of $X$ and $X'$ we can obtain another distance preserving map $\phi':X'\to X$, such that $\phi'(X'{}^m)\subset X{}^m$ and $\phi'(p'{}^m)=p^m$ for all $m\in \mathbb{N}$.

    The distance preserving maps pass to the quotient under $\sim$, thus we can define sequences of the distance-preserving maps $\phi^m=\phi\vert_{\mathring{X}^m}/\!\sim$ and $\phi'^{m}=\phi'\vert_{\mathring{X}'^m}/\!\sim$. Note that the characterizing property of $\phi^m$ is that $\phi^m([x]^m)=[\phi(x)]^m$ for any $x\in \mathring{X}^m$.

    As $BX^r$ and $BX'{}^r$ are bounded Lorentzian metric spaces, by  \cite[Thm. 3.5]{minguzzi22} it follows that for every $m$ the map $\phi^{m}$ is bijective. We are going to show that this implies the bijectivity of $\phi$. Namely, that for $x'\in X'$ there is exactly one $x\in X$ such that $\phi(x)=x'$.

    Take $l$ such that $x'\in I(p'{}^1,\ldots,p'{}^l)$. Let $i,j\in \{1,\ldots, l\}$ be such that $x'\in I(p'{}^i,p'{}^j)$.
    Then for any $m \ge l$, by bijectivity of $\phi^m$, there is a unique $x^m\in BX^m$ such that $\phi^m(x^m)=[x']^m$. We claim that $x^m$ are nested, i.e.\ for $m\ge n\ge l$  we have $x^m\subset x^n$.
    Indeed, take $y\in x^m$. Note that $\phi(y)\in I(p'{}^i,p'{}^j)$ (because $[\phi(y)]^m=[x']^m$, where $x'\in I(p'{}^i,p'{}^j)$).
    Thus, we can apply Lemma \ref{cqptx} to derive
    \[
        x'\in [\phi(y)]^m\subset [\phi(y)]^n=\phi^n([y]^n).
    \]
    But this means that $\phi^{n}([y]^n)=[x']^n$, thus $[y]^n=x^n$ (note that the injectivity of $\phi^n$, implying uniqueness of $x^n$ is essential at this point) which is the same as $y\in x^n$. As this holds for arbitrary $y\in x^m$, we conclude that $x^m\subset x^n$.

    Applying the second statement of Lemma \ref{lemClasses} to the family $\{x^m\}_{m\ge l}$ we get that there is one and only one  point $x\in X$ such that $\star$: $[x]^m=x^m$ for any $m\ge l$. By construction, property $\star$ of $x$  is equivalent to
    \[
        [\phi(x)]^m=[x']^m, \qquad \forall m\ge l .
    \]
    Applying the first statement of Lemma \ref{lemClasses},
    \[
    \{\phi(x)\}=\cap_{m\ge l} \, [\phi(x)]^m= \cap_{m\ge l}\, [x']^m= \{x'\},    \]
    which reads    $\phi(x)=x'$.
    The argument above shows that for any $x'\in X'$ there is a unique $x\in X$ such that $\phi(x)=x'$, so this concludes the proof of the bijectivity of $\phi$.
\end{proof}

\begin{remark}\label{rmk:prop:quasi-corr-implies-iso}
    From the proof of Proposition \ref{prop:quasi-corr-implies-iso} we see that a sufficient condition for the isomorphism of the Lorentzian metric spaces $(X,d,\{p^k\}_{k\in\mathbb{N}})$ and $(X',d',\{p'{}^k\})$ is the existence of an $(m,\epsilon_m)$ $X-X'$ quasi-correspondence for every $m\in\mathbb{N}$, where $\epsilon_m>0$ and $\lim_{m\to \infty}\epsilon_m>0$.
\end{remark}
\subsection{Gromov-Hausdorff convergence}
\begin{definition}\label{def:GHconv}
We say that the sequence of sequenced Lorentzian metric spaces $({X}_n,d_n,\{p_n^m\}_{m\in \mathbb{N}})$ GH-converges to the sequenced Lorentzian metric space $({X},d,\{p^m\}_{m\in \mathbb{N}})$ if, for every $m$ and $\delta>0$, there exists $n_0$, such that for every $n \ge n_0$ there exists an $(m,\delta)$ $X_n-X$ quasi-correspondence.
\end{definition}

The following alternative formulation is less insightful, but often more practical.
\begin{lemma}\label{lem:weakGH-alt}
Let $\{\delta_m\}_{m\in\mathbb{N}}$ be   a sequence of real positive numbers  such that
\[
\lim_{m\rightarrow\infty}\delta_m=0.
\]
Let $({X}_n,d_n,\{p_n^k\}_{k\in\mathbb{N}})$ be a sequence of sequenced Lorentzian metric spaces, and let $({X},d,\{p^k\}_{k\in\mathbb{N}})$ be one more sequenced Lorentzian metric space. The  following statements are equivalent:
\begin{enumerate}
    \item $({X}_n,d_n,\{p_n^k\}_{k\in\mathbb{N}})$ GH-converges to  $({X},d,\{p^k\}_{k\in\mathbb{N}})$;
    \item There is an increasing sequence of integers $\{N_m\}_{m\in\mathbb{N}}$ and  relations $R_n\subset X_n\times X$ for $n\in\mathbb{N}$,  such that for every $m\in\mathbb{N}$  and $n$ such that $N_m\leq n \leq N_{m+1}$, $R_n$ is a $(m,\delta_m)$ $X_n-X$ quasi-correspondence.
\end{enumerate}
\end{lemma}
The main advantage of this formulation is that instead of several families of correspondences (one for each $m$) we have just one family $R_n$ which is both enlarging so that its domains eventually exhaust the whole space, and becoming precise as $n$ grows.

It is also convenient to introduce, along with $N_m$, a dual sequence
\begin{equation}\label{eqn:rn-from-Nn}
r_n=\max\{m\in\mathbb{N}|n\geq N_m\}.
\end{equation}
It is easy to see that $r_n$ is (not necessarily strictly) increasing,
\begin{equation}\label{eqn:rn-and-Nn}
    N_m \leq n \leq N_{m+1} \Leftrightarrow r_n=m,\quad \forall n,m\in\mathbb{N},
\end{equation}
and
\begin{equation}\label{eqn:Nn-from-rn}
    N_m=\min\{n\in\mathbb{N} | r_n\geq m\}.
\end{equation}

\begin{proof}
Let us assume that   $({X}_n,d_n,\{p_n^k\}_{k\in\mathbb{N}})$  $GH$-converges  to $({X},d,\{p^k\}_{k\in\mathbb{N}})$. For every $m\in\mathbb{N}$ there is $N_m$ such that for every $n\ge N_m$ there is an $X_n-X$ quasi-correspondence $R^{n,m}$ of order $m$ and distortion $\delta_{m}$. By increasing $N_m$ when necessary, we can always make it into a strictly increasing sequence. Set $R_n=R^{n,r_n}$, where $r_n$ is defined by (\ref{eqn:rn-and-Nn}). Taking (\ref{eqn:rn-and-Nn}) into account, we get that for every $n$ such that $N_m\leq n\leq N_{m+1}$, $R_n$ is a $(m,\delta_m)$ $X_n-X$ quasi-correspondence. We conclude that the first statement implies the second one.

For the converse, assume that $N_m$ and $R_n$ as in the second statement are given for each $m, \,n\in\mathbb{N}$.  Fix $m\in \mathbb{N}$ and $\epsilon>0$. Then there is $m_0\in\mathbb{N}$ such that for every $m'\geq m_0$ we have $\delta_{m'}< \epsilon$. We can always assume that $m_0>m$. In order to prove the Gromov-Hausdorff convergence, we have to show that we can pick some $N$ such that for every $n\geq N$ there is an $(m,\epsilon)$ $X_n-X$ quasi-correspondence $R'_n$. We intend to show that this holds for $N=N_{m_0}$.
By the assumption, for $n\geq N_{m_0}$ there is an $(r_n,\delta_{r_m})$ $X_n-X$ quasi-correspondence $R_n$. Here $r_n$ is defined by (\ref{eqn:rn-and-Nn}). By Remark \ref{rmk:quasi-cor-big-eps}, $R_n$ is also a $(r_n,\epsilon)$ quasi-correspondence. Note that $n\geq N_m$ implies $r_n\geq m$, so, applying Corollary \ref{crl:sub-quasi-corr}, there is a $(m,\epsilon)$ $X_n-X$ quasi-correspondence $R'_n$. We conclude that the second statement implies the first one.
\end{proof}

\begin{theorem}\label{thm:uniqueGHlim}
Suppose that a sequence of sequenced Lorentzian metric spaces $({X}_n,d_n,\{p_n^k\}_{k\in\mathbb{N}})$  $GH$-converges  to both the sequenced Lorentzian metric spaces  $({X},d,\{p^k\}_{k\in\mathbb{N}})$ and $({X}',d',\{p'^k\}_{k\in\mathbb{N}})$, then ${X}$ and ${X}'$ are isomorphic.
\end{theorem}

\begin{proof}
Let us consider a sequence of positive numbers $\epsilon_m$ such that $\lim_{m\to\infty}\epsilon_m=0$. By Proposition \ref{prop:quasi-corr-implies-iso} and Remark \ref{rmk:prop:quasi-corr-implies-iso}, in order to show that $X$ and $X'$ are isomorphic it is enough to build an $(m,\epsilon_m)$ $X-X'$ quasi-correspondence. Applying Lemma \ref{lem:weakGH-alt} to $\epsilon_m/2$ in the place of $\epsilon_m$, for every $m\in \mathbb{N}$ there is some $n$ such that there is an $(m,\epsilon_m/2)$ $X_n-X$ quasi-correspondences $R_m$ and an $(m,\epsilon_m/2)$ $X_n-X'$ quasi-correspondences $R'_m$. Then by Lemma \ref{lem:quasicor-op}, $Q_m=R'_mR_m^T$  is an $(m,\epsilon_m)$ $X-X'$ quasi-correspondence.
\end{proof}

\subsection{Dependence of GH limit on the generating sequences}
This section can be skipped on first reading.
We want to study how the GH-convergence and GH limit depend on the chosen generating sequence.
We start with the following observation.

\begin{lemma}\label{lem:gen-set-fin-rem}
   Let $\mathscr{G}$ be a generating set of a Lorentzian metric space $(X,d)$ and $S\subset \mathscr{G}$ be a finite subset. Then
   $\mathscr{G}\setminus S$ is a generating set of $X$.
\end{lemma}
\begin{proof}
    The proof goes by induction on the cardianality of $S$.
    The base $S=\varnothing$ is clear. So, suppose that the result is already established for $|S|<n$ and let us prove it for
    \[
    S=\{p^1,\ldots,p^n\}.
    \]
    Because $\mathscr{G}$ is a generating set, and the chronological order is anti-reflexive, there are some $q,r\in \mathscr{G}\setminus \{p^n\}$ such that $q\ll p^{n} \ll r$. By transitivity of the chronological order, it follows that
    \[
    I(\mathscr{G})=    I(\mathscr{G} \setminus \{p^n\}),
    \]
    so $\mathscr{G} \setminus \{p^n\}$ is a generating set. But then by inductive assumption
    \[
    \mathscr{G} \setminus S= (\mathscr{G} \setminus \{p^n\})\setminus \{p^1,\ldots,p^{n-1}\}
    \]
    is a again a generating set.
\end{proof}

\begin{theorem}\label{thm:change-gensec-X}
         Let $(X_n,d_n,\{p_n^k\}_{k\in \mathbb{N}})$ be a sequence of sequenced Lorentzian metric spaces GH-convergent to a sequenced Lorentzian metric space $(X,d,$ $\{p^k\}_{k\in\mathbb{N}})$. Let $\{q^k\}_{k\in\mathbb{N}}$ be another generating sequence for $X$. Then it is possible to choose for each $n\in\mathbb{N}$ a generating sequence $\{q^k_n\}_{k\in\mathbb{N}}$ for $X_n$ so that
          $(X_n,d_n,\{q_n^k\}_{k\in \mathbb{N}})$ GH-converges to  $(X,d,\{q^k\}_{k\in \mathbb{N}})$.

          Moreover, suppose that there is a set $A\subset\mathbb{N}$ and a function $\rho: A\to \mathbb{N}$ such that for every $k\in A$, $q^{k}=p^{\rho(k)}$. Assume that \emph{at least one} of the following conditions hold:
          \begin{enumerate}
              \item[(a)] The image of $\rho$ is a cofinite subset of $\mathbb{N}$;
              \item[(b)] The set $\mathbb{N}\setminus A$ is infinite.
          \end{enumerate}
          Then we can accomplish the above result while satisfying
          for every $n\in\mathbb{N}$ and $k\in A$
          \[
          q^{k}_n=p^{\rho(k)}_n.
          \]
\end{theorem}
With option (a), the theorem allows one to remove finitely many points from the generating set, change the order, add or remove duplicates,  and add finite or countable amount of the new points. The idea behind the co-finitness assumption is that it ensures $q_n^k$ being a generating set (by Lemma \ref{lem:gen-set-fin-rem}).

The option (b) (which includes the case $A=\varnothing$, i.e.\ the main statement of the theorem) instead  allows to change the generating set drastically, keeping arbitrary small portion of the original points. Then we need sufficient points so as  to get a new generating set. This is why condition (b) appeared.
\begin{proof}
        We start from the main statement.
        For every $k\in\mathbb{N}$ let us choose $\mu(k)\in\mathbb{N}$ and $\epsilon_k>0$, in such a way that
        \[
        q^k\in I_{\epsilon_k}(p^1,\ldots,p^{\mu(k)}).
        \]
        Without loss of generality we can assume that $\epsilon_k$ is decreasing, $\epsilon_k\to 0$, and $\mu$ is  increasing, and $\mu(k)\geq k$ for every $k\in \mathbb{N}$.

Let $\delta_m=\epsilon_{m}/2$.
        Let $\{N_m\}_{m\in\mathbb{N}}$, $\{R^n\}_{n\in\mathbb{N}}$ be sequences described in Lemma \ref{lem:weakGH-alt} associated to the sequence $\delta_m$.

        Set $N'_m=N_{\mu(m)}$. As a composition of two increasing maps, $N'_m$ is also increasing. Analogously to (\ref{eqn:rn-and-Nn}) we can define
        \[
        r'_n=\max\{m\in\mathbb{N}|n\geq N_m'\},
        \]
        which implies $r'_n\le r_n$.

        Let $n$ be given and let $k\leq r'_n$. This is equivalent to $n\geq N'_k=N_{\mu(k)}$ and hence to $\mu(k)\leq r_n$, where $r_n$ is defined by (\ref{eqn:rn-and-Nn}).
        Recall that $R^n$ is a $(r_n,\delta_{r_n})$ $X_n-X$ quasi-correspondence, so the conditions above imply (note that $k\le \mu(k)\leq r_n$ so $\delta_{r_n}\le \delta_k \le \epsilon_k$)
        \[
        q^k\in I_{\epsilon_k}(p^1,\ldots,p^{\mu(k)})\subset I_{\delta_{r_n}}(p^1,\ldots,p^{r_n})\subset R^n_2.
        \] Therefore, for $k\leq r'_n$, we can choose $q_n^k$ so that $(q_n^k,q^k)\in R^n$. For all $k>r'_n$ we pick $q_n^k$ in such a way that for each $n$ the sequence $\{q_n^k\}_{k\in\mathbb{N}}$ is generating.

        Now let us build new quasi-correspondences $R'{}^n$ between the spaces  $(X_n,d_n,$ $\{q_n^k\}_{k\in \mathbb{N}})$ and  $(X,d,\{q^k\}_{k\in \mathbb{N}})$. For $k\leq r'_n$ we have
        \[
          \max_{j\in\{1,\ldots,r_n\}}d_n(q_n^k,p_n^j)>\max_{j\in\{1,\ldots,\mu(k)\}}d(q^k,p^j)-\delta_{r_n}\ge \epsilon_k-\delta_{r_n}\ge 2\delta_{k}- \delta_{r_n} \ge \delta_{r_n},
        \]
        \[
          \max_{j\in\{1,\ldots,r_n\}}d_n(p_n^j,q_n^k)>\max_{j\in\{1,\ldots,\mu(k)\}}d(p^j,q^k)-\delta_{r_n}\ge \epsilon_k-\delta_{r_n}\ge 2\delta_{k}- \delta_{r_n} \ge \delta_{r_n},
        \]
        Thus, $q_n^k\in I_{\delta_{r_n}}(p_n^1,\ldots, p_n^{r_n})$, and, by the reverse triangle inequality (or transitivity of $ I_{\delta_{r_n}}$, see Subsection \ref{ssec:notation}),
        \[
        I_{\delta_{r_n}}(q_n^1,\ldots,q_n^{r'_n})\subset I_{2\delta_{r_n}}(p_n^1,\ldots, p_n^{r_n})\subset I_{\delta_{r_n}}(p_n^1,\ldots, p_n^{r_n})\subset R^n_1.
        \]
        Similarly,
        \[
        I_{\delta_{r_n}}(q^1,\ldots,q^{r'_n})\subset  I_{2\delta_{r_n}}(p^1,\ldots, p^{r_n})\subset I_{\delta_{r_n}}(p^1,\ldots, p^{r_n})\subset R^n_2.
        \]
        These two inclusions together with the condition $(q_n^k,q^k) \in R^n$ for each $k\leq r'_n$ allows us to use Lemma \ref{lem:sub-cor-quasi-cor} to build from $R^n$ an $(r'_n,\delta_{r_n})$ quasi-correspondence
        \[
        R'{}^n=R^n\cap \big(\overline{I_R(q^1_n,\ldots,q^{r'_n}_n)}\times \overline{I_R(q^1,\ldots,q^{r'_n}_n)}\big)
        \]
        between the spaces  $(X_n,d_n,\{q_n^k\}_{k\in \mathbb{N}})$ and  $(X,d,\{q^k\}_{k\in \mathbb{N}})$. By the analogue of (\ref{eqn:rn-and-Nn}), this means that whenever $N_m'\leq n\leq N'_{m+1}$, $R'_n$ is an $(m=r'_n,\delta_{r_n})$ quasi-correspondence, hence a $(m=r'_n,\delta_{r'_n})$ quasi-correspondence (because $r'_n \le r_n$ implies $\delta_{r'_n} \ge \delta_{r_n}$),  so by Lemma \ref{lem:weakGH-alt}, $(X_n,d_n,\{q_n^k\}_{k\in \mathbb{N}})$ GH-converges to  $(X,d,\{q^k\}_{k\in \mathbb{N}})$.

        Finally, let us move to the additional statement. Assume that there is a set $A\subset \mathbb{N}$ and a function $\rho: A\to \mathbb{N}$ such that $q^{k}=p^{\rho(k)}$ for any $k\in A$. Then we can adjust the function $\mu$ so that
        \[
        \mu(m)\geq \max_{k\in \{1,\ldots,m\}\cap A} \rho(k).
        \]
        When, repeating construction of $N'_m$ and $r'_n$, we get that for every $k\in A$ such that $k\leq r'_n$ it holds $\rho(k)\le \mu(k)\leq r_n$, and thus $(p^{\rho(k)}_n,p^{\rho(k)})\in R^n$. This allows to set $q^k_n=p^{\rho(k)}_n$ without destroying the construction. Now we need to show that it is  also possible to set $q^k_n=p^{\rho(k)}_n$  for $k>r'_n$ (and $k\in A$). We recall that the only constraint for $q^k_n$ with large values of $k$ is that $\{q^k_n\}_{k\in\mathbb{N}}$ is a generating sequence. Here the proof goes in different ways for the options (a) and (b).
        \begin{enumerate}
            \item[(a)] If the image of $\rho$ is co-finite, then for every $n\in\mathbb{N}$ the set $\{q^k_n|k\in\mathbb{N}\}$ contains the set $\{p^k_n\}_{k\in\mathbb{N}}\setminus S_n$, where $S_n$ is finite. So, by Lemma \ref{lem:gen-set-fin-rem} $\{q^k_n|k\in\mathbb{N}\}$ generates $X_n$ independently of the choice made for $q^k_n$ for $k\notin A$.
            \item[(b)] If the set $B=\mathbb{N}\setminus A$ is infinite, then
            \[
            B_n=\{k\in B|k>r'_n\}
            \]
            is also infinite for any $n\in\mathbb{N}$. In the construction above the points $q_n^k$ with $k\in B_n$ can be chosen arbitrarily, so we can always ensure that $\{q_n^k|k\in B_n\}$ is a generating set of $X_n$.
         \end{enumerate}
\end{proof}

\subsubsection{An example} \label{exmp:bad-renum}
We finish by an example showing that the GH-limit in fact depends on the chosen generating sequences.
    Set $X_n=(0,+\infty)$, $d_n(x,y)=(x-y)_{+}$, $p^{2k}_n=n+k$, $p^{2k-1}_n=n-k+1$ for $k<n$ and $p^{2k-1}_n=(k-n+1)^{-1}$ for $k\geq n$. It is easy to see that for each $n\in\mathbb{N}$, $(X_n,d_n,\{p_n^k\}_{k\in\mathbb{N}})$ is a sequenced Lorentzian metric space.

    Now set $X=\mathbb{R}$, $d(x,y)=(x-y)_{+}$, $p^{2k}=k$, $p^{2k-1}=1-k$. This defines yet another Lorentzian metric space $(X,d,\{p^k\}_{k\in\mathbb{N}})$. Moreover, the sequence $(X_n,d_n,\{p_n^k\}_{k\in\mathbb{N}})$  GH-converges to  $(X,d,\{p^k\}_{k\in\mathbb{N}})$ (of course if we had chosen the same generating sequences for all $X_n$ they would have been the same space and the $GH$-limit would have been $X=(0,+\infty)$ with that same sequence, which shows that the limit depends on the sequence).

Let us prove this claim.
    For any $n\in\mathbb{N}$ there is a function $\phi_n: X_n\to X$  defined by
    \[
    \phi_n(x)=x-n.
    \]
    It is easy to see that $\phi_n(p^k_n)=p^k$ for $k<2n-1$, and $\phi_n$ is clearly distance-preserving. Thus it allows us to define an $(m,\delta)$ $X_n-X$ quasi-correspondence for any $m < 2n-1$ and any $\delta>0$.

\subsection{Relation to bounded GH-convergence}
In the  metric space theory any convergent sequence of compact metric spaces can be made into a convergent sequence of pointed spaces (with the same limit) by choosing the preferred points appropriately. This result can not be directly generalised to our setting, because bounded spaces can not be sequenced. Instead, we have the following.
\begin{proposition} \label{nner}
    Let $(X_n,d_n)$ be a sequence of  bounded Lorentzian metric spaces GH-convergent to a Lorentzian metric space $(X,d)$. Suppose that for each $n$ the subset $I(X_n)$  is dense in  $X_n$, and similarly  $I(X)$ is dense in $X$,
    and let $\{p^k\}_{k\in\mathbb{N}}$ be a generating sequence  of $I(X)$. Then there are generating sequences $\{p^k_n\}_{k\in\mathbb{N}}$ of $I(X_n)$ (for each $n\in\mathbb{N}$) such that the sequence of sequenced Lorentzian metric spaces $(I(X_n),d_n,\{p_n^k\}_{k\in\mathbb{N}})$ GH-converges to  $(I(X),d,\{p^k\}_{k\in\mathbb{N}})$.
\end{proposition}
\begin{proof}
    From the assumptions it follows that we can choose for each $n\in\mathbb{N}$ a correspondence $R_n\subset X_n\times X$ in such a way that $\mathrm{dis}R_n<\delta_n$, where $\delta_n\rightarrow 0$ as $n\rightarrow \infty$. Then, by Corollary \ref{crl:bound-corr-quasi-corr}, we can choose for each $n$ a generating sequence  $\{p_n^k\}_{k\in\mathbb{N}}$ of $\mathrm{X}_n$ such that for every $l\le n$ there is an $(l,\delta_n)$ $I(X_n)-I(X)$ quasi-correspondence. Since $\delta_n$ converges to zero, for any given $m>0$  and $\epsilon>0$ we can find  for large enough $n$ (larger than $m$ and such that $\delta_n<\epsilon$) a $(m,\epsilon)$ $I(X_n)-I(X)$ quasi-correspondence.
\end{proof}

\subsection{GH-limit stability of (pre)length spaces}
\begin{theorem}\label{thm:prelength-prserved}
    Let $X_n$ be a sequence of sequenced Lorentzian (pre)length spaces GH-convergent to a Lorentzian metric space $X$. Then $X$ is a Lorentzian (pre)length space.
\end{theorem}
We give two proofs, one based on the results for bounded Lorentzian metric spaces, another direct.
\begin{proof}
    We provide details only for the prelength space case, the preservation of the maximization property being simple.
    Define the family of subsets of $X$:  $\mathscr{F}=\{I_{\epsilon}(p,q)|p,q\in X,\,\epsilon>0\}$. Clearly, $\mathscr{F}$ satisfies the requirement of Theorem \ref{thm:length-local}, so we can use an alternative formulation of the length property provided by it. Namely, we have to show that for every $n,m\in\mathbb{N}$, $\epsilon>0$ and every $r,x,y,s\in X$  such that $r\ll x\ll y \ll s$ and $r,s\in I_{\epsilon}(p^1,\ldots,p^m)$ there is an isocausal curve in $BI_{\epsilon}(p^1,\ldots,p^m)$ connecting $[x]_{I_{\epsilon}(p^1,\ldots,p^m)}$ with $[y]_{I_{\epsilon}(p^1,\ldots,p^m)}$. From now on we assume that such $r$, $x$, $y$, $s$, $\epsilon$ and $m$ are given. Set $\delta=\min(d(r,x),d(x,y),d(y,s))$

    Fix a sequence $\{\delta_m\}_{m\in\mathbb{N}}$ of real positive numbers converging to zero, and let $\{N_m\}_{m\in\mathbb{N}}$, $\{r_n\}_{n\in\mathbb{N}}$ and $\{R^n\}_{n\in\mathbb{N}}$ be the sequences provided by Lemma \ref{lem:weakGH-alt}. Define $T^{n}=R^{n}\cap \pi_2^{-1}(I_{\epsilon}(p^1,\ldots,p^m))$.   We have $T^{n}_2=I_{\epsilon}(p^1,\ldots,p^m)$, so $T^n$ is a $T^n_1-I_{\epsilon}(p^1,\ldots,p^m)$ correspondence, and $\mathrm{dis}T^n<\delta_n$. So, we have a sequence of bounded Lorentzian metric spaces $BT^n_1$ GH-convergent to $BI_{\epsilon}(p^1,\ldots,p^m)$. Choose $N\in\mathbb{N}$ such that $\delta_n<\delta$ for all $n>N$.

    For each $n>N$ we can choose $r_n,s_n,x_n,y_n\in T^n_1$ so that
    \[
    (r_n,r),(s_n,s),(x_n,x),(y_n,y)\in T^n_1.
    \]
    We have
    \[
    d(r_n,x_n)>\delta-\delta_n>0,\quad  d(x_n,y_n)>\delta-\delta_n>0, \quad d(y_n,s_n)>\delta-\delta_n>0,
    \]
    so $r_n\ll x_n\ll y_n \ll s_n$.
    By Theorem \ref{thm:length-local},
    for every $n\geq N$ there is an isocausal curve in $BT^n_1$ connecting $[x_n]$ with $[y_n]$. By construction in the proof of \cite[Theorem 5.18]{minguzzi22} it yields an isocausal curve connecting $x$ with $y$ in $BI_{\epsilon}(p^1,\ldots,p^m)$.
\end{proof}

\begin{proof}[Direct proof of Theorem \ref{thm:prelength-prserved}]
    Let us consider a sequence of sequenced Lorentzian metric spaces $(X_n,d_n,\{p_n^k\}_{k\in \mathbb{N}})$ GH-convergent to a sequenced Lorentzian metric space $(X,d,\{p^k\}_{k\in \mathbb{N}})$, and let us assume that all $X_n$ are prelength spaces.

    Let $\delta_m$ be a decreasing sequence of positive real numbers converging to zero, and let $\{N_m\}_{m\in\mathbb{N}}$, $\{r_n\}_{n\in\mathbb{N}}$, $\{R^n\}_{n\in\mathbb{N}}$ be  the sequences  provided by Lemma \ref{lem:weakGH-alt}.

    Let us find a convenient time function on $X$ to parametrize the curve  that we are going to construct,  and the approximations $\tau_n$ on the spaces $X_n$.
    Let $\mathscr{S}\subset X$ be a countable dense subset and let us define $\mathscr{S}^l=X\cap I_{\delta_l}(p^1,\ldots,p^l)$ for every $l\in\mathbb{N}$.  Let $\{s^i_l\}_{i \in \mathbb{N}}$ be a sequence whose image is $\mathscr{S}^l$ (possibly with repetitions)


    For any $n,l,i \in \mathbb{N}$ with $l\leq r_n$ let $s_{i,n}^l\in X^n$ be such that $(s_{i,n}^l,s_i^l )\in R^n$. It exists because  $s_i^l\in R^n_2$. For $l>r_n$ the points $s_{i,n}^l\in X_n$ can be chosen arbitrarily.
    We define the time function as in proof of Lemma \ref{lem:time-func}
    \[
        \tau(p):=\alpha \sum_{k=1}^{\infty}\sum_{j=1}^{\infty} \frac{1}{2^{j+k+1}} \left(\frac{d(s_j^k,p)}{1+d(s_j^k,p)}-\frac{d(p,s_j^k)}{1+d(p,s_j^k)}\right)+\beta.
    \]
    The constants $\alpha>0$ and $\beta$ are chosen so that $\tau(x)=0$, $\tau(y)=1$ and  the proof  that $\tau$ is a time function remains the same. Let us define
         \[
    \tau_{n}(p):=\alpha \sum_{k=1}^{r_n}\sum_{j=1}^{\infty} \frac{1}{2^{j+k+1}} \left(\frac{d(s_{j,n}^k,p)}{1+d(s_{j,n}^k,p)}-\frac{d(p,s_{j,n}^k)}{1+d(p,s_{j,n}^k)}\right)+\beta.
    \]
    By the same argument used for $\tau$, $\tau_n(p)\leq \tau_n(q)$ if $p\leq q$, $\tau_n$ is continuous, but $\tau_n$ does not have to be a time function, because the set $\{s_{i,n}^k|k\leq r_n\}$ does not have to distinguish points of $X_n$. Observe that for $(p',p)\in R^n$ we have\footnote{We use that  $\left|\frac{x}{1+x}-\frac{y}{1+y}\right|\leq |x-y|$ for $x,y>0$.}
    \begin{equation} \label{cnpp}
    |\tau_n(p')-\tau(p)|\leq \epsilon_n,
    \end{equation}
    where $\epsilon_n=\alpha(\delta_{r_n}+2^{-r_n})$, and $\lim_{n\rightarrow\infty}\epsilon_n=0$.

     Let us prove the prelength space property. Let $x,y\in X$ be such that $x\ll y$. We have to find an isocausal curve connecting $x$ with $y$. Let $m\in\mathbb{N}$ and $\delta>0$ be such that $x,y\in I_{\delta}(p^1,\ldots, p^m)$. By enlarging $m$ if necessary, we can assume that $\delta_l<\min(\delta/2, d(x,y))$ for all $l\geq m$. Thus, $x,y\in (R^n)_2$ for every $n\geq N_m$. This allows us to choose for every $n\ge N_m$ points $x_n,y_n\in X_n$ so that $(x_n,x)\in R^{n}$ and $(y_n,y)\in R^{n}$. Note that by the assumptions $r_n\ge m$ and
    \[
    d(x_n,y_n)>d(x,y)-\delta_{r_n}>0.
    \]
    So, $x_n$ and $y_n$ are connected by an isocausal curve  in $X_n$ for all $n\geq N_m$. Choose for each $n\geq N_m$ such a curve $\sigma_n:[0,1]\to X_n$.
    For similar reason, we have, using $\delta-\delta_{r_n}\ge \delta/2 \ge \delta_{r_n}$,
    \begin{equation}\label{eqn:thm:prelength-prserved-xnyn}
        x_n,y_n\in I_{\delta/2}(p_n^1,\ldots,p_n^m)\subset I_{\delta_{r_n}}(p_n^1,\ldots,p_n^m),\quad \forall n\geq N_m.
    \end{equation}
    We are ready to build an isocausal curve $\zeta:[0,1]\to X$ connecting $x$ with $y$. For any $n\geq N_m$ and $a \in (\epsilon_n,1-\epsilon_n)\cap \mathbb{Q}$ let $z_n(a)\in X^n$ be a point on the curve $\sigma_n$ such that
    \[
    \tau_n(z_n(a))=a.
    \]
    Note that such a point always exists because by Eq.\ (\ref{cnpp})
    \[
    \tau_n(\sigma_n(0))=\tau_n(x)\leq \epsilon_n,\qquad  \tau_n(\sigma_n(1))=\tau_n(y)\geq 1-\epsilon_n.
    \]
    For $a\in [0,\epsilon_n)\cap \mathbb{Q}$ set $z_n(a)=x_n$, and for $a\in (1-\epsilon_n,1]\cap \mathbb{Q}$ set $z_n(a)=y_n$.
    By (\ref{eqn:thm:prelength-prserved-xnyn}) we always have $z_n(a)\in I_{\delta_{r_n}}(p_n^1,\ldots,p_n^r)\subset R^n_1$, so we can choose $\zeta_n(a)\in X$ such that
    \[
    (z_n(a), \zeta_n(a))\in R^n.
    \]
    In particular, we choose $\zeta_n(0)=x$ and $\zeta_n(1)=y$.
    Note that
    \[
    z_n(a)\in  J_n(x_n,y_n)\subset I_{\delta}(p_n^1,\ldots,p_n^m),
    \]
    so for large enough $n$, $\zeta_n(a)\in  I(p^1,\ldots,p^m)$
    which is  relatively compact.
    So, for fixed $a$ we can always pass to a convergent sequence $\zeta_{n_k}(a)$. By the diagonal argument we can make this sequence convergent for  every $a\in [0,1] \in \mathbb{Q}$. Let $\zeta(a)$ denote the limits of these subsequences. In particular, $\zeta(0)=x$ and $\zeta(1)=y$.

    From
    \[
    |\tau(\zeta_n(a))-a|=|\tau(\zeta_n(a))-\tau_{n}(z_n(a))|\leq \epsilon_n,
    \]
    by taking the limit on the appropriately chosen subsequences, $\tau(\zeta(a))=a$ for $a\in [0,1]\cap \mathbb{Q}$.

    We intend to apply Lemma \ref{lem:curves-from-dense}. For that we have only to show that for any $t,t'\in [0,1]\cap \mathbb{Q}$ such that $t< t'$, $\zeta(t)\leq \zeta(t')$. Suppose that this is not so. Then there is $z\in X$ such that either $d(\zeta(t),z)<d(\zeta(t'),z)$ or $d(z,\zeta(t))>d(z,\zeta(t'))$. Define
    \[
    \Delta=\max\big(d(\zeta(t'),z)-d(\zeta(t),z), d(z,\zeta(t))-d(z,\zeta(t'))\big)>0.
    \]
    Let $m'\geq m$ be such that for every $l\geq m'$,   $z\in I_{\delta_l}(p^1,\ldots,p^l)$ and $\Delta> 2 \delta_l$.
    Let $n>N_{m'}$ so that $r_n\ge m'$ (and hence $z\in I_{\delta_{r_n}}(p^1,\ldots, p^{r_n})\subset R^n_{2}$) and such that
    \[
    t,t'\in \{0,1\}\cup (\epsilon_n,1-\epsilon_n)\cap \mathbb{Q}.
    \]
    But $z\in R^n_2$, so we can find $z'\in X_n$ such that $(z',z)\in R^n$.
    We have
    \[
    \tau_n(z_n(t'))-\tau_n(z_n(t))=t'-t>0,
    \]
    and
    \[
    \max(d_n(z_n(t'),z')-d_n(z_n(t),z'), d_n(z',z_n(t))-d_n(z',z_n(t')))\ge \Delta - 2\delta_{r_n}>0.
    \]
    Taking into account that by construction both $z_n(t')$ and $z_n(t)$ lie on an isocausal curve, this gives a contradiction. So, $\zeta(t)\leq \zeta(t')$ whenever $t<t'$, and $\zeta$ extends, by Lemma \ref{lem:curves-from-dense},  to an isocausal curve connecting $x$ to $y$.
\end{proof}


\section{Canonical quasi-uniform structure of Lorentzian metric spaces}

This section is devoted to the exploration of some global properties of a Lorentzian metric space.

\subsection{Kuratowski-like embedding}
In \cite{minguzzi22} the metrizability of bounded Lorentzian metric spaces was shown via a Kuratowski-like embedding.
We already know from Proposition \ref{prop:cg-Polish} that the countably-generated Lorentzian metric spaces are Polish, and  therefore completely metrizable. Still it may be interesting to find such an embedding in the unbounded case. As a side result, we shall show how to construct examples of metrics on countably generated Lorentzian metric spaces.

Similarly to \cite{minguzzi22}, for a Lorentzian metric space $X$ we define a map
\[
I: X\to C(X)\times C(X),
\]
\[
x\mapsto (d_x,d^x).
\]
Unlike the bounded case, in general $X$ can not be assumed to be compact, and thus $C(X)$ is not a Banach space. The natural topology for $C(X)$ is that of uniform convergence on compact subsets, i.e.\ the locally convex topology generated by the semi-norms
\[
||f||_{K}=\sup_{x\in K} |f(x)|
\]
parametrised by compact sets $K\subset X$. It makes $C(X)$ into a Fr\'echet space provided $X$ is $\sigma$-compact \cite{treves67}. In particular, if $X$ is countably generated, $C(X)$ is completely metrisable.
\begin{lemma}
    Let $(X,d)$ be a Lorentzian metric space. Then $I: X\to C(X)\times C(X)$ defined above is continuous.
\end{lemma}
\begin{proof}
    Fix $x\in X$, a compact set $K\subset X$ and $\epsilon>0$. We have to show that there is a neighborhood $W\subset X$ of $x$ such that whenever $y\in W$ we have $||d_x-d_y||_K<\epsilon$ and $||d^x-d^y||_K<\epsilon$.

    By continuity of $d$, for any $z\in K$ we can find neighborhoods $V_1^z$ and $V_2^z$ respectively of $x$ and $z$ such that for any $y\in V_1^z$ and $w\in V_2^z$ we have
    \[
    |d(x,z)-d(y,w)|<\frac{\epsilon}{2}.
    \]
    As $K$ is compact, we can choose $z_1,\ldots,z_n$ so that $\bigcup_{i=1}^n V_2^{z_i}\supset K$. Let $W_1=\bigcap_{i=1}^n V_1^{z_i}$. Take $y\in W_1$ and $z\in K$. For some $i$ we have $z\in V_2^{z_i}$, and thus
    \[
    |d(x,z)-d(y,z)|\leq |d(x,z)-d(x,z_i)|+ |d(x,z_i)-d(y,z)|<\epsilon.
    \]
    So, for every $y\in W_1$, $||d_x-d_y||_K<\epsilon$. Similarly, for another compact set $K'$ and constant $\epsilon'>0$ we can construct $W_2$ such that $||d^x-d^y||_{K'}<\epsilon'$ whenever $y\in W_2$, so by setting $W=W_1\cap W_2$ we get continuity of $I$.
\end{proof}
\begin{theorem} \label{cmqnn}
    Let $(X,d)$ be a Lorentzian metric space without  chronological boundary. Then $I$ is a homeomorphism onto its image.

    If, in addition to that, $X$ is countably generated, and $\{p^k\}_{k\in\mathbb{N}}$ is a generating sequence for $X$, then the function $\gamma: X\times X\to [0,+\infty)$ defined by
    \[
    \gamma (x,y):=\sum_{m=1}^{\infty}2^{-m} \frac{||d_x-d_y||_{X^m}}{1+||d_x-d_y||_{X^m}}+\sum_{m=1}^{\infty}2^{-m} \frac{||d^x-d^y||_{X^m}}{1+||d^x-d^y||_{X^m}}
    \]
    is a metric on $X$ inducing the Lorentzian metric space topology $\mathcal{T}$. Here $X^m=\overline{I_R(p^1,\ldots,p^m)}$ as in Section \ref{sec:GH}.
\end{theorem}
\begin{proof}
    Clearly, $I$ is injective, so we need only to show that it is an open map. Let $O\subset X$ be an open set and $x\in O$. We have to show that there is an open neighborhood $V$ of $I(x)$ such that $V\cap I(X)\subset I(O)$.

    By Proposition \ref{prop:lms-hdf-lc-unique}, there are some points $p_1,\ldots,p_n,q_1,\ldots,q_n\in X$ and numbers $a_1,\ldots,a_n,b_1,\ldots,b_n, c_1,\ldots,c_n,e_1,\ldots,e_n\in [-\infty,+\infty]$ such that
    \begin{equation}
    x\in W= \bigcap_{i=1}^n (d_{p_i})^{-1}((a_i,b_i))\cap (d^{q_i})^{-1}((c_i,e_i))\subset O.
    \label{eqn:proof:Kuratowski-basis-nbhd}
    \end{equation}
    Note that $a_i,c_i<+\infty$ and $b_i,e_i>-\infty$ because $W\ne \emptyset$.
    The case in which all $a_i,b_i,c_i,e_i$ are infinite is trivial as $W=X$ (which implies $O=X$ and so we can take $V=C(X)\times C(X)$), so it is safe to assume that at least some of them is finite.
    Define finite (and thus compact) set $K=\{p_1,\ldots,p_n,q_q,\ldots,q_n\}$ and
    \[
    \epsilon= \min_{i} \min\big(d(p_i,x)-a_i,b_i-d(p_i,x), d(x,q_i)-c_i,e_i-d(x,q_i)\big).
    \]
    Note that $0<\epsilon<\infty$ due to (\ref{eqn:proof:Kuratowski-basis-nbhd}), and assumption that $W\neq X$. Set
    \[
      V=\left\{(f,g)\in C(X)\times C(X)| \ \ ||f-d_x||_K<\epsilon, \ ||g-d^x||_K<\epsilon\right\}.
    \]
    Clearly, $V$ is an open neighborhood of $I(x)$. Moreover, if $I(y)\in V$, then $y\in W$, so $V\cap I(X)\subset I(W)\subset I(O)$, as desired.

    For the second statement (we do not insist in giving a detailed proof as it shall also follows from the independent proofs of Section \ref{cjor}) let us show that the topology of $C(X) \times C(X)$ is in fact generated by seminorms $||\cdot||_{X^m}$. Let $K$ be a compact set. Since $X$ has no chronological boundary, $K$ can be covered by chronological diamonds, and thus, by compactness of $K$, $K\subset I(A)$ for some finite set $A\subset X$. Then, taking $m$ large enough to ensure $A\subset I(p^1,\ldots,p^m)$, we have $K\subset X^m$. Therefore,
    \[
    ||f||_{K}\leq ||f||_{X^m}.
    \]

    Thus, $C(X)\times C(X)$ is a locally-convex topological vector space with topology generated by a countable set of seminorms. It is a standard result of functional analysis that such spaces are metrisable, and $\gamma$ defined in the statement is the $I$-pullback for the standard choice of the corresponding metric \cite[Chapter 10]{treves67}.
\end{proof}
%
%
\begin{remark}
    Although the metric $\gamma$ is induced by a complete metric on $C(X)\times C(X)$, it is not in general true that $\gamma$ is a complete metric of $X$, because the image of $I$ is not necessary closed. For example, consider a bounded Lorentzian metric space $X$ with a dense chronological interior and a spacelike boundary $i^0$. Consider a sequence of points $x_i\in I(X)$ converging to $i^0$. Then it is easy to see that $\{I(x_i)\}_{i\in \mathbb{N}}$ converges to $(0,0)$ in $C(X)\times C(X)$, while the sequence $x_i$ itself has no limit in $I(X)$.
\end{remark}

\subsection{Basics of quasi-uniform spaces}

In the following sections we determine a canonical quasi-uniformity associated to a Lorentzian metric space. To start with, let us recall the notion of quasi-uniformity.

 We refer to \cite{nachbin65,fletcher82}  for further details and motivation (the term {\em semi-uniformity} used by Nachbin is not standard nowadays).
\begin{definition}
    Let $X$ be a set. A \emph{quasi-uniformity} (or a \emph{quasi-uniform structure}) $\mathscr{Q}$ is a filter on the set $X\times X$ such that
    \begin{enumerate}
        \item Every element $U\in \mathscr{Q}$ contains the diagonal,
        \[
        \Delta=\{(x,x)| \ x\in X\}\subset U\]
        \item  For every $U\in \mathscr{Q}$ there is $V\in \mathscr{Q}$ such that $V\circ V\subset U$.
    \end{enumerate}
    If, in addition to that, $\mathscr{Q}$ satisfies
    \begin{enumerate}
        \item[3.] for every $U\in \mathscr{Q}$, $U^{-1}\in \mathscr{Q}$, where\footnote{We stick here to the standard notation in the theory of quasi-uniformities.} $U^{-1}:=\{(x,y)| \ (y,x)\in U\}$.
    \end{enumerate}
    we say that $\mathscr{Q}$ is an \emph{uniformity} (or a \emph{uniform structure}). A space $X$ endowed with a (quasi)uniformity is a \emph{(quasi)uniform space}.
\end{definition}
\begin{remark}\label{rmk:qu-gen}
    It is easy to see that if a collection $\mathcal{Q}$ of subset of $X\times X$ satisfies conditions 1-2 (respectively, conditions 1-3) of the definition above, then it generates a filter $\mathscr{Q}$ which is a quasi-uniformity (respectively, a uniformity) for $X$.
\end{remark}
If $\mathscr{Q}$ is a quasi-uniformity, then
\[
    \mathscr{U}=\mathscr{Q}^*:=\{U \cap V^{-1} | \ U,V\in \mathscr{Q} \}
\]
is a uniformity, and
\[
    G=\bigcap\mathscr{Q}:=\cap_{U\in\mathscr{Q}}\,U
\]
is a preorder (reflexive and transitive relation),
to which we refer respectively as \emph{the uniformity} and \emph{the preorder associated with} $\mathscr{Q}$.
So, any quasi-uniform space is a preordered uniform space.
The uniformity is Hausdorff if and only if the preorder is antisymmetric \cite{nachbin65} (hence an order).


Starting from a uniformity $\mathscr{U}$, we define the \emph{topology $\mathcal{T}({\mathscr{U}})$ associated with $\mathscr{U}$}, by setting the filter of neighborhoods of a point $x\in X$ as that generated by the sets of the form
\[
 U[x]:=\{y\in X| \ (x,y)\in \mathscr{U} \}.
\]
If $\mathscr{Q}$ is a quasi-uniformity, $\mathscr{U}$ is the associated uniformity
$\mathcal{T}(\mathscr{Q}^*)$
is the topology associated to $\mathscr{Q}$. So, any quasi-uniform space induces a topological preordered space $(X,\mathcal{T}(\mathscr{Q}^*),\bigcap\mathscr{Q})$, which turns out to be a closed preordered space as the preorder $G$ is closed in the product topology \cite[Proposition II.2.8]{nachbin65}. Every closed preordered space that arises in this way for some quasi-uniformity is said to be quasi-uniformizable.

To every metric space $(X,h)$ naturally corresponds a uniform structure, defined as a filter generated by all sets of the form
\[
\{(x,y)\in X\times X| \ \ d(x,y)\leq  t\},
\]
with $t$ running over $(0,+\infty)$. It is easy to see that the associated topology is the metric space topology.


The concept of uniformity allows one to extend the notions initially introduced for metric spaces to a larger class of topological spaces and avoid explicit choice of metric when it is not really needed. In particular, we will need the following.
\begin{definition}
    Let $(X,\mathscr{U})$, $(X',\mathscr{U}')$ be two uniform spaces. We say that a map $f:X\to X'$ is \emph{uniformly continuous} if for any $U'\in \mathscr{U}'$ there is $U\in \mathscr{U}$ such that
    \[
    (f\times f)(U)=\{(f(x),f(y))|\  \ (x,y)\in U\}\subset U'.
    \]
\end{definition}
Any uniformly continuous map is continuous with respect to the associated topologies. Moreover, if $(X,h)$ and $(X',h')$ are metric spaces, then a map $f:X\to Y$ is uniformly continuous in the usual metric geometry sense if and only if it is uniformly continuous map between the corresponding uniform spaces.

\begin{definition}
    Let $\{x_n\}_{n\in\mathbb{N}}$ be a sequence of point of a uniform space $(X,\mathscr{U})$. We say that $x_n$ is a \emph{Cauchy sequence} if for any $U\in \mathscr{U}$ there is $N\in\mathbb{N}$ such that
    \[
     (x_n,x_n')\in U \ \forall n,n'\geq N.
    \]
\end{definition}
Again, clearly, this definition coincides with the usual one if the uniform structure is induced by a metric.

We also need the following construction.
\begin{definition}\label{def:quasiu}
    Let $X$ be a set, and $\{(X_{\alpha},\mathscr{Q}_{\alpha},f_{\alpha}\}_{\alpha\in \mathscr{A}}$ be a family of quasi-uniform spaces $(X_{\alpha}, \mathscr{Q}_{\alpha})$ and maps $f_{\alpha}: X\to X_{\alpha}$. Then by \emph{initial quasi-uniform structure} we mean the filter generated by the sets of the form
    \[
      (f_{\alpha}\times f_{\alpha})^{-1}(U):=\{(x,y)\in X\times X| \ \ (f_{\alpha}(x),f_{\alpha}(y))\in U\},
    \]
    with all possible indices $\alpha\in \mathscr{A}$ and all possible $U\in \mathscr{Q}_{\alpha}$.
\end{definition}
It is easy to see (e.g.\ via Remark \ref{rmk:qu-gen}) that this construction indeed produces a quasi-uniformity. Moreover, we have the following.
\begin{lemma}\label{lem:init-qu}
    Let   $X$ be a set, and $\{(X_{\alpha},\mathscr{Q}_{\alpha},f_{\alpha})\}_{\alpha\in \mathscr{A}}$ be as in Definition \ref{def:quasiu}. Let $\mathscr{Q}$ be the initial quasi-uniformity of this family.
    Then:
    \begin{enumerate}
        \item The uniformity $\mathscr{Q}^*$ is the initial uniformity of the family $\{(X_{\alpha},\mathscr{Q}^*_{\alpha},$ $f_{\alpha})\}_{\alpha\in \mathscr{A}}$.
        \item The topology $\mathcal{T}(\mathscr{Q}^*)$ is the initial topology of the family
        $\{(X_{\alpha},$ $\mathcal{T}({\mathscr{Q}^*_\alpha}),$ $ f_{\alpha})\}_{\alpha\in \mathscr{A}}$.
    \end{enumerate}
\end{lemma}
To finish this exposition, let us give an elementary example which will be useful in a moment. On the real line $\mathbb{R}$ we introduce the canonical quasi-uniform structure $\mathscr{Q}_{\mathbb{R}}$ as the filter generated by the sets of the form
\[
\{(x,y)\in \mathbb{R}| \ \  -t< y-x\},
\]
for all possible $t> 0$. Evidently, $(x,x)\in U_t$ for all $x\in\mathbb{R}$ and $t>0$, and $U_{t/2}\circ U_{t/2}\subset U_t$. So, $\mathscr{Q}_{\mathbb{R}}$ is a quasi-uniformity by Remark \ref{rmk:qu-gen}. Note that the associated uniformity is generated as a filter by the sets of the form
\[
\{(x,y)\in \mathbb{R}|\ \ |y-x|< t\},
\]
and thus coincides with the usual metric uniformity. On the other hand, the order, associated with $\mathscr{Q}_{\mathbb{R}}$ is the usual order $\leq$ of $\mathbb{R}$. We also need the dual quasi-uniformity
\[
\mathscr{Q}_{\mathbb{R}}^{-1}=\{U^{-1}| \ \ U\in \mathscr{Q}_{\mathbb{R}}\}.
\]
Clearly, the uniformities and topologies associated with $\mathscr{Q}_{\mathbb{R}}$ and $\mathscr{Q}_{\mathbb{R}}^{-1}$ coincide. The order associated with $\mathscr{Q}_{\mathbb{R}}^{-1}$ is $\geq$.

\subsection{Lorentzian metric space as a quasi-uniform space}

It is known that in the smooth setting globally hyperbolic spacetimes are quasi-uniformizable \cite{minguzzi12d}. This result generalizes to Lorentzian metric spaces and actually we point out that there exists an associated canonical quasi-uniformity.

\begin{theorem}\label{thm:qu-lms}
    Let $(X,d)$ be a Lorentzian metric space without chronological boundary. Let $\mathscr{Q}$ be the initial quasi-uniformity of the family of maps
    \[
    d_p: X\to (\mathbb{R},  \mathscr{Q}_{\mathbb{R}}) \quad \textrm{and} \quad
    d^p: X\to (\mathbb{R},  \mathscr{Q}_{\mathbb{R}}^{-1}).
    \]
    The topology associated with $\mathscr{Q}$, $\mathcal{T}(\mathscr{Q}^*)$, coincides with the Lorentzian metric space topology. The order associated with $\mathscr{Q}$, $\bigcap \mathscr{Q}$,  is the extended causal relation $J$.
\end{theorem}
From now on we refer to $\mathscr{Q}$ constructed in the theorem above as the quasi-uniformity of the Lorentzian metric space $X$. All related notions, such as uniform convergence, are considered with respect to this quasi-uniformity unless  explicitly stated otherwise.
\begin{proof}
By Proposition \ref{prop:lms-hdf-lc-unique}
 the Lorentzian metric space topology is precisely the initial topology of functions $(d_p,d^p)$ with $p\in X$. So, by Lemma \ref{lem:init-qu}, it is the topology associated with $\mathscr{Q}$.

Now, let $G$ be the order associated with $\mathscr{Q}$. For a pair of points $x,y \in X$, $(x,y)\in G$ if and only if for any $p\in X$ and $t> 0$,
\[
-t<d_p(y)-d_p(x),
\]
and
\[
d^p(y)-d^p(x)<t.
\]
This holds if and only if $(x,y)\in J$, so $J=G$.
\end{proof}

We provide new proofs of Lemma \ref{lem:curves-from-dense} and Theorem \ref{thm:d-uni}.
\begin{lemma}\label{lem:curves-from-dense2}
     Let $(X,d)$ be a countably-generated Lorentzian metric space, and let $\tau: X\to \mathbb{R}$ be a time function. Let $a,b\in \mathbb{R}$  be such that $a<b$, and assume that $Q$ is a dense subset of $[a,b]$ such that $a,b\in Q$. Let $\zeta: Q\to X$ be such that
    \begin{itemize}
        \item $\tau(\zeta(t)))=t$, $\forall t\in Q,$
        \item  $\zeta(t)\leq \zeta(t')$ whenever $t,t'\in Q$, $a\leq t <t' \leq b$.
    \end{itemize}
    Then $\zeta$ extends to an isocausal curve $\overline{\zeta}:[a,b]\to X$ parametrised by $\tau$, $\tau(\overline{\zeta}(s))=s$. Moreover, if $\zeta$ satisfies the maximality condition (\ref{eqn:maxcurv}) for all $t,t',t''\in Q$, then $\overline{\zeta}$ is maximal.
\end{lemma}

\begin{proof}
   Let $p,q\in X$ be such that $p\ll \zeta(a)$ and $q\gg \zeta(b)$. Then for every $t\in Q$, $\zeta(t)\in J(\zeta(a),\zeta(b))\subset \overline{I(p,q)}$.

    Let us show that $\zeta$ is Cauchy-continuous, i.e. it maps a Cauchy sequence to a Cauchy sequence. Here the uniformity structure of $Q$ is the one associated with the metric on $Q$ induced from $[a,b]$, and for $X$ we use the quasi-uniformity defined in Theorem \ref{thm:qu-lms}.
    By construction of Theorem \ref{thm:qu-lms}, it is enough to show that for every $z\in X$ the functions $d^z\circ \zeta$ and $d_z\circ \zeta$ are Cauchy continuous. Assume that for some sequence $\{t_n\}$ of points in $Q$ and $z\in X$ the sequence $\{d_z(\zeta(t_n))\}$ is not Cauchy. Then there is $\epsilon>0$ and, for every $n\in\mathbb{N}$, integers $m_n,l_n\geq n$ such that
    \begin{equation}\label{eqn:lem:curves-from-dense:1}
        d_z(\zeta(t_{m_n}))-d_z(\zeta(t_{l_n}))>\epsilon.
    \end{equation}
    Note that this is possible only if $t_{m_n}>t_{l_n}$ for every $n\in\mathbb{N}$, because  $d_z\circ\zeta$ is non-decreasing. Thus,
    \begin{equation}\label{eqn:lem:curves-from-dense:2}
    \zeta(t_{m_n})\geq \zeta(t_{l_n}).
    \end{equation}

    Since these sequences
    take values  in the compact set $\overline{I(p,q)}$, by passing to subsequences we can assume that
    \[
    \lim_{n\to \infty}\zeta(t_{m_n})=x,\ \lim_{n\to \infty}\zeta(t_{l_n})=y.
    \]
    for some $x,y\in  \overline{I(p,q)}$. Note that this procedure does not spoil the condition $m_n,l_n\geq n$ for all $n\in\mathbb{N}$. By taking the limit of (\ref{eqn:lem:curves-from-dense:1}-\ref{eqn:lem:curves-from-dense:2}), we have $d_z(x)-d_z(y)\geq \epsilon$ and $x\geq y$. Thus, $x>y$, and $t:=\tau(x)-\tau(y)>0$. By cutting finitely many terms of the sequences $m_n$, $l_n$ we may always assume that $\tau(\zeta(t_{m_n}))-\tau(\zeta(t_{l_n}))>t/2$. In other words, for every $n\in\mathbb{N}$ there are integers $l_n,m_n\geq n$ such that $t_{m_n}-t_{l_n}\geq t/2$, i.e.\ the sequence $\{t_n\}$ can not be Cauchy. This concludes the proof of Cauchy continuity of $d_z\circ \zeta$ for every $z\in X$. Analogously, $d^z\circ \zeta$ is Cauchy continuous for all $z\in X$. So, $\zeta$ is Cauchy continuous.

     To finish the proof we use the uniqueness of the uniformity on the compact set $\overline{I(p,q)}$. In particular, the uniformity induced from Lorentzian metric space $X$ coincides with the uniformity associated with a complete continuous metric, which exists since $X$ is a Polish space (Proposition \ref{prop:cg-Polish}). So, the Cauchy continuous function $\zeta$ has a continuous extension\footnote{ Actually we could use directly the following fact. A uniformly continuous function from a dense subset of a uniform space into a complete uniform space can be extended (uniquely) into a uniformly continuous function on the whole space.}
 $\overline{\zeta}$. The claimed properties of the extension, including maximality of $\overline{\zeta}$ provided that $\zeta$ is maximal, easily follow by continuity.
\end{proof}

\begin{theorem}\label{thm:d-uni2}
    Let us consider a countably generated Lorentzian metric space $(X,d)$,
    let $K$ be a compact subset of $X$, let $\sigma_n: [a_n,b_n]\to K$ be a sequence of continuous curves, and let $\sigma: [a,b]\to K$ be yet another continuous curve. The sequence $\sigma_n$ uniformly converges to $\sigma$ iff for each $z\in X$ the sequences of functions  $d_z\circ\hat \sigma_n$ and $d^z\circ\hat \sigma_n$ converge uniformly to $d_z\circ\hat \sigma$ and $d^z\circ\hat \sigma$, respectively.
\end{theorem}
\begin{proof}
    This follows from the construction of the uniformity of a Lorentzian metric space.
\end{proof}

\subsection{The quasi-metric of a sequenced Lorentzian metric space} \label{cjor}

In the previous section we have introduced a quasi-uniformity  canonically associated to a Lorentzian metric space. Actually, there is a useful second option which we call {\em the fine quasi-uniformity} of $(X,d)$, denoting it $\mathscr{F}$. It is generated by sets of the following form
\[
\big\{(x,y)\vert \quad  \sup_K (d_y-d_x)<a, \ \  \sup_K (d^x-d^y)<b \big\}
\]
where $K$ is a compact subset of $X$ and $a,b\in (0,+\infty]$. It is easy to check that it is indeed a quasi-uniformity. It contains the quasi-uniformity of the previous section which is obtained by restricting $K$ to finite subsets.

We have still the validity of the following result
\begin{theorem}\label{thm:qu-lms2}
    Let $(X,d)$ be a Lorentzian metric space without chronological boundary. Let $\mathscr{F}$ be the fine quasi-uniformity of $(X,d)$.
    The topology associated with $\mathscr{F}$, $\mathcal{T}(\mathscr{F}^*)$, coincides with the Lorentzian metric space topology. The order associated with $\mathscr{F}$, $\bigcap \mathscr{F}$,  is the extended causal relation $J$.
\end{theorem}

\begin{proof}
Since $\mathscr{Q}\subset \mathscr{F}$ we have  $\bigcap \mathscr{F}
\subset \bigcap \mathscr{Q}=J$. But if $(x,y)\in J$ we have for every $z\in K$, $d_x(z)\ge d_y(z)$ thus for every $a>0$, $\sup_K (d_y-d_x)<a$ and similarly for the other inequality, which proves that $(x,y)$ belongs to every element of $\mathscr{F}$, and hence $J=\bigcap \mathscr{F}$.

As $\mathscr{Q}\subset \mathscr{F}$ the topology $\mathcal{T}(\mathscr{F}^*)$ is finer than $\mathcal{T}(\mathscr{Q}^*)$, thus we have only to prove the other direction. We need only to prove that the neighborhood of $x$, $W:=\{y\vert \  \sup_K  \vert d_y-d_x\vert<a\}$ contains a $\mathcal{T}(\mathscr{Q}^*)$ neighborhood  (the proof for the neighborhoods $\{y\vert \  \sup_K  \vert d^x-d^y\vert<b\}$ being analogous). This result uses the continuity of the function $d$. We know that  $\mathcal{T}(\mathscr{Q}^*)$ is the Lorentzian metric space topology thus for every $z\in K$ such that $\vert d_y(z)-d_x(z)\vert <a$, as $d$ is continuous in such topology we can  find $U^z,V^z\in \mathcal{T}(\mathscr{Q}^*)$, $(x,z)\in U^z\times V^z$ such that for $(y',z')\in U^z\times V^z$, the inequality  $\vert d_{y'}(z')-d_{x}(z')\vert <a$ still holds. Let $\{z_i\}$ be a finite family of points such that $\{V^{z_i}\}$ cover $K$ and set  $U=\cap_i U^{z_i}$, then for $(y,z)\in U\times K$, we have for some $i$,  $z\in V^{z_i}$, and hence $(y,z)\in U^{z_i}\times V^{z_i}$ which implies $\vert d_y(z)-d_x(z)\vert <a$, i.e.\ $x\in U\subset W$. As $U\in \mathcal{T}(\mathscr{Q}^*)$ we conclude that $\mathcal{T}(\mathscr{Q}^*)$ is finer than  $\mathcal{T}(\mathscr{F}^*)$.
\end{proof}

A {\em quasi-pseudo-metric} \cite{kelly63,patty67}  on a set $X$ is
a function $p:X\times X \to [0,+\infty)$ such that for $x,y,z\in X$
\begin{itemize}
\item[(i)] $p(x,x)= 0$,
\item[(ii)] $p(x,z)\le p(x,y)+p(y,z)$.
\end{itemize}

The quasi-pseudo-metric is called {\em pseudo-metric} if
$p(x,y)=p(y,x)$.  It is called a {\em (Albert's) quasi-metric} if $p(x,y)=p(y,x)=0 \Rightarrow x=y$.
If $p$ is a quasi-pseudo-metric then $q$, defined
by
\[q(x,y)=p(y,x),\] is a quasi-pseudo-metric called {\em conjugate} of
$p$. This structure, called {\em qua\-si-pse\-udo-me\-tric spa\-ce}, is
denoted $(X,p,q)$ and we might equivalently use the notation
$p^{-1}$ for $q$.

From a quasi-pseudo-metric space $(X,p,q)$ we can construct a
quasi-uni\-for\-mi\-ty $\mathcal{U}$, and  subsequently an associated topological
preordered space
\[
(X, \mathcal{T}(\mathcal{U}^{*}),\bigcap
\mathcal{U}).
\]
Indeed, Nachbin \cite{nachbin65} defines the quasi-uniformity $\mathcal{U}$ as the filter
generated by the countable base
\begin{equation} \label{vaz}
W_n=\{(x,y)\in X\times X: p(x,y)<1/n\}.
\end{equation}
thus the graph of the preorder is $G=\bigcap
\mathcal{U}=\{(x,y): p(x,y)=0\}$ and the topology
$\mathscr{T}(\mathcal{U}^{*})$ is that of the
pseudo-metric
$p+q$. In particular this topology is Hausdorff if and only if $p+q$
is a metric i.e.\ $p$ is a quasi-metric, which is the
case if and only if the preorder $G$ is an order.

\begin{example}

For $X=\mathbb{R}$, the canonical quasi-uniformity $\mathscr{Q}_{\mathbb{R}}$ is induced by the pseudo-metric $p(x,y)=\textrm{max}(x-y,0)$, while $\mathscr{Q}_{\mathbb{R}}^{-1}$ is induced by $p^{-1}$.
\end{example}

If a quasi-uniformity admits a quasi-pseudo-metric then it is quasi-pseudo-metrizable. Nachbin \cite[Thm.\ 8]{nachbin65} proves that they are precisely the quasi-uniformities that admit a countable base (for the filter). The quasi-pseudo-metric, in general, is not uniquely determined.

It is known that in the smooth setting globally hyperbolic spacetimes are quasi-metrizable \cite{minguzzi12d}. This result generalizes to Lorentzian metric spaces as a consequence of the following stronger theorem (we denote $\vert r \vert^+:=\max(r,0)$ and for a function $f$ and a  set $K$ on the domain of $f$, $\Vert f \Vert^+_K:=\sup_K \vert f \vert^+= \vert \sup_K f \vert^+$).
\begin{theorem} \label{vnqjp}
    Let $(X,d)$ be a countably generated Lorentzian metric space and let $\{p^k\}_{k\in\mathbb{N}}$  be generating sequence for $X$, then the function $p: X\times X\to [0,+\infty)$ defined by
    \[
    p (x,y):=\sum_{m=1}^{\infty}2^{-m} \frac{||d_y-d_x||^+_{X^m}}{1+||d_y-d_x||^+_{X^m}}+\sum_{m=1}^{\infty}2^{-m} \frac{||d^x-d^y||^+_{X^m}}{1+||d^x-d^y||^+_{X^m}}
    \]
    is a quasi-metric for the fine quasi-uniformity of $(X,d)$.
     In particular, $J=p^{-1}(0)$ and $\gamma$, the metric of Thm.\ \ref{cmqnn}, is Lipschitz equivalent to $p+p^{-1}$, in fact $\gamma\le p+p^{-1}\le 2 \gamma$. Here $X^m=\overline{I_R(p^1,\ldots,p^m)}$ as in Section \ref{sec:GH}.
\end{theorem}

The constructed quasi-metric is not uniquely determined as it depends on the chosen sequence. However, we have clearly a canonically associated quasi-metric to every sequenced Lorentzian metric space.

\begin{proof}
For a continuous function $f$ over a compact set $K$, $\Vert f\Vert_K=\textrm{max}(\Vert f\Vert^+_K, \Vert -f\Vert_K)$ thus
\[
\frac{||f||_{K}}{1+||f||_{K}}=\textrm{max}\Big(\frac{||f||^+_{K}}{1+||f||^+_{K}}, \frac{||-f||^+_{K}}{1+||-f||^+_{K}} \Big)\le  \frac{||f||^+_{K}}{1+||f||^+_{K}}+\frac{||-f||^+_{K}}{1+||-f||^+_{K}},
\]
\[
\frac{||f||^+_{K}}{1+||f||^+_{K}}+\frac{||-f||^+_{K}}{1+||-f||^+_{K}}\le 2 \frac{||f||_{K}}{1+||f||_{K}}.
\]
From these inequalities we get $\gamma\le p+p^{-1}\le 2 \gamma$.
If we can show that $p$ induces the canonical fine quasi-uniformity $\mathscr{F}$ of the Lorentzian metric space, then, by the general theory of quasi-pseudo-metrics, $p+p^{-1}$ induces the topology $\mathcal{T}(\mathscr{F}^*)$ which, as shown previously, is the topology of the Lorentzian metric space. Note that $p+p^{-1}$ is a metric because $J$ is antisymmetric, so $p$ is a quasi-metric.
Observe that $p$ satisfies the triangle inequality $p(x,z)\le p(x,y)+p(y,z)$ because, dropping the index $X^m$ for shortness,
\begin{align*}
    \frac{||d_z-d_x||^+}{1+||d_z-d_x||^+}&\leq     \frac{||d_z-d_y||^++||d_y-d_x||^+}{1+||d_z-d_y||^++||d_y-d_x||^+}\\
    &\!\!\!\!\!\!\!\!\!\!\!\!\!\!\!\!\!\le
    \frac{||d_z-d_y||^+}{1+||d_z-d_y||^++||d_y-d_x||^+}+ \frac{||d_y-d_x||^+}{1+||d_z-d_y||^++||d_y-d_x||^+} \\
    &\leq
     \frac{||d_y-d_x||^+}{1+||d_y-d_x||^+}+ \frac{||d_z-d_y||^+}{1+||d_z-d_y||^+}.
\end{align*}
In the first inequality we used $||d_z-d_x||^+\le ||d_z-d_y||^++||d_y-d_x||^+$ and the fact that the composition of the non-decreasing maps $x \mapsto x/(1+x)$ and $x\mapsto \vert x\vert^+$, is non-decreasing. The term $\frac{||d^x-d^z||^+}{1+||d^x-d^z||^+}$ from the sum defining $p$ is treated analogously. The identity $p(x,x)=0$ is clear.

It remains to show that $p$ is  induces the fine quasi-uniformity $\mathscr{F}$.

Let us prove that the quasi-uniformity induced by $p$ is finer than that by $\mathscr{F}$.
Let $K\subset X$ be a compact set, then there is $k\in \mathbb{N}$ such that $K\subset X^k$. For any $x,y\in X$ we have
    \[
    \frac{||d_y-d_x||^+_{X^k}}{1+||d_y-d_x||^+_{X^k}}\leq 2^{k}p(x,y).
    \]
thus the set $\{(x,y)\vert \ \sup_K (d_y-d_x)<a\}$ contains the  filter element $\{p<\epsilon\}$ for sufficiently small $\epsilon>0$.
Similarly, the set $\{(x,y)\vert \  \sup_K (d^x-d^y)<b\}$ contains a filter element induced by $p$.

Let us prove that the quasi-uniformity induced by $p$ is coarser than that by $\mathscr{F}$. Let $0<\epsilon <1$ and let $k$ be so large that $2^{-(k-1)}<\epsilon/2$. For $j\le k$, $||d_y-d_x||^+_{X^j} \le ||d_y-d_x||^+_{X^k}$ thus
\begin{align*}
p(x,y)&\le \sum_{m=1}^k 2^{-m} \frac{||d_y-d_x||^+_{X^m}}{1+||d_y-d_x||^+_{X^m}} + 2\sum_{m=k+1}^\infty \frac{1}{2^m} + \sum_{m=1}^k 2^{-m} \frac{||d^x-d^y||^+_{X^m}}{1+||d^x-d^y||^+_{X^m}} \\
&\le  \frac{||d_y-d_x||^+_{X^k}}{1+||d_y-d_x||^+_{X^k}}+\frac{||d^x-d^y||^+_{X^m}}{1+||d^x-d^y||^+_{X^m}} + \frac{\epsilon}{2}
\end{align*}
Thus the $\mathscr{F}$ element given by the intersection of $\sup_{X^m}(d_y-d_x)<\frac{\epsilon}{(4-\epsilon)}$ and  $\sup_{X^m}(d^x-d^y)<\frac{\epsilon}{(4-\epsilon)}$ is contained in the  set $\{p<\epsilon\}$.
\end{proof}

\section{Conclusions}
In a previous work we introduced an abstract notion of Lorentzian metric space restricting ourselves to the bounded case. The goal there was to show the feasibility of the approach, the general strategy being that of being guided by the stability under GH-convergence for the identification of the correct definition of the Lorentzian (pre-)length space concept. Having obtained such objective, we turned in this work to the removal of the boundedness condition. This has introduced some technicalities, particularly when it comes to study GH-convergence, as one has to attach a sequence to the space, unlike the metric theory where a point is sufficient.

In other directions, we confirmed the validity of some nice results, such as the GH-stability of the  (pre-)length property, the existence of time functions, the validity of the limit curve theorem, the upper semi-continuity of the length functional, the characterization of the length space property via the length functional.

Finally, in some directions  our investigation into the non-compact case also elucidated several aspects of our original construction. For instance, we clarified that our Lorentzian metric space definition can be reformulated in terms of very desirable and minimal properties such as the reverse triangle inequality, compactness of chronological diamonds, and distinction via the Lorentzian distance, all of them reminiscent of conditions entering the property of global hyperbolicity in the smooth setting. We also realized that our  Lorentzian metric space is canonically quasi-uniformizable, a fact that leads to considerable simplification, e.g.\ in dealing with the limit curve theorem. Sequenced Lorentzian metric spaces are even canonically quasi-metrizable (and hence the topology is canonically metrizable) which could suggest new ways to attack old problems related to the precompactness of families of Lorentzian metric spaces.

\section*{Acknowledgments}
Stefan Suhr is partially supported by the Deutsche Forschungsgemeinschaft
(DFG, German Research Foundation) Project-ID 281071066 SFB/TRR
191.

This study was funded by the European Union - NextGenerationEU, in the framework of the PRIN Project (title) {\em Contemporary perspectives on geometry and gravity} (code 2022JJ8KER - CUP B53D23009340006). The views and opinions expressed in this article are solely those of the authors and do not necessarily reflect those of the European Union, nor can the European Union be held responsible for them. \\
%

\end{document}